\documentclass{amsart}

\usepackage{amssymb,amsmath,amsthm}
\usepackage{mathrsfs}
\usepackage{latexsym}
\usepackage{prettyref}
\usepackage{newtxmath}
\usepackage{helvet}
\usepackage{courier}
\usepackage{graphicx}
\usepackage{multicol}
\usepackage{footmisc}
\usepackage{dsfont}
\usepackage{mathtools}
\usepackage{hyperref}
\usepackage{amsbsy}
\usepackage{amscd}
\usepackage{enumerate}
\usepackage{pdfsync}

\DeclareFontFamily{U}{mathx}{\hyphenchar\font45}
\DeclareFontShape{U}{mathx}{m}{n}{
      <5> <6> <7> <8> <9> <10>
      <10.95> <12> <14.4> <17.28> <20.74> <24.88>
      mathx10
      }{}
\DeclareSymbolFont{mathx}{U}{mathx}{m}{n}
\DeclareFontSubstitution{U}{mathx}{m}{n}
\DeclareMathAccent{\widecheck}{0}{mathx}{"71}
\DeclareMathAccent{\wideparen}{0}{mathx}{"75}

\theoremstyle{plain}
\newtheorem{theorem}{Theorem}[section] 
\newrefformat{theorem}{Theorem~\hyperref[#1]{\ref{#1}}}
\newtheorem{lemma}[theorem]{Lemma} 
\newrefformat{lemma}{Lemma~\hyperref[#1]{\ref{#1}}}
\newtheorem{proposition}[theorem]{Proposition} 
\newrefformat{proposition}{Proposition~\hyperref[#1]{\ref{#1}}}
\newtheorem{corollary}[theorem]{Corollary} 
\newrefformat{corollary}{Corollary~\hyperref[#1]{\ref{#1}}}
\newtheorem{remark}[theorem]{Remark} 
\newrefformat{remark}{Remark~\hyperref[#1]{\ref{#1}}}
\newtheorem{assumption}[theorem]{Assumption} 
\newrefformat{assumption}{Assumption~\hyperref[#1]{\ref{#1}}}
\newtheorem{definition}[theorem]{Definition} 
\newrefformat{definition}{Definition~\hyperref[#1]{\ref{#1}}}
\newrefformat{example}{Example~\hyperref[#1]{\ref{#1}}}

\begin{document}

\title{On the Atomic Decomposition of Coorbit Spaces with Non-Integrable Kernel}

\author[S. Dahlke]{Stephan Dahlke}
\address{FB12 Mathematik und Informatik, Philipps-Universit\"at, Marburg, Hans-Meerwein Stra{\ss}e, Lahnberge, 35032 Marburg, Germany}
\email{dahlke@mathematik.uni-marburg.de}

\author[F.~de~Mari]{Filippo~De~Mari}
\address{Dipartimento di Matematica, Universit\`a di Genova, Via Dodecaneso 35, Genova, Italy}
\email{demari@dima.unige.it}

\author[E.~De~Vito]{Ernesto~De~Vito}
\address{Dipartimento di Matematica, Universit\`a di Genova, Via Dodecaneso 35, Genova, Italy}
\email{devito@dima.unige.it}

\author[L.~Sawatzki]{Lukas~Sawatzki}
\address{FB12 Mathematik und Informatik, Philipps-Universit\"at, Marburg, Hans-Meerwein Stra{\ss}e, Lahnberge, 35032 Marburg, Germany}
\email{sawatzkl@mathematik.uni-marburg.de}

\author[G.~Steidl]{Gabriele~Steidl}
\address{Department of Mathematics, University of Kaiserslautern, Paul-Ehrlich-Str. 31, 67663 Kaiserslautern, Germany}
\email{steidl@mathematik.uni-kl.de}

\author[G.~Teschke]{Gerd~Teschke}
\address{Institute for Computational Mathematics in Science and Technology, Hochschule Neubrandenburg, University of Applied Sciences, Brodaer Str. 2, 17033 Neubrandenburg, Germany}
\email{teschke@hs-nb.de}

\author[F.~Voitglaender]{Felix~Voigtlaender}
\address{Katholische Universit\"at Eichst\"att-Ingolstadt, Ostenstra{\ss}e 26, 85072 Eichst\"att, Germany}
\email{felix@voigtlaender.xyz}


\begin{abstract}
This paper ist concerned with recent progress in the context of coorbit space theory.
Based on a square integrable group representation, the coorbit theory provides new families of associated smoothness spaces, 
where the smoothness of a function is measured by the decay of the associated voice transform.
Moreover, by discretizing the representation, atomic decompositions and Banach frames can be constructed.
Usually, the whole machinery works well if the associated reproducing kernel is {\em integrable}  with respect to a weighted Haar measure on the group.
In recent studies, it has turned out that to some extent coorbit spaces can still  be established if this condition  is violated.
In this paper, we clarify in which sense atomic decompositions and Banach frames for these generalized coorbit spaces can be obtained.
\end{abstract}



\maketitle

\section{Introduction}\label{intro}

This paper is concerned with specific problems arising in the context
of signal analysis. The overall goal in signal analysis is the efficient
extraction of the relevant information one is interested in.
For this, the signal---usually modeled as an element in a
suitable function space---has to be processed, denoised, compressed, etc.
The first step is always to decompose the signal into appropriate building
blocks. This is performed by an associated transform, such as the
wavelet transform, the Gabor transform or the shearlet transform, just
to name a few. Which transform to choose clearly depends on the type
of information one wants to extract from the signal. In recent years,
it has turned out that group theory---in particular representation theory---acts
as a common thread behind many transforms. 
Indeed, many transforms are related
with square-integrable representations of certain locally compact groups.
For instance, the wavelet transform is associated with the affine group whereas
the Gabor transform stems from the Weyl-Heisenberg group. We refer e.g.,
to \cite{fegr88,fegr89a} for details.

This connection with group theory
paves the way to the application of another very important concept, namely
coorbit theory. This theory has been developed by Feichtinger and
Gr\"ochenig already in the late 1980's, see \cite{fegr88,fegr89a,fegr89b,gro91}.
In recent years, coorbit theory has experienced a real renaissance.
Among other things, the connections to the various shearlet transforms
\cite{DaDeGrLa} and to the concept of decomposition spaces
\cite{voigt,FuehrVoigtlaenderCoorbitAsDecomposition} have been investigated.

Based on a square integrable group representation, by
means of coorbit space theory it is possible to construct canonical
smoothness spaces, the coorbit spaces, by collecting all functions for
which the associated voice transform has a certain decay. Moreover, by
discretizing the underlying representation, it is possible to obtain
atomic decompositions for the coorbit spaces.
Moreover, also Banach frames can be constructed.

The coorbit space theory is based on certain assumptions. In particular,
it is not enough that the representation is square-integrable, it must
also be {\em integrable}, i.e., the reproducing kernel must be contained
in a weighted $L_1$-space on the group. Unfortunately, this condition
is restrictive, and even in very simple settings such as for the case of
band-limited functions, it is not satisfied. Nevertheless, in
\cite{dadedelastte17}, it has been shown that there is a way out. Instead
of using a classical $L_1$-space as the space of generalized test
functions, one can work with the weaker concept of Fr\'echet spaces. Then,
more or less all the basic steps to establish the associated coorbit
spaces can be performed. We refer to Sect. \ref{sec:overview} for brief
discussion of this approach.

However, in \cite{dadedelastte17} one issue remained open, namely the
construction of atomic decompositions for the resulting coorbit spaces.
This is exactly the problem we are concerned with here.
As a surprise, it turns out that this part of the coorbit space theory does not
directly carry over to the Fr\'echet setting.
There are two essential differences:
First of all, a  synthesis map can be constructed, but only at
the price that the integrability parameters of the
discrete norms on the coefficient spaces and of the coorbit norms are different.
At first sight, this might look strange, but in the setting of non-integrable
kernels this is in a certain sense not too surprising.
Indeed, in the context of coorbit space
theory, sooner or later convolution estimates of Young type have to be employed,
which yield bounded mappings between $L_p$-spaces with \emph{different}
integrability exponents for domain and codomain if the
convolution kernel is not in $L_1$. Concerning the atomic decomposition  part, the situation
is even more involved. It turns out that for any element in the coorbit space a
suitable approximation by linear combinations of the atoms can be derived, but
at the price that the weighted sequence norms of the expansion coefficients
cannot be uniformly bounded by the coorbit norm.  These results will be stated
and proved in Sect. \ref{sec:discretization}, see in particular Theorem \ref{maintheorem}.

Looking at these results, the inclined reader might have the impression that the
authors  were simply
unable to prove sharper results, whereas such results might still be true, and
provable with a more refined analysis.
This might be true, but only partially.
Indeed, in Sect. \ref{sec:Obstructions} we prove an additional result which
shows that, under some very natural conditions, {\em uniform} bounds can only
be obtained if the kernel operator acts as a bounded operator on the weighted
$L_p$-spaces, that is, this additional assumption is necessary for obtaining uniform bounds.
These facts strongly indicate that with the
decomposition results stated in Sect. \ref{sec:discretization} we have almost reached the ceiling.   
However, there is still a little bit of flexibility which we can use to improve our results.  
Indeed, in Sect. \ref{sec:DiscretziationUnderAssumptions}
we prove that if there exists a second  kernel $W$ that satisfies additional smoothness assumptions 
and acts as the identity by left and right convolution on the reproducing kernel of the representation, then {\em uniform} bounds 
for both, the synthesis and the analysis part,  can be obtained. Fortunately, in one important practical application given by the Paley-Wiener spaces such a kernel can be found.

\medskip{}

This paper is organized as follows.  First of all, in Sect. \ref{sec:overview}, we recall the construction  of coorbit spaces based on non-integrable kernels. 
We keep the explanation as short as possible and refer to  \cite{dadedelastte17} for further details. 
Then, in  Sect. \ref{sec:discretization},  we provide first discretization results for the  associated coorbit spaces; 
the main result is Theorem \ref{maintheorem}. Then, in Sect. \ref{sec:Obstructions}, we  are concerned with `negative' results. 
Indeed, in Theorem \ref{thm:DiscretizationObstruction} we show that stable decompositions can only be obtained if the right convolution 
by the reproducing kernel is bounded on the underlying $L_p$-spaces.    
Finally, in Sect. \ref{sec:DiscretziationUnderAssumptions} 
we present satisfactory discretization results with the aid of an additional kernel $W$.
Indeed, in the Theorems \ref{thm:GoodKernelBanachFrameCoorbit} 
and    \ref {thm:GoodKernelAtomicDecompositionCoorbit}, respectively, we show that atomic decompositions and Banach frames with uniform bounds can be constructed, 
just as in the context of the classical coorbit theory.

\section{An Overview}\label{sec:overview}

Throughout this paper, $G$ denotes a fixed locally compact second countable
group with left \emph{Haar measure} $\beta$ and \emph{modular function} $\Delta$.
For a definition of these terms, we refer to \cite{fol95}.
We simply write $\int_G f(x)~dx$ instead of $\int_G f(x)~d\beta(x)$ and we
denote by $ L_0(G) $ the space of Borel-measurable functions.
Given $f\in L_0(G)$ the functions $\widecheck{f}$ and $\overline{f}$ are
\[
  \widecheck{f}(x)=f(x^{-1}), \qquad \overline{f}(x)=\overline{f(x)},
\]
and for all $x\in G$ the left and right regular
representations $\lambda$ and $\rho$ act on $f$ as
\begin{alignat*}{2}
  \lambda(x)f\,(y) & =f(x^{-1}y) \qquad &\text{a.e }y\in G,\\
\rho(x)f\,(y) & =f(yx) \qquad &\text{a.e }y\in G.
\end{alignat*}
Finally, the convolution $f\ast g$ between $f,g\in  L_0(G) $ is the function
\[
  f\ast g(x)
  = \int_G f(y) g(y^{-1} x) ~dy
  = \int_G f(y) \cdot (\lambda(x)\widecheck{g})(y) ~dy
  \qquad \text{a.e. } x \in G,
\]
provided that, for almost all $x \in G$, the function
$y \mapsto f(y) \cdot (\lambda(x) \widecheck{g})(y)$ is integrable.

Furthermore, given two functions $f, g \in L_0(G)$, with slight abuse of
notations, we write
\begin{align*}
  \langle{f},{g}\rangle_{L_2} = \int_G f(x)\overline{g(x)}\,dx,
\end{align*}
provided that the function $fg$ is integrable.

We fix a continuous weight $w : G \to (0, \infty)$ satisfying
\begin{subequations}
  \begin{align}
    w(xy) & \leq w(x)w(y), \label{eq:ControlWeightSubmultiplicative} \\
    w(x) & = w(x^{-1})    \label{eq:ControlWeightSymmetric}
  \end{align}
for all $x,y \in G$. As a consequence, it also holds that
\begin{equation}
  \inf_{x\in G} w(x) \geq 1 \label{eq:ControlWeightBoundedBelow}.
\end{equation}
\end{subequations}
The symmetry~\eqref{eq:ControlWeightSymmetric} can always be satisfied by
replacing $w$ with $w+\widecheck{w}$, where the latter weight is easily seen
to still satisfy the submultiplicativity condition
\eqref{eq:ControlWeightSubmultiplicative}.

For all $p\in[1,\infty)$ define the separable Banach space
\[
  L_{p,w}(G) = \left\{f\in  L_0(G)  ~\middle|~ \int_G \lvert{w(x)f(x)}\rvert^p ~dx<\infty\right\}
\]
with norm
\[
  \lVert f \rVert_{L_{{p,w}}}^p = \int_G \lvert{w(x)f(x)}\rvert^p ~dx,
\]
and the obvious modifications for $L_\infty(G)$, which
  however is not separable. When $w \equiv 1$ we simply write $L_p(G)$.

With terminology as in \cite{dadedelastte17} we choose,
as a {\em target space} for the coorbit space theory, the following space
\[
  \mathcal{T}_w =\bigcap_{1<p<\infty} L_{p,w}(G).
\]
We recall some basic properties of $\mathcal{T}_w$; for proofs we refer to Theorem 4.3 of \cite{dadedelastte17}, which is based on results in \cite{damuwe70}.
We endow $\mathcal{T}_w$ with the
(unique) topology such that a sequence $(f_n)_{\in\mathbb{N}}$ in $\mathcal{T}_w$
converges to $0$ if and only if $\lim_{n\to+\infty} \lVert f_n \rVert_{L_{p,w}}=0$
for all $1<p<\infty$. With this topology, $\mathcal{T}_w$ becomes a reflexive
Frech\'et space. The (anti)-linear dual space of
$\mathcal{T}_w$ can be identified with
\[
  \mathcal{U}_w = {\operatorname{span}}\bigcup_{1<q<\infty} L_{q,{w^{-1}}}(G)
\]
under the pairing
\begin{equation}
  \int_G \Phi(x)\overline{f(x)}~dx=\langle{\Phi},{f}\rangle_w,\qquad
  \Phi \in \mathcal{U}_w,\,f\in\mathcal{T}_w.\label{eq:67}
\end{equation}

\begin{remark}
 The space $\mathcal{U}_w$ is endowed with one of the following equivalent
 topologies, both compatible with the pairing~\eqref{eq:67}.
 \begin{enumerate}[i)]
   \item The finest topology making the inclusions
         $L_{q,{w^{-1}}}(G)\hookrightarrow \mathcal{U}_w$ continuous for all
         $1<q<\infty$.

   \item The topology induced by the family of semi-norms
         $\left(\lVert{\cdot}\rVert_{p,r}\right)_{1 < p < r < \infty}$, where
         \[
           \lVert{\Phi}\rVert_{p,r}
           = \sup
               \left\{
                     \lvert{\langle{\Phi},{f}\rangle_w}\rvert
                     ~\middle|~
                     f \in \mathcal{T}_w
                     \text{ and }
                     \max\left\{\lVert{f}\rVert_{L_{p,w}}, \lVert{f}\rVert_{L_{r,w}} \right\} \leq 1
                   \right\},
         \]
         for $\Phi \in \mathcal{U}_w$.
 \end{enumerate}
\end{remark}

The representation $\lambda$ leaves invariant both $\mathcal{T}_w$ and $\mathcal{U}_w$, it
acts continuously on $\mathcal{T}_w$, and the contragradient
representation $\,^t\!{\lambda}$ of $\lambda_{|\mathcal{T}_w}$, given by
\[ \langle{\,^t\!{\lambda}_g \Phi},{f}\rangle_w =\langle{ \Phi},{\lambda_{g^{-1}} f}\rangle_w \quad \text{for} \quad \Phi \in \mathcal{U}_w
      \quad \text{and} \quad f \in \mathcal{T}_w,\] is simply $\,^t\!{\lambda} = \lambda_{|\mathcal{U}_w}$.

Take $g \in \mathcal{T}_w$ with $\widecheck{g} \in \mathcal{T}_w$. For all $f \in \mathcal{T}_w$
the convolution $f \ast g$ is in $\mathcal{T}_w$ and the map
\[
  f\mapsto f\ast g
\]
is continuous from $\mathcal{T}_w$ into $\mathcal{T}_w$.
Furthermore, for all $\Phi \in \mathcal{U}_w$ the convolution $\Phi \ast g$ is in
$\mathcal{U}_w$ and the map
\[
  \Phi \mapsto \Phi \ast g
\]
is continuous from $\mathcal{U}_w$ into $\mathcal{U}_w$.

Take now a (strongly continuous) unitary representation $\pi$ of $G$ acting
on a separable complex Hilbert space $\mathcal{H}$ with scalar product
$\langle{\cdot},{\cdot}\rangle_\mathcal{H}$ linear in the first entry. We assume that $\pi$ is
reproducing, namely there exists a vector $u \in \mathcal{H}$ such that the
corresponding voice transform
\[
  Vv(x) = \langle{v},{\pi(x)u}\rangle_\mathcal{H},
  \qquad
  v \in \mathcal{H},\, x \in G,
\]
is an isometry from $\mathcal{H}$ into $L_2(G)$. We observe that this implies that
$V$ is injective, whence
$\operatorname{span} \left\{\pi(x)u\right\}_{x \in G}$ is dense in $\mathcal{H}$.

We denote by $K$ the
reproducing kernel
\begin{equation}
  K(x)
  =Vu(x)
  =\langle{u},{\pi(x)u}\rangle_\mathcal{H},
  \qquad
  x \in G, v \in \mathcal{H},
\label{eq:KernelDefinition}
\end{equation}
which is a bounded continuous function and enjoys the following basic properties
\begin{subequations}
  \begin{align}
    & \overline{K} = \widecheck{K}, \label{eq:32a} \\
    & \sum_{i,j = 1}^n c_i \overline{c_j} K(x_i^{-1} x_j) \geq 0,
      \qquad
      c_1, \ldots, c_n \in \mathbb{C},\
      x_1,\ldots,x_n\in G, \label{eq:32b} \\
    & K \ast K = K \in L_2(G). \label{eq:32c}
  \end{align}
\end{subequations}
In general, $\pi$ is not assumed to be irreducible, but the
reproducing assumption implies that $u$ is a cyclic
vector. 
Properties~\eqref{eq:32a} and~\eqref{eq:32b} uniquely define
the representation $\pi$ up to a unitary equivalence, see
Theorem~3.20 and Proposition~3.35 of
\cite{fol95}. Equation~\eqref{eq:32c} states that $\pi$ is equivalent
to the sub-representation of the left-regular representation (on $L_2(G)$)
having $K$ as a cyclic vector. 
Conversely, if a bounded continuous function
$K$ satisfies~\eqref{eq:32a},~\eqref{eq:32b} and~\eqref{eq:32c}, then
there exists a unique (up to a unitary equivalence) reproducing
representation $\pi$ whose reproducing kernel is $K$.

For the remainder of the paper, we will always impose the following
basic assumption:
\begin{assumption}\label{assume:KernelAlmostIntegrable}
  We assume $K \in \mathcal{T}_w$, {\em i.e},
  \begin{equation}\label{eq:KernelAlmostIntegrable}
    K \in L_{p,w}(G)
    \text{ for all } 1 < p < \infty.
  \end{equation}
\end{assumption}

We add some remarks.

\begin{remark}
\begin{enumerate}[i)]
  \item Since $w(x)\geq 1$, Assumption~\eqref{eq:KernelAlmostIntegrable} implies
        that $K \in L_p(G)$ for all $p > 1$. If $\pi$ is irreducible, this last
        fact gives that $V$ is an isometry up to a constant, so that $\pi$ is
        always a reproducing representation.
        If $\pi$ is reducible, condition~\eqref{eq:KernelAlmostIntegrable} is not
        sufficient to ensure that $\pi$ is reproducing; however if \linebreak
        $K \ast K = K$, then $\pi$ is always reproducing.

  \item If $w^{-1}$ belongs to $L_q(G)$ for some $1 < q < \infty$, then
        H\"older's inequality shows $K \in L_1(G)$, but in general
        $K \notin L_{1,w}(G)$.
        However in many interesting examples $w$ is independent of one or more
        variables, so that $w^{-1} \not\in L_q(G)$ for all \linebreak $1 < q < \infty$.
\end{enumerate}
\end{remark}

We now define the \emph{test space} $\mathcal{S}_w$ as
\begin{equation}
  \mathcal{S}_w = \left\{
                v\in\mathcal{H}
                ~\middle|~
                Vv \in L_{p,w}(G) \text{ for all } 1 < p < \infty
              \right\},
  \label{eq:43}
\end{equation}
which becomes a locally convex topological vector space under the
family of semi-norms
\begin{equation}
  \lVert{v}\rVert_{p,\mathcal{S}_w} = \lVert{Vv}\rVert_{L_{p,w}}.
  \label{eq:29}
\end{equation}
We recall the main properties of $\mathcal{S}_w$.

\begin{theorem}[Theorem 4.4 of \cite{dadedelastte17}]\label{intersections}
Under Assumption~\eqref{eq:KernelAlmostIntegrable}, the following hold:
  \begin{enumerate}[i)]
    \item the space $\mathcal{S}_w$ is a reflexive Fr\'echet space,
          continuously and densely embedded in $\mathcal{H}$;

    \item\label{2.1b} the representation $\pi$ leaves $\mathcal{S}_w$ invariant and its
         restriction to $\mathcal{S}_w$ is a continuous representation;

    \item the space $\mathcal{H}$ is continuously and densely embedded into the
          (anti)-linear dual $\mathcal{S}'_w$, where both spaces are endowed with the
          weak topology;

    \item the restriction of the voice transform $V:\mathcal{S}_w\to \mathcal{T}_w$
          is a topological  isomorphism from $\mathcal{S}_w$ onto the
          closed subspace $\mathcal{M}^{\mathcal{T}_w}$ of $\mathcal{T}_w$, given by
          \[
            \mathcal M^{\mathcal{T}_w} = \left\{f\in \mathcal{T}_w ~\middle|~ f \ast K = f\right\},
          \]
          and it intertwines $\pi$ and $\lambda$;

    \item\label{bar} for every $f\in \mathcal{T}_w$, there exists a unique element
          $\pi(f)u\in\mathcal{S}_w$ such that
          \[
            \langle{\pi(f)u},{v}\rangle_\mathcal{H}
            = \int_G f(x) \langle{\pi(x)u},{v}\rangle_\mathcal{H}~dx
            = \int_G f(x) \overline{Vv(x)}~dx,
            \qquad v\in\mathcal{H}.
          \]
          Furthermore, it holds that
          \[
            V\pi(f)u = f \ast K,
          \]
           and the map
          \[
            \mathcal{T}_w \ni f \mapsto \pi(f)u \in \mathcal{S}_w
          \]
          is continuous and its restriction to $\mathcal M^{\mathcal{T}_w}$
          is the inverse of $V$.
 \end{enumerate}
\end{theorem}

Here and in the following the notation $\pi(f)u$ is motivated by the
following fact.

\begin{remark}
  In the framework of abstract  harmonic analysis, any function \linebreak $f\in L_1(G)$
  defines a bounded operator $\pi(f)$ on $\mathcal{H}$, which is weakly given by
  \[
    \langle{\pi(f) v},{v'}\rangle_\mathcal{H} = \int_G f(x) \langle{\pi(x)v},{v'}\rangle_\mathcal{H} ~dx,
    \qquad v,v'\in\mathcal{H},
  \]
  see for example Sect. 3.2 of \cite{fol95}. However, if $f\not\in L_1(G)$,
  then in general $\pi(f)v$ is well defined only if $v=u$, where
  $u$ is an admissible vector for the representation $\pi$.

\end{remark}

Recalling that the (anti-)dual of $\mathcal{T}_w$ is $\mathcal{U}_w$ under the
  pairing~\eqref{eq:67}, we denote by $\,^t\!{V}$ the contragradient  map
  $\,^t\!{V}: \mathcal{U}_w\to \mathcal{S}'_w$ given by
\[
\langle{\,^t\!{V}\Phi },{v}\rangle_{\mathcal{S}_w}= \langle{\Phi},{Vv}\rangle_{w},\qquad
\Phi\in\mathcal{U}_w,\, v\in\mathcal{S}_w.
\]

As usual, we extend the voice transform from $\mathcal{H}$ to the (anti-)dual $\mathcal{S}'_w$
of $\mathcal{S}_w$, where $\mathcal{S}'_w$ plays the role of the space of distributions.
For all $T \in \mathcal{S}_w'$ we set
\begin{equation}
  V_e T(x) = \langle{T},{\pi(x)u}\rangle_{\mathcal{S}_w},
  \qquad x\in G,
\label{EVT2}
\end{equation}
which is a continuous function on $G$ by item~\ref{2.1b}) of the previous
theorem and $\langle{\cdot},{\cdot}\rangle_{\mathcal{S}_w}$ denotes the pairing between $\mathcal{S}_w$
and $\mathcal{S}'_w$, whereas $\langle{\cdot},{\cdot}\rangle_w$ is the pairing between $\mathcal{T}_w$ and
$\mathcal{U}_w$.

We summarize the main properties of the extended voice transform in
the following theorem.

\begin{theorem}[Theorem 4.4 of \cite{dadedelastte17}]\label{intersections-1}
  Under assumption~\eqref{eq:KernelAlmostIntegrable}, the following hold:
  \begin{enumerate}[i)]
    \item for every $\Phi \in \mathcal{U}_w$ there exists a unique element
          $\pi(\Phi)u\in\mathcal{S}_w'$ such that
          \[
            \langle{\pi(\Phi)u},{v}\rangle_{\mathcal{S}_w}
            = \int_G \Phi(x) \langle{\pi(x)u},{v}\rangle_{\mathcal{H}}\,dx
            = \int_G \Phi(x) \overline{Vv(x)}\,dx,
            \qquad v \in \mathcal{S}_w.
          \]
          Furthermore, it holds that
          \[
            V_e \pi(\Phi)u = \Phi \ast K;
          \]

  \item\label{2.2b} for all $T\in\mathcal{S}_w'$ the voice transform  $V_eT$ is in
        $\mathcal{U}_w$ and satisfies
        \begin{align}
          & V_e T = V_e T \ast K, \label{eq:ExtendedVoiceTransformReproducing} \\
          & \langle{T},{v}\rangle_{\mathcal{S}_w} = \langle{V_e T},{V v}\rangle_w,
          \qquad v \in \mathcal{S}_w;
          \label{eq:ExtendedVoiceTransformDuality}
        \end{align}

  \item the extended voice transform $V_e$ is injective, continuous
        from $\mathcal{S}_w'$ into $\mathcal{U}_w$ (when both spaces are endowed with the
        strong topology), its range is the closed subspace
        \begin{equation}
          \mathcal M^{\mathcal{U}_w}
          = \left\{\Phi \in \mathcal{U}_w ~\middle|~ \Phi \ast K = \Phi\right\}
          = \operatorname{span}
              \bigcup_{p \in (1,\infty)}
                \mathcal{M}^{L_{p,w}(G)}
          \subset L_{\infty,{w^{-1}}}(G)
          \label{eq:38}
        \end{equation}
      and it intertwines the contragradient representation of $\pi_{|\mathcal{S}_w}$ and
      $\lambda_{|\mathcal{U}_w}$;

  \item\label{sotto} the map
        \[
          \mathcal M^{\mathcal{U}_w} \ni \Phi \mapsto \pi(\Phi) u \in \mathcal{S}_w'
        \]
        is the left inverse of $V_e$ and coincides with the restriction of the map
        $\,^t\!{\,V}$ to $\mathcal M^{\mathcal{U}_w}$, namely
        \begin{equation}
          V_e( \,^t\!{\,V}\Phi)
          = V_e\pi(\Phi)u
          = \Phi,
          \qquad \Phi \in \mathcal M^{\mathcal{U}_w};
          \label{eq:39}
        \end{equation}

  \item regarding $\mathcal{S}_w\hookrightarrow\mathcal{H}\hookrightarrow\mathcal{S}'_w$, it holds
        \[
          \mathcal{S}_w
          = \left\{T \in \mathcal{S}_w' ~\middle|~ V_e T \in \mathcal{T}_w\right\}
          = \left\{\pi(f)u ~\middle|~ f\in \mathcal{M}^{\mathcal{T}_w} \right\}.
        \]
  \end{enumerate}
\end{theorem}

Item~\ref{2.2b}) of the previous theorem states that
the voice transform of any distribution $T\in\mathcal{S}_w'$ satisfies the
reproducing formula~\eqref{eq:ExtendedVoiceTransformReproducing} and uniquely
defines the distribution $T$ by means of the reconstruction
formula~\eqref{eq:ExtendedVoiceTransformDuality}, i.e.
\[
  T = \int_G \langle{T},{\pi(x)u}\rangle_{\mathcal{S}_w} \pi(x)u~dx,
\]
where the integral is a Dunford-Pettis integral with respect to the duality
between $\mathcal{S}_w$ and $\mathcal{S}_w'$, see, for example,
Appendix~3 of \cite{fol95}.

We now fix an exponent $r \in [1,\infty)$,
and a $w$-moderate weight $m$,
i.e. a continuous function $m : G \to (0, \infty)$ such that
\begin{equation}
  m(xy) \leq w(x) \cdot m(y)
  \quad \text{and} \quad
  m(xy) \leq m(x) \cdot w(y)
  \quad \text{for all} ~ x,y \in G \, .
  \label{eq:ModerateWeight}
\end{equation}
%


\begin{remark}\label{rem:alt_w-moderate}
The definition \eqref{eq:ModerateWeight} of a $w$-moderate weight $m$ is equivalent to the condition
\begin{align*}
m(xyz) \leq w(x)\cdot m(y)\cdot w(z) \quad \text{for all}~x,y,z\in G
\end{align*}
up to the constant $w(e)$.
\end{remark}

The result of the following lemma is used multiple times in this paper.

\begin{lemma}\label{lem:1/m_w-moderate}
If $m$ is a $w$-moderate weight on $G$, then so is $m^{-1}$.
\end{lemma}

\begin{proof}
To prove the estimates in \eqref{eq:ModerateWeight} for $m^{-1}$ we fix $x,y\in G$, then by the $w$-moderateness of $m$ it holds
\begin{align*}
m(y) = m(x^{-1}xy) \leq w(x^{-1}) \cdot m(xy) = w(x) \cdot m(xy),
\end{align*}
which implies $m(xy)^{-1}\leq w(x) \cdot m(y)^{-1}$. Similarly we observe that
\begin{align*}
m(x) = m(xyy^{-1}) \leq m(xy) \cdot w(y^{-1}) = m(xy) \cdot w(y),
\end{align*}
which in turn implies $m(xy)^{-1}\leq m(x)^{-1} \cdot w(y)$.
\end{proof}

With terminology as in \cite{dadedelastte17}, we
choose as a {\em model space} for the coorbit space theory, the
Banach space $Y = L_{r,m}(G)$ with $r\in (1,\infty)$.
The corresponding coorbit space is defined as
\begin{equation}
 \operatorname{Co}(Y) = \left\{T\in \mathcal{S}'_w ~\middle|~ V_e T \in Y \right\}
 \label{eq:21}
\end{equation}
endowed with the norm
\begin{equation}
 \lVert{T}\rVert_{\operatorname{Co}(Y)} = \lVert{V_eT}\rVert_Y.
 \label{eq:23}
\end{equation}
We summarize the main properties of $\operatorname{Co}(Y)$ in the following proposition.

\begin{proposition}\label{prop6}
The space $\operatorname{Co}(Y)$ is a Banach space invariant under the action of the
contragradient representation of $\pi_{|\mathcal{S}_w}$.
The extended voice transform is an isometry from $\operatorname{Co}(Y)$ onto the
$\lambda$-invariant closed subspace
\[
  \mathcal M^Y = \left\{F\in Y ~\middle|~ F \ast  K =F \right\} \subset \mathcal{U}_w \, ,
\]
and we have
\[
  \operatorname{Co}(Y) = \left\{\pi(F)u ~\middle|~ F \in \mathcal{M}^Y \right\}.
\]
Furthermore
\begin{align}
 V_e \pi(F)u & = F,
 \qquad F \in \mathcal{M}^Y,
 \label{eq:65} \\
 \pi(V_eT)u & = T,
 \qquad T \in \operatorname{Co}(Y).
 \label{eq:66}
\end{align}
\end{proposition}

\begin{proof}
The proof is essentially an application of Theorem 3.5 in \cite{dadedelastte17}.
    We first note that convergence with respect to
    $\|\cdot\|_Y = \|\cdot\|_{L_{r,m}}$ implies convergence in measure.
    Furthermore, since $m$ is $w$-moderate, it is not hard to see that
    $Y = L_{r,m}(G)$ is $\lambda$-invariant, and that the restriction of
    $\lambda$ to $Y$ is a continuous representation of $G$.
    Therefore, we only need to prove that Assumption 5 and Assumption 6
    in \cite{dadedelastte17} are satisfied.

We first show that $Y\subset \mathcal{U}_w$.
By~\eqref{eq:ModerateWeight} and~\eqref{eq:ControlWeightSymmetric} we get for any $x \in G$ that
\begin{align}\label{eq:weight_e_inequality}
  m(e) = m(x x^{-1}) \leq m(x) \cdot w(x^{-1}) = m(x) \cdot w(x) \, ,
\end{align}
and hence $[w(x)]^{-1} \cdot m(e) \leq m(x)$,
whence $Y = L_{r,m}(G) \hookrightarrow L_{r,w^{-1}}(G) \subset \mathcal{U}_w$
since $r > 1$.

Since $\mathcal{U}_w = \mathcal{T}_w'$ under the pairing~\eqref{eq:67},
for all $F \in Y$ and $f \in \mathcal{T}_w$ it holds that $F f \in L_1(G)$.
In particular, by assumption~\eqref{eq:KernelAlmostIntegrable},
$FK \in L_1(G)$ for all $F \in Y$ and, by construction, $F \,Vv \in L_1(G)$ for
all $v \in \mathcal{S}_w$ and $F \in \mathcal M^Y$,
so that Assumption 5 and Assumption 6 in \cite{dadedelastte17} hold
true.
%
\end{proof}



\section{Discretization}\label{sec:discretization}

The aim of this section is to establish certain atomic decompositions
for the coorbit spaces described in Sect. \ref{sec:overview}.
In particular, we recall that
\begin{equation} \label{basicsetting}
  {\mathcal T}_{w} = \bigcap_{p \in (1,\infty)} L_{p,w}(G),
  \quad\quad
  {\mathcal T}'_{w}
  = {\mathcal U}_{w}
  = \mbox{span} \bigcup_{q \in (1,\infty)} L_{q,w^{-1}}(G) 
\end{equation}
and for some $1<r<\infty$,
\begin{equation} \label{Lmspaces}
  Y = L_{r,m}(G).
\end{equation}
Proposition~\ref{prop6} shows that  the correspondence principle holds,
i.e., the extended voice transform $V_e$ is an
isometry from  the associated  coorbit space
\begin{equation}\label{defcoorbit}
  \mbox{Co}( L_{r,m})
  := \left\{ T \in {\mathcal S}'_{w} ~\middle|~ V_e(T) \in L_{r,m}(G) \right\}
\end{equation}
onto the corresponding reproducing kernel Banach space
\begin{equation}\label{eq:ReproducingKernelLp}
  \mathcal{M}_{r,m} = \mathcal{M}^{L_{r,m}(G)} = \left\{ f \in L_{r,m}(G) ~\middle|~ f \ast K = f \right\}.
\end{equation}

\begin{remark}
Assumption~\eqref{eq:KernelAlmostIntegrable} on the kernel $K$ and the
fact  that  $m$ is $w$-moderate imply that for all $f\in L_{r,m}(G)$ the
convolution $f\ast K$ is well-defined; see Proposition \ref{prop:Young}.
\end{remark}

In this setting we can characterize the anti-dual $\mathcal{M}'_{r,m}$ of
the reproducing kernel space.

\begin{lemma}\label{dualreproducingkernel}
The anti-dual $\mathcal{M}'_{r,m}$ of $\mathcal{M}_{r,m}$ is
canonically isomorphic to \linebreak $L_{r',m^{-1}}(G)/\mathcal{M}_{r,m}^\bot$, where
\begin{align}
  \mathcal{M}_{r,m}^\bot
  = \left\{
      \widetilde{F} \in L_{r',m^{-1}}(G)
      ~\middle|~
      \langle{\widetilde{F} },{F}\rangle_{L_2} = 0
      \mbox{ for all } F \in \mathcal{M}_{r,m}
    \right\}
  \label{eq:ReproducingKernelSpaceDual}
\end{align}
and $1/r+1/r'=1$.
Hence, for every $\Gamma\in\mathcal{M}'_{r,m}$ there is a
$\widetilde{F} \in L_{r',m^{-1}}(G)$ such that
$\Gamma(F) = \langle{\widetilde{F} },{F}\rangle_{L_2}$
for all $F \in \mathcal{M}_{r,m}$.
\end{lemma}

\begin{proof}
Since $\mathcal{M}_{r,m}$ is a closed subspace of $L_{r,m}(G)$,
\cite[Proposition 1.4]{Mo00} yields that $\mathcal{M}'_{r,m}$ is canonically isomorphic
to $L'_{r,m}(G)/\mathcal{M}_{r,m}^\bot$.
The claim follows because $L'_{r,m}(G)$ is canonically isomorphic to $L_{r',m^{-1}}(G)$.
\end{proof}

Some more preparations are necessary.
Given a compact neighborhood $Q \subset G$ of $e$
with $Q = \overline{\operatorname{int} Q}$,
the \emph{local maximal function (with respect to the right regular representation)} $M^\rho_Q f$ of $f \in  L_0(G) $ is defined by
\begin{align}\label{local_maximum_function}
  M^\rho_Q f(x) := \|f \cdot \rho(x)\chi_Q\|_{L_\infty},
  \quad \text{whence} \quad
  \widecheck{M}^\rho_Q f(x) := M^\rho_Q f (x^{-1}) = \| f \|_{L_\infty(Q x)} \, .
\end{align}
Then, for a function space $Y$ on $G$, we define
\begin{align}\label{wiener_amalgam_space}
  \mathcal{M}^\rho_Q(Y)
  := \left\{
        f \in  L_0(G) 
        ~\middle|~
        \vphantom{M^j} \smash{\widecheck{M}^\rho_Q} f \in Y
     \right\}.
\end{align}

Now we define the {\em $Q$-oscillation} of a function $f$ with respect to $Q$ as
\begin{equation} \label{oscillation}
  \mbox{osc}_Q f(x) := \sup_{u \in Q} |f(ux)-f(x)|.
\end{equation}
The decay-properties of the $Q$-oscillation play an important role in view of
the discretization of coorbit spaces. To this end, the following lemma is useful.
Since the proof is a simple generalization of the proof of \cite[Lemma~4.6]{gro91}, it is deferred to the appendix.

\begin{lemma}\label{lem:osc}
  Let $w$ be a weight on $G$, let $p\in(1,\infty)$, and assume that $f : G \to \mathbb{C}$ is
  continuous and that $f \in \mathcal{M}^\rho_{Q_0}(L_{p,w})$ for some compact unit
  neighborhood $Q_0$ with $Q_0 = \overline{\operatorname{int}
    Q_0}$.
  Then the following hold:
  \begin{enumerate}[i)]
  \item $\lVert{\mathrm{osc}_{Q_0} f}\rVert_{L_{p,w}} < \infty$.

  \item For arbitrary $\varepsilon > 0$, there is a unit
            neighborhood $Q_\varepsilon \subset { Q_0}$ such that for each
            unit neighborhood $Q \subset Q_\varepsilon$, we have
            $\lVert{\mathrm{osc}_Q f}\rVert_{L_{p,w}} < \varepsilon$.
            Put briefly,
            \begin{align*}
              \lim_{Q \to \{e\}} \lVert{\mathrm{osc}_Q f}\rVert_{L_{p,w}} = 0.
            \end{align*}
  \end{enumerate}
\end{lemma}

\subsection{An Assumption on the Kernel}

From now on we make the following assumption
on the reproducing kernel space.

\begin{assumption}\label{assumption1}
  Assumption \ref{assume:KernelAlmostIntegrable} is satisfied, and
  $\operatorname{span} \{\lambda(x) K\}_{x\in G}$
  is dense in $\mathcal{M}_{r,m}$.
\end{assumption}

This assumption is similar to the density of
    $\mathrm{span} \, \{\pi(x) K\}_{x \in G}$ in $\mathcal{H}$---which is
    equivalent to $K$ being a cyclic vector for the representation $\pi$
    on $\mathcal{H}$---and in $\mathcal{M}^{\mathcal{T}_w}$, which is
    Assumption 3 of \cite{dadedelastte17} and fulfilled in our setting, as can be seen by
    combining Theorems \ref{intersections} and \ref{intersections-1}

In the following we will denote with $RC_K$ the right convolution operator  $RC_K f := f\ast K$, where the space on which $RC_K$ acts may vary depending of the context.

Before we provide a sufficient condition under which
Assumption \ref{assumption1} is fulfilled (see Lemma~\ref{lem:kernel_continuity}),
we need a couple of auxiliary results.

\begin{proposition}\label{prop:RCK_bounded}
Assume that for all $f\in  L_{r,m}(G) $, $f$ and
$K$ are convolvable (in the sense that $f \cdot \lambda(x)\widecheck{K} \in L_1(G)$ for almost all $x\in G$) and $f\ast K\in  L_{r,m}(G) $, then the right
convolution operator
\[
RC_K:  L_{r,m}(G) \to  L_{r,m}(G) , \qquad RC_Kf= f\ast K,
\]
is bounded.
\end{proposition}

\begin{proof}
For $r=2$ the result is stated in \cite[Proposition~3.10]{Halmos}, whose
proof holds true for any $p$.  Indeed, by the closed graph theorem, it is
enough to show that $RC_K$ is a closed operator. Take a sequence
$(f_n)_{n\in\mathbb{N}}$ converging to $f\in   L_{r,m}(G) $ such that $(RC_Kf_n)_{n\in\mathbb{N}}$
converges to $g\in L_{r,m}(G) $. By a sharp version of the Riesz-Fischer
theorem, see \cite[Theorem~13.6]{Hitchhiker}, there exists a
positive function $g\in L_{r,m}(G) $ such that, possibly passing twice to a
subsequence, there exist two null sets $E,F$ such that for all
$y\in G \setminus E$ and $x\in G \setminus F$
\begin{alignat*}{1}
 & |f_n(y)|\leq g(y),   \\
 & \lim_{n\to\infty} f_n(y) = f(y), \\
 & \lim_{n\to\infty}RC_K f_n(x) = g(x).
\end{alignat*}
Furthermore, by definition of convolution and possibly re-defining the
null set $F$, we get that for all $x\in G \setminus F$ and all $n\in\mathbb{N}$ the mappings
\[
y \mapsto f_n(y) K(y^{-1}x), \qquad y \mapsto g(y) K(y^{-1}x) 
\]
are integrable. Then, given $x\in G \setminus F$, for all  $y\in G \setminus E$
\begin{align*} 
  |f_n(y) K(y^{-1}x) |\leq |g(y) K(y^{-1}x)|,   \quad \lim_{n\to\infty} f_n(y) K(y^{-1}x) = f(y) K(y^{-1}x) .
\end{align*}
For $x \in G \setminus F$, the function $y\mapsto g(y) K(y^{-1}x)$ is integrable, so that by
dominated convergence we obtain
\[
g(x)=\lim_{n\to\infty} \int_G f_n(y) K(y^{-1}x)\,dy = \int_G f(y) K(y^{-1}x)\,dy = f\ast K (x),
\] 
so $RC_K$ is indeed closed.
\end{proof}

\begin{proposition}\label{prop:kernel_property}
Denote by $r'$ the dual exponent
$1/r+1/r'=1$. Assume that the right 
convolution operator
\[
RC_K:  L_{r,m}(G) \to  L_{r,m}(G) , \qquad RC_Kf= f\ast K
\]
is bounded, then  
\begin{enumerate}[i)]
\item the right convolution operator is bounded on $ L_{r',m^{-1}}(G) $ and it
  coincides with the adjoint of $RC_K$;
\item the operator $RC_K$ is a projection from $ L_{r,m}(G) $ onto the
  reproducing kernel Banach space $\mathcal M_{r,m}$.
\end{enumerate}
\end{proposition}

Here and in the following, the duality pairing is the sesqui-linear
form
\[
\langle{f},{g}\rangle_{L_2}= \int_G f(x) \overline{g(x)}\, dx, \qquad f\in L_{r,m}(G) ,\, g\in L_{r',m^{-1}}(G) .
\]

\begin{proof}
Since $RC_K$ is a bounded operator on $L_{rm}(G)$, the adjoint is a bounded
operator on $ L_{r,m}(G) '$. Take $g\in L_{r',m^{-1}}(G) $ and $f\in C_c(G)\subset L_{r,m}(G) $, then
\begin{alignat*}{1}
 \langle{RC_K^*g},{f}\rangle_{L_2} & =  \langle{g},{RC_Kf}\rangle_{L_2} =\int_G g(x)
 \left(\int_G \overline{ f(y)K(y^{-1}x)}\,dy\right)dx \\
& = \int_G 
 \left(\int_G g(x) K(x^{-1}y)\,dx\right) \overline{ f(y)}\,dy\\
& =\langle{ g\ast K},{f}\rangle_{L_2},
\end{alignat*}
where $\overline{K(y^{-1}x)}=K(x^{-1}y)$. Note that we can interchange
the integral by Fubini's theorem since 
\[
\int_G \lvert{g(x)}\rvert
 \left(\int_G \lvert{ f(y)K(y^{-1}x)}\rvert\,dy\right)dx \leq \lVert{g}\rVert_{L_{r',m^{-1}}}\cdot
 \lVert{\lvert{f}\rvert*\lvert{K}\rvert}\rVert_{L_{r,m}} 
\]
and $\lvert{f}\rvert*\lvert{K}\rvert\in L_{r,m}(G) $ by Young's inequality~\eqref{eq:young2}
with $q=1$ and $p=r$, $f\in L_{1,m}(G)$ and $g=K\in L_{r,w}(G)$.
Note that Fubini's theorem shows that
    \[
      \int |f(y)| \cdot (|g| \ast |K|)(y) \, dy
      < \infty \, .
    \]
Since this holds for any $f \in C_c (G)$, we see $|g| \ast |K| < \infty$
almost everywhere, so that $g$ and $K$ are convolvable.
By density of $C_c(G)$ in $ L_{r,m}(G) $ we get that $RC_K^*g=g\ast K$, so that
$g\ast K\in  L_{r',m^{-1}}(G) $.  Hence the convolution operator acts continuously
on $ L_{r',m^{-1}}(G) $ and it coincides with $RC_K^*$. 

To show the second claim, observe first that for any $f\in
C_c(G)\subset \mathcal{T}_w\subset  L_{r',m^{-1}}(G) $, since $K\in\mathcal{T}_w$, both $\lvert{f}\rvert*\lvert{K}\rvert$ and
$(\lvert{f}\rvert*\lvert{K}\rvert)\ast\lvert{K}\rvert$ exist, so that by~(77d) of
\cite{dadedelastte17} the convolution is associative and
\[ RC_K^2 f = (f\ast K)\ast K = f\ast (K\ast K)= f \ast K = RC_K f.\]
By density, and since $RC_K$ is bounded on $ L_{r,m}(G) $ by assumption, we get that $RC^2_K=RC_K$ and hence
$\operatorname{Ran}RC_K\subset \mathcal M_{r,m}$. The other inclusion is trivial.  
\end{proof}

As a consequence of the above result, we get the following corollary.

\begin{corollary}
 Denote by $r'$ the dual exponent
$1/r+1/r'=1$ and assume that the right 
convolution operator $RC_K$ is bounded on $ L_{r,m}(G) $.  The sesqui-linear
pairing on $\operatorname{Co}(L_{r,m}) \times \operatorname{Co}(L_{r',m^{-1}})$ given by
\[\langle{T},{T'}\rangle_{\operatorname{Co}(L_{r,m})} = \langle{V_eT},{V_eT'}\rangle_{L_2}\]
is such that the linear map
\[
      T' \mapsto \Big(
                   T \mapsto \overline{\langle T, T' \rangle_{\operatorname{Co}(L_{r,m})}}
                 \Big)
    \]
    is an isomorphism of $\operatorname{Co}(L_{r',m^{-1}})$ onto the anti-linear dual
    of $\operatorname{Co}(L_{r,m})$.
\end{corollary}

\begin{proof}
We identify $\operatorname{Co}(L_{r,m})$ with $\mathcal M_{r,m}$  by the extended
voice transform $V_e$, so that the pairing becomes
\[
\langle{f},{g}\rangle_{L_2} = \int_G f(x) \overline{g(x)}\, dx,\qquad f\in \mathcal
M_{r,m},\, g\in \mathcal M_{r',m^{-1}}.
\] 
Since $g\in L_{r',m^{-1}}(G) $, clearly $f\mapsto \overline{\langle{f},{g}\rangle_{L_2}}$ is a
continuous anti-linear map, which we denote by $\Gamma_g$, on $\mathcal
M_{r,m}$ whose norm 
is 
\begin{align*}
  \lVert{\Gamma_g}\rVert &=  \sup\left\{ \lvert{\langle{f},{g}\rangle_{L_2}}\rvert \mid f\in \mathcal
      M_{r,m},\,\lVert{f}\rVert_{L_{r,m}} \leq 1 \right\} \\
      &\leq \sup\left\{\lvert{\langle{h},{g}\rangle_{L_2}}\rvert \mid h\in  L_{r,m}(G) ,\,\lVert{h}\rVert_{L_{r,m}} \leq 1 \right\}=\lVert{g}\rVert_{L_{r',m^{-1}}}.
\end{align*}
Next, since $L_{r,m}(G)$ is the dual of $L_{r',m^{-1}}(G)$, there is
    $h \in L_{r,m}(G)$ with $\|h\|_{L_{r,m}} \leq 1$ such that
    $\|g\|_{L_{r',m^{-1}}} = \langle h, g \rangle_{L_2}$.
    Now, setting $c := \|RC_K\|_{L_{r,m} \to L_{r,m}}$ and
    $f = c^{-1} \cdot RC_K h$, we have $\|f\|_{L_{r,m}} \leq 1$
    and
    \[
      \langle f, g \rangle_{L_2}
      = c^{-1} \langle RC_K h, g \rangle_{L_2}
      = c^{-1} \langle h, RC_K g \rangle_{L_2}
      = c^{-1} \langle h, g \rangle_{L_2}
      = c^{-1} \|g\|_{L_{r',m^{-1}}} \, .
    \]
    Hence, $c^{-1} \cdot \|g\|_{L_{r',m^{-1}}} \leq \|\Gamma_g\| \leq \|g\|_{L_{r',m^{-1}}}$.

We now  prove that the map $g\mapsto \Gamma_g$ is surjective. Take
$\Gamma$ in the anti-linear dual of $\mathcal M_{r,m}$. Since
$\mathcal M_{r,m}$ is a subspace of $ L_{r,m}(G) $ there exists $g'\in L_{r',m^{-1}}(G) $
such that $\Gamma(f)=\overline{\langle{f},{g'}\rangle_{L_2}}$ for all $f\in \mathcal
M_{r,m}$.  By setting $g=RC_Kg'\in \mathcal M_{r',m^{-1}}$, as above
\[
\Gamma(f)=\overline{\langle{f},{g'}\rangle_{L_2}}=
\overline{\langle RC_K f, g \rangle_{L_2}}=\overline{\langle{f},{g}\rangle_{L_2}}=\Gamma_g(f) \quad \text{for all}\,f\in \mathcal M_{r,m},
\]
thus $\Gamma=\Gamma_g$.
\end{proof}

Now we can prove that in the following setting
Assumption \ref{assumption1} is fulfilled.

\begin{lemma}\label{lem:kernel_continuity}
Fix $r\in (1,\infty)$ and assume that  the right 
convolution operator $RC_K$ is bounded on $ L_{r,m}(G) $. Then the sets
$\operatorname{span}\{\pi(x)u\}_{x\in G}$ and $\operatorname{span}\{\lambda(x)K\}_{x\in G}$ are dense in
$\operatorname{Co}(L_{r,m})$ and $\mathcal M_{r,m}$, respectively. Thus, Assumption~\ref{assumption1} is fulfilled.
\end{lemma}

\begin{proof}
By the correspondence principle, it is enough to show the second
claim. Let $\Gamma\in \mathcal M'_{r,m}$ be such that for all $x\in   G$,
\[ \Gamma(\lambda(x)K)=0.\]
By the above corollary, there exists $g \in \mathcal{M}_{r',m^{-1}}$ such
    that $\Gamma(f) = \langle g,f \rangle_{L_2}$ for all
    $f \in \mathcal{M}_{r,m}$.
    In particular,
    \[
      0
      = \Gamma \big(\lambda(x) K \big)
      = \langle g, \lambda(x) K \rangle_{L_2}
      = g \ast K(x)
      = RC_K g (x)
    \]
    for all $x \in G$, that is, $RC_K g = 0$.
    Since $g \in \mathcal{M}_{r',m^{-1}}$, this implies $g=0$ and then
    $\Gamma = 0$.
    Since this holds for any $\Gamma \in \mathcal{M}'_{r,m}$ such that
    $\Gamma(\lambda(x)K) = 0$ for all $x \in G$, we see that
    $\mathrm{span} \{\lambda(x) K\}_{x \in G}$ is dense in $\mathcal{M}_{r,m}$.
\end{proof}

By Young's inequality we know that the $L_1(G)$-integrability of $K \cdot w$
implies that the (right) convolution operator $RC_K$ is a bounded operator acting on
$L_{p,m}(G)$ for all $1<p<\infty$.
But for general $K \in \mathcal{T}_w$ this question is unclear.
As we will show in Sect. \ref{sec:examples} there are kernels that
act boundedly on all $L_p(G)$ without being integrable.
But in Sect. \ref{sec:Obstructions} we also show that there
exist kernels for a very similar setting that are contained in $\mathcal{T}_w$
but that do \emph{not} give rise to bounded operators on $L_{p,m}(G)$.

\subsection{Atomic Decompositions} \label{atomicdecomp}

This section is dedicated to finding possible atomic decompositions of
coorbit spaces, provided that Assumption \ref{assumption1} is fulfilled.
The main results of this section will be stated in Theorem \ref{maintheorem}.

But before that, we need to introduce some notation.
First, for each $n \in \mathbb{N}$, we choose a countable subset
$Y_n = \{ x_{j,n} \}_{j \in \mathcal{J}_n} \subset G$
such that
\begin{align}
  & Y_n\subset Y_{n+1}, \\
  & \overline{\bigcup_{n\in\mathbb{N}} Y_n} = G .
\end{align}
Moreover, for every $n \in \mathbb{N}$, we assume that there exists
a compact neighborhood $Q_n$ of the identity $e \in G$, such that $Y_n$ is
\emph{$Q_n$-dense} in $G$, i.e.,
\begin{align}
  G = \bigcup_{j\in\mathcal{J}_n} x_{j,n} Q_n.
\end{align}
Additionally we assume each $Y_n$ to be 
    uniformly  relatively $Q_n$-separated, i.e. there
  exists an integer $\mathcal{I}$, independent of $n$, and
  subsets $Z_{n,i}\subset Y_n$, $1\leq i\leq \mathcal{I}$, such that
\begin{align}\label{eq:75}
Y_n = \bigcup_{i=1}^{\mathcal{I}}Z_{n,i}
\end{align}
and for all $x,y\in Z_{n,i}$, $1\leq i\leq \mathcal{I}$,
it holds $xQ_n\cap yQ_n\neq\emptyset$ if and only if $x=y$.

By $\Psi_n = \{\psi_{n,x}\}_{x \in Y_n}$ we denote a \emph{partition of unity}
subordinate to the $Q_n$-dense set $Y_n$, i.e.,
\begin{align}
  & 0\leq \psi_{n,x} \leq 1,             \label{eq:70a} \\
  & \sum_{x\in Y_n} \psi_{n,x} \equiv 1, \label{eq:70b} \\
  & \operatorname{supp}(\psi_{n,x}) \subset xQ_n.      \label{eq:70c}
\end{align}

We also assume that  the family
$\Psi_n = \{\psi_{n,x}\}_{x \in Y_n}$ is linearly
  independent as a.e.~defined functions, i.e. for any finite subset $ X\subset Y_n$ and
$(\alpha_x)_{x\in X}\in\mathbb C^X$, the condition
\[
  \sum_{x\in X} \alpha_x  \psi_{n,x}(y)=0 \label{eq:2}
\]
for almost all $y\in G$ implies that $\alpha_x= 0$ for all $x\in X$.

We now denote with $X_n$ a \emph{finite} subset of $Y_n$, such that
\begin{align}
  & X_n \subset X_{n+1},               \label{eq:68}\\
  &\overline{\bigcup_{n\in\mathbb{N}}X_n} = G. \label{eq:69}
\end{align}
Therefore, for every $n\in\mathbb{N}$, the finite set of functions
$\{\psi_{n,x}\}_{x\in X_n}$ is similar to a partition of unity
subordinate to the family $(xQ_n)_{x \in X_n}$.

For each $n \in \mathbb{N}$ and $1 < r < \infty$, set
\begin{align}\label{T_n}
  T_n: L_{r,m}(G)\to \mathcal{M}_{r,m},
       \qquad
       T_n F := \sum_{x \in X_n} \langle{F},{\psi_{n,x}}\rangle_{L_2} \, \lambda(x) K.
\end{align}
We observe that this operator is well-defined.
Since the sum is finite, we only have to verify that each term of the sum
is a well-defined element of $\mathcal{M}_{r,m}$.
It is easy to verify that the reproducing identity holds for $\lambda(x) K$,
since it holds for $K$. Moreover we have $\lambda(x) K \in L_{r,m}(G)$ by
Assumption \ref{assume:KernelAlmostIntegrable} and by translation invariance
of the spaces $L_{r,m}(G)$; thus, $\lambda(x) K \in \mathcal{M}_{r,m}$.
Finally, the pairing
\begin{align*}
  \langle{F},{\psi_{n,x}}\rangle_{L_2} = \int_{G} F(y) \psi_{n,x}(y) \ dy
\end{align*}
is well-defined for all $x \in X_n$, since $\psi_{n,x}$ is bounded with compact
support, so that $\psi_{n,x} \in L_{r',m^{-1}}(G)$.

Now we define $V_n=\operatorname{Ran}T_n$, which is a finite dimensional subspace of
$\mathcal{M}_{r,m}$, as well as $\widetilde V_n = V_e^{-1}(V_n)$, which is a
finite dimensional subspace of $\operatorname{Co}(L_{r,m})$ by the correspondence principle.
We show the following result concerning the structure of the spaces $V_n$:

\begin{lemma}\label{lem:VnCharacterization}
The following holds for all $n \in \mathbb{N}$:
  \begin{alignat}{1}
  & V_n = \operatorname{span}\left\{ \lambda(x)K \right\}_{x\in X_n}, \label{eq:70} \\
  & V_n \subset V_{n+1},                                \label{eq:71} \\
  & \overline{\bigcup_{n\geq 1} V_n} = \mathcal{M}_{r,m}         \label{eq:72}.
\end{alignat}
\end{lemma}
\begin{proof}
We start by showing \eqref{eq:70}.
By the construction we made above,
$V_n\subseteq \operatorname{span}\left\{ \lambda(x)K \right\}_{x\in X_n}$.

We first observe that the map
\[
F\mapsto \left(\langle F, \psi_{n,x} \rangle_{L_2(G)}\right)_{x \in X_n}
\]
is surjective from $L_{r,m}(G)$ to  $\mathbb C^{X_n}$. Indeed, if this was not true, there would be a nonzero family  $ (\alpha_x)_{x \in X_n} \in \mathbb{C}^{X_n}$
satisfying $\sum_{x \in X_n} \alpha_x \langle F, \psi_{n,x}
\rangle_{L_2}=0$ for all $F\in L_{r,m}(G)$, then
$\sum_{x \in X_n} \alpha_x \psi_{n,x}=0$ in $L_{r',m^{-1}}(G)$ and,
hence, almost everywhere, then by assumption
$\alpha_x =0$ for all $x\in X_n$, a contradiction.
It follows that $\operatorname{span}\left\{ \lambda(x) K \right\}_{x\in X_n} \subseteq V_n$
and~\eqref{eq:70} holds true.


Equation~\eqref{eq:71} is an easy consequence of~\eqref{eq:68}
and~\eqref{eq:70}.

It remains to show \eqref{eq:72}.
Since the sequence $(V_n)_{n\in\mathbb{N}}$ is an increasing family of subspaces, and
since $V_n \subset \mathcal{M}_{r,m}$ for all $n \in \mathbb{N}$, the set
$\overline{\bigcup_{n\geq 1} V_n}$ is a subspace of the closed space $\mathcal{M}_{r,m}$.
Hence, by the Hahn-Banach theorem, condition~\eqref{eq:72} is equivalent to the
following condition: If $\Gamma\in \mathcal{M}'_{r,m}$ satisfies
\[
  \langle{\Gamma},{F}\rangle_{\mathcal{M}'_{r,m}\times\mathcal{M}_{r,m}} = 0,
  \quad \text{for all } F \in V_n,\ n \in \mathbb{N},
\]
then $\Gamma = 0$ in $\mathcal{M}'_{r,m}$.
By Lemma \ref{dualreproducingkernel} we can write
$\langle{\Gamma},{F}\rangle_{\mathcal{M}'_{r,m} \times \mathcal{M}_{r,m}} = \langle{g},{F}\rangle_{L_2}$ for all
$F \in \mathcal{M}_{r,m}$, for a suitable $g \in L_{r',m^{-1}}(G)$.
Since $\lambda(x) K \in \mathcal{M}_{r,m}$, $x \in G$, for every $f \in L_{r',m^{-1}}(G)$ with
$f-g \in \mathcal{M}_{r,m}^\bot$, it holds for all $x \in G$,
\begin{align*}
  (g \ast K)(x)
  = \langle{g},{\lambda(x) \, K}\rangle_{L_2}
  = \langle{\Gamma},{\lambda(x) \, K}\rangle_{\mathcal{M}_{r,m}' \times \mathcal{M}_{r,m}}.
\end{align*}
Now, with $F = T_nf$ for some $f \in L_{r,m}(G)$, we obtain
\begin{align*}
  0
  &= \langle{\Gamma},{T_nf}\rangle_{\mathcal{M}'_{r,m}\times\mathcal{M}_{r,m}}
   = \sum_{x\in X_n}
       \langle{\psi_{n,x}},{f}\rangle_{L_2}
       \cdot
       \langle{\Gamma},{\lambda(x)K}\rangle_{\mathcal{M}'_{r,m}\times\mathcal{M}_{r,m}} \\
  &= \sum_{x\in X_n} \langle{\psi_{n,x}},{f}\rangle_{L_2} \cdot (g \ast K)(x)
   = \langle{\sum_{x \in X_n} (g \ast K)(x)\psi_{n,x}},{f}\rangle_{L_2}.
\end{align*}
Since this holds for any $f \in L_{r,m}(G)$, we get
$\sum_{x \in X_n} (g \ast K)(x) \psi_{n,x} = 0$ in $L_{r',m^{-1}}(G)$
for all $n \in \mathbb{N}$.
Because the finite family $\left\{\psi_{n,x} \right\}_{x \in X_n}$ is linearly
independent as elements of $L_{r',m^{-1}}(G)$, we have $(g \ast K)(x) = 0$ for all $x \in X_n$ and $n \in \mathbb{N}$.
Therefore, by \eqref{eq:69}, the function $g\ast K$ vanishes on a dense subset
of $G$.
But since we have $g \ast K (x) = \langle g, \lambda(x) K \rangle_{L_2}$ with
$g \in L_{r',m^{-1}}(G)$, and since the map $G \to L_{r,m}(G), x \mapsto \lambda(x) K$
is continuous, we see that $g \ast K : G \to \mathbb{C}$ is a
continuous functions, so that we get $g \ast K \equiv 0$, i.e.,
$\langle \Gamma, \lambda(x) K \rangle_{\mathcal{M}_{r,m}' \times \mathcal{M}_{r,m}} = 0$ for all $x \in G$.

By Assumption \ref{assumption1}, this implies $\Gamma = 0$ as an element
of $\mathcal{M}'_{r,m}$, which proves \eqref{eq:72}.
\end{proof}

\begin{remark}
By the correspondence principle, analogous results to \eqref{eq:70},
\eqref{eq:71} and \eqref{eq:72} hold true for $\widetilde V_n$.
This can be seen as follows: Since it holds $V_e\pi(x)u = \lambda(x) K$ for all
$x \in X_n$, by \eqref{eq:70} we obtain
\begin{align}\label{eq:80}
  \widetilde V_n = \operatorname{span} \left\{\pi(x)u\right\}_{x\in X_n}.
\end{align}
Hence, the nesting property $\widetilde V_n \subset \widetilde V_{n+1}$
analogous to \eqref{eq:71} is straightforward.
By the correspondence principle, it follows from \eqref{eq:71} that
\begin{align}\label{eq:86}
  \overline{\bigcup_{n\in\mathbb{N}}\widetilde V_n} = \operatorname{Co}(L_{r,m}).
\end{align}
\end{remark}

With the spaces $V_n$ at hand, in the following we will turn to projections
from $\mathcal{M}_{r,m}$ onto $V_n$ and their properties.
To this end, let $\pi_n:\mathcal{M}_{r,m}\to V_n$ be the metric projection defined by
\begin{equation}
  \pi_n (F) = \operatorname{argmin}_{g \in V_n} \lVert{F - g}\rVert_{\mathcal{M}_{r,m}}.
  \label{eq:MetricProjectionDefinition}
\end{equation}
Since $\mathcal{M}_{r,m}$ is a closed subspace of $L_{r,m}(G)$ with $1 < r < \infty$,
the space $\mathcal{M}_{r,m}$ is a uniformly convex Banach space and every $V_n$ is
convex and closed; therefore $\pi_n$ is a well-defined and unique function,
see \cite[Proposition 3.1]{GoRe84}.
Similarly, we define the projection
$\widetilde\pi_n : \operatorname{Co}(L_{r,m}) \to \widetilde V_n$ by setting
$\widetilde \pi_n = V_e^{-1}\pi_n V_e$.

The following lemma gives us a first norm estimate for this metric projection.

\begin{lemma}\label{density}
Given $\varepsilon > 0$ and $F \in \mathcal{M}_{r,m}$, there exists $n^{*}=n^{*}_{F,\varepsilon} \in \mathbb{N} $ such that for all $n\geq n^{*}$
it holds
\begin{alignat}{1}
  & \lVert{F - \pi_n(F)}\rVert_{\mathcal{M}_{r,m}} \leq \varepsilon,                   \label{eq:71a} \\
  & \lVert{\pi_n(F)}\rVert_{\mathcal{M}_{r,m}} \leq (1+\varepsilon)\lVert{F}\rVert_{\mathcal{M}_{r,m}}. \label{eq:71b}
\end{alignat}
\end{lemma}

\begin{proof}
If $F = 0$ the claim is clear since $0 \in V_n$ so that $\pi_n(F) = 0$.
Hence, we can assume that $F \neq 0$.
Let $\delta := \min\left\{1, \lVert{F}\rVert_{\mathcal{M}_{r,m}} \right\} \cdot \varepsilon > 0$.
By~\eqref{eq:72} there exists $n^*\geq 1$ and $g\in V_{n^*}$ such that
$\lVert{F-g}\rVert_{\mathcal{M}_{r,m}}\leq \delta$. For all $n\geq n^*$, by~\eqref{eq:71}
we have $g \in V_n$ and, by definition of the metric projection,
\[
  \lVert{F-\pi_n(F)}\rVert_{\mathcal{M}_{r,m}}
  \leq \lVert{F-g}\rVert_{\mathcal{M}_{r,m}}
  \leq \delta
  \leq \varepsilon.
\]
The triangle inequality gives
\[
  \lVert{\pi_n(F)}\rVert_{\mathcal{M}_{r,m}}
  \leq \lVert{F -\pi_n(F)}\rVert_{\mathcal{M}_{r,m}} + \lVert{F}\rVert_{\mathcal{M}_{r,m}}
  \leq \delta + \lVert{F}\rVert_{\mathcal{M}_{r,m}}
  \leq (1+\varepsilon) \lVert{F}\rVert_{\mathcal{M}_{r,m}} ,
\]
which concludes the proof.
\end{proof}

The following auxiliary result establishes a first upper bound for certain
coefficients related to functions $F\in \mathcal{M}_{r,m}$.
This will be used for the atomic decomposition afterwards.

\begin{proposition}\label{prop:upper_bound}
 For any $F \in L_{r,m}(G)$ and $n \in \mathbb{N}$, let the
 coefficients $c_{n,x} \in \mathbb{C}$, $x \in X_n$, be defined via
 \[
   c_{n,x} := \int_G F(y) \psi_{n,x}(y)~dy.
 \]
 Then the inequality
 \begin{equation}
   \left( \sum_{x \in X_n} \lvert{c_{n,x}}\rvert^r m(x)^r \right)^{1/r}
   \leq |Q_n|^{1/r'} \cdot \sup_{q \in Q_n} w(q) \cdot \lVert{F}\rVert_{L_{r,m}}
   \label{eq:76}
 \end{equation}
 holds, where $|Q_n|$ denotes the Haar measure of the set $Q_n$
 and $r'$ denotes the dual exponent of $r$.
\end{proposition}


\begin{proof}
We first note that, since $\psi_{n,x}$ is compactly supported and bounded,
the coefficient $c_{n,x}$ is well-defined.

Next, we observe that if $\psi_{n,x}(y) \neq 0$, then $y = x q_n$ for
some $q_n \in Q_n$, and hence $m(x) = m(x q_n q_n^{-1})
\leq m(x q_n) \cdot w(q_n^{-1})
\leq m(y) \cdot \sup_{q \in Q_n} w(q)$.
This shows
\begin{equation}
  \begin{split}
    m(x) \cdot | c_{n,x} |
    & \leq m(x) \cdot \int_G | F(y) | \cdot \psi_{n,x}(y) \, dy \\
    & \leq \sup_{q \in Q_n} w(q)
           \cdot \int_G | (mF)(y) | \cdot \psi_{n,x}(y) \, dy \, .
  \end{split}
  \label{eq:CoefficientNormWeightConsideration}
\end{equation}
We will now further estimate the integral on the right-hand side,
setting $F_0 := m \cdot F$ for brevity.

To this end, we define the measure $d\mu_x$ on $G$ (for $x \in X_n$) by setting
\[
  d\mu_x(y) = \frac{\psi_{n,x}(y)}{\|\psi_{n,x}\|_{L_1}}~dy
\]
and readily observe that $\int_G 1~d\mu_x = 1$.
Thus, by Jensen's inequality, see \cite[Theorem 10.2.6]{DudleyRealAnalysisProbability}, we obtain
\begin{align*}
        \left(
          \int_G
            |F_0(y)|
            \frac{\psi_{n,x}(y)}{\|\psi_{n,x}\|_{L_1}}
          ~dy
        \right)^r
  &=    \left( \int_G |F_0(y)| ~d\mu_x(y) \right)^r \\
  &\leq \int_G |F_0(y)|^r ~d\mu_x(y) \\
  &=    \int_G |F_0(y)|^r \frac{\psi_{n,x}(y)}{\|\psi_{n,x}\|_{L_1}} ~dy.
\end{align*}
By the properties of $\Psi_n$, see \eqref{eq:70a}, \eqref{eq:70b}
and \eqref{eq:70c}, it holds
\[
  \lVert{\psi_{n,x}}\rVert_{L_1}
  = \int_G \psi_{n,x}(y)~dy\leq \int_{xQ_n} 1~dy
  = \int_{Q_n} 1~dy
  = |Q_n|.
\]
Recalling \eqref{eq:CoefficientNormWeightConsideration}, we thus see
\begin{align*}
      &\sum_{x\in X_n} (m(x) \cdot \lvert{c_{n,x}}\rvert)^r \\
   \leq &~\sup_{q \in Q_n} w(q)^r \cdot
         \sum_{x \in X_n}
           \lVert{\psi_{n,x}}\rVert^r_{L_1}
           \left(
             \int_G
               \lvert{F_0(y)}\rvert\,
               \frac{\psi_{n,x}(y)}{\lVert{\psi_{n,x}}\rVert_{L_1}}
             ~dy
           \right)^r \\
   \leq &~\sup_{q \in Q_n} w(q)^r \cdot
         \sum_{x\in X_n}
           \lVert{\psi_{n,x}}\rVert^r_{L_1}
           \int_G
             \lvert{F_0(y)}\rvert^r\,
             \frac{\psi_{n,x}(y)}{\lVert{\psi_{n,x}}\rVert_{L_1}}
           ~dy \\
   \leq &~\sup_{q \in Q_n} w(q)^r \cdot
         \sup_{x \in X_n}
           \lVert{\psi_{n,x}}\rVert^{r-1}_{L_1}
           \sum_{x \in X_n}
             \int_G
               \lvert{F_0(y)}\rvert^r \psi_{n,x}(y)
             ~dy \\
   \leq &~\sup_{q \in Q_n} w(q)^r \cdot
         |Q_n|^{r-1} \cdot \lVert{F}\rVert_{L_{r,m}}^r,
\end{align*}
which concludes the proof.
\end{proof}

With this at hand we are able to give a first atomic decomposition for functions
$F \in V_n$, $n \in \mathbb{N}$, as well as an estimate for the norm of the
coefficients involved.

\begin{lemma}
Given $n \in \mathbb{N}$, for all $F \in V_n$ the following
atomic decomposition holds true:
\begin{equation}
  F = \sum_{x\in X_n} c(F)_{n,x} \, \lambda(x) K,
  \label{eq:73}
\end{equation}
where the coefficients $c(F)_{n,x}$ are of the form
\begin{equation}
  c(F)_{n,x} = \langle{S_n F},{\psi_{n,x}}\rangle_{L_2},
  \label{eq:VnAtomicDecompositionCoefficientsExplicit}
\end{equation}
where $S_n$ denotes any linear right inverse of $T_n : L_{r,m}(G) \to V_n$.
In particular, the coefficients depend linearly on $F$ and they satisfy
\begin{equation}
 \left(
   \sum_{x \in X_n} \lvert{c(F)_{n,x}}\rvert^r m(x)^r
 \right)^{1/r}
 \leq C_n \lVert{F}\rVert_{\mathcal{M}_{r,m}},
  \label{eq:VnAtomicDecompositionCoefficientNorm}
\end{equation}
with $C_n = \| S_n \| \cdot | Q_n |^{1/r'} \cdot
\sup_{q \in Q_n} w(q)$.
\end{lemma}
%

\begin{proof}


We first observe that the operator $T_n$ admits a bounded right inverse
$S_n: V_n\to L_{r,m}(G)$.
Indeed, by \cite[Theorem 2.12]{bre11} the existence of a bounded right inverse
is equivalent to the existence of a topological supplement of
the kernel of $T_n$. However, since the spaces $V_n$ are finite dimensional,
such a topological supplement exists, see \cite[Example 2.4.2]{bre11}.

In the remainder of the proof, we denote by $S_n$ an arbitrary linear right
inverse of $T_n$.
Thus, for all $F\in V_n$ we have the decomposition
\[
  F = T_n S_n F
    = \sum_{x \in X_n} \langle{S_n F},{\psi_{n,x}}\rangle_{L_2} \lambda(x) K,
\]
so that~\eqref{eq:73} holds true if we define the coefficients $c(F)_{n,x}$ as
in \eqref{eq:VnAtomicDecompositionCoefficientsExplicit}.
With this notation the coefficients depend linearly on $F$.
By applying~\eqref{eq:76} we obtain the estimate
\[
  \left(
    \sum_{x\in X_n} \lvert{c(F)_{n,x}}\rvert^r m(x)^r
  \right)^{1/r}
  \leq |Q_n|^{1/r'} \cdot 
       \sup_{q \in Q_n} w(q) \cdot \lVert{S_nF}\rVert_{\mathcal{M}_{r,m}}
  \leq C_n \,\lVert{F}\rVert_{\mathcal{M}_{r,m}},
\]
where $C_n$ is as in the statement of the lemma, and where $\lVert{S_n}\rVert$ is
the operator norm of $S_n$ as an operator from $V_n$ into $L_{r,m}(G)$.
This proves \eqref{eq:VnAtomicDecompositionCoefficientNorm}.
\end{proof}

\begin{remark}\label{rem5}
Note that if the sequence
$(|Q_n|^{1/r'} \cdot \sup_{q \in Q_n} w(q) \cdot \lVert{S_n}\rVert)_{n \in \mathbb{N}}$ is
bounded, then the constant $C_n$ in
\eqref{eq:VnAtomicDecompositionCoefficientNorm} can be bounded independently
of $n$. Naturally the question arises under which conditions this really is
the case. To answer this question, it is necessary to determine the asymptotic
behaviour of the operator norm of $S_n$.
As we will show in Sect. \ref{sec:examples}
this task is already non-trivial for a very simple setting.
Still, in Lemma \ref{BB} we give a partial answer, as we present a technique
to characterize the operator-norm in a different manner.
\end{remark}

The proof of the following technical lemma can be found in the appendix.
We recall that the integer $\mathcal{I}$ is defined
  through assumption~\eqref{eq:75}.
  
\begin{lemma}\label{lem:3.13}
Let $1 \leq p \leq \infty$ and $(d_x)_{x\in Y_n}\in\ell_{p,m}(Y_n)$ for some
$n\in\mathbb{N}$, then
\begin{align*}
  \Big\|\sum_{x\in Y_n}|d_x|\chi_{xQ_n}\Big\|_{L_{p,m}}
  \leq \mathcal{I}^{1-\frac{1}{p}}
       \cdot \sup_{q\in Q_n}
               w(q) \cdot |Q_n|^{\frac{1}{p}}
               \cdot \lVert{(d_x)_{x\in Y_n}}\rVert_{\ell_{p,m}}
\end{align*}
with the convention $\frac{1}{\infty} := 0$.
\end{lemma}

With the auxiliary results above, we are in the position to state and
prove our main result.

\begin{theorem} \label{maintheorem}
We assume that $K$ satisfies \eqref{eq:KernelAlmostIntegrable} and that there exists $p<r$ such that
\begin{alignat}{2}
  &K \in L_{p,w\Delta^{-1/p}}(G), \notag\\
  &\operatorname{osc}_{Q_n}(K) \in L_{p,w}(G) \cap L_{p,w\Delta^{-1/p}}(G),
  \label{assumption}
\end{alignat}
for all $n\in\mathbb{N}$.
\begin{enumerate}[i)]
  \item Fix $\varepsilon > 0$; then for any $T \in \operatorname{Co}(L_{r,m})$ there exists
            $n^{*} = n^{*}_{T,\varepsilon} \in \mathbb{N}$ such that for all $n \geq n^{*}$
            \begin{align*}
              \Big\|\,
                T - \sum_{x \in X_n} c(T)_{n,x} \pi(x) u
              \,\Big\|_{\operatorname{Co}(L_{r,m})}
              \leq \varepsilon,
            \end{align*}
            where the family $(c(T)_{n,x})_{x \in X_n}$ satisfies
            \begin{align*}
            \| (c(T)_{n,x})_{x \in X_n}\|_{\ell_{r,m}}
            \leq C_n (1+\varepsilon) \|T\|_{\operatorname{Co}(L_{r,m})},
            \end{align*}
            with $C_n = |Q_n|^{1/r'} \cdot
            \sup_{q\in Q_n} w(q) \cdot\lVert{S_n}\rVert$,
            where $S_n$ denotes any linear right inverse to the operator
            $T_n : L_{r,m}(G) \to V_n$ defined in \eqref{T_n}.

  \item Let $n \in \mathbb{N}$, and let $d = (d_x)_{x \in Y_n} \in \ell_{q,m}(Y_n)$.
            Then $T = \sum_{x \in Y_n} d_x \pi(x)u$ is in $\operatorname{Co}(L_{r,m})$.
            Furthermore the estimate
            \begin{align*}
              \|T\|_{\operatorname{Co}(L_{r,m})}
              \leq D_n \|(d_x)_{x\in Y_n}\|_{\ell_{q,m}}
            \end{align*}
            holds, where $1/q +1/p = 1 + 1/r$, and
            \begin{align}\label{eq:D_n}
              D_n
              := |Q_n|^{\frac{1}{q} - 1}
                 \cdot \mathcal{I}^{1-\frac{1}{q}}
                               \cdot\sup_{q \in Q_n}w(q) \cdot \theta_n
            \end{align}
            with $\theta_n :=  \max\left\{
                      \|\, \operatorname{osc}_{Q_n}(K) + |K| \, \|_{L_{p,w}}, \,
                      \lVert{\, \operatorname{osc}_{Q_n}(K) + |K| \,}\rVert_{L_{p,w\Delta^{-1/p}}}
                    \right\}$.

\end{enumerate}
\end{theorem}

\begin{proof}
To prove i), choose $n^{*} = n^{*}_{F,\varepsilon}$ as in Lemma \ref{density}
with $F = V_e T \in \mathcal{M}_{r,m}$. By applying \eqref{eq:73} and
\eqref{eq:VnAtomicDecompositionCoefficientsExplicit} to $\pi_n(F) \in V_n$
we obtain the atomic decomposition
\begin{align*}
  \widetilde\pi_n(T)
  &= V_e^{-1}\pi_n(F) = V_e^{-1}
     \left(
       \sum_{x \in X_n}
         \langle{S_n\pi_n(F)},{\psi_{n,x}}\rangle_{L_2} \cdot \lambda(x) K
     \right) \\
   &= \sum_{x \in X_n}
       \langle{S_n\pi_n V_e(T)},{\psi_{n,x}}\rangle_{L_2}
       \cdot V_e^{-1} \lambda(x) K = \sum_{x \in X_n} c(T)_{n,x} \pi(x) u,
\end{align*}
where $c(T)_{n,x} = \langle{S_n \pi_n V_e(T)},{\psi_{n,x}}\rangle_{L_2}$.
Using \eqref{eq:71a} and the correspondence principle we derive
\begin{align*}
       \Big\| T - \sum_{x \in X_n} c(T)_{n,x} \pi(x) u \Big\|_{\operatorname{Co}(L_{r,m})} &=    \big\| T - \widetilde{\pi_n} (T) \big\|_{\operatorname{Co}(L_{r,m})} =    \big\| V_e T - V_e \widetilde{\pi_n} (T) \big\|_{\mathcal{M}_{r,m}} \\
       &=    \big\| F - \pi_n (F) \big\|_{\mathcal{M}_{r,m}}
  \leq \varepsilon \, .
\end{align*}
Now \eqref{eq:VnAtomicDecompositionCoefficientNorm} and \eqref{eq:71b} yield
the estimate
\begin{align*}
  \left(
    \sum_{x \in X_n} | c(T)_{n,x} |^r m(x)^r
  \right)^{\frac{1}{r}}
  &\leq C_n \| \pi_n(V_e T) \|_{\mathcal{M}_{r,m}}
   \leq C_n (1+\varepsilon) \| V_e T \|_{\mathcal{M}_{r,m}} \\& =    C_n (1+\varepsilon) \|T\|_{\operatorname{Co}(L_{r,m})}
\end{align*}
for any $n\geq n^{*}$.

\medskip{}

It remains to prove ii). In \cite[Chap. 3, p. 100]{DaDeGrLa} the following
pointwise estimate for $y \in G$ has been established:
\begin{align*}
  |\sum_{x \in  Y_{n}} d_x  \lambda(x)  K(y)|
  \leq \left(
         \sum_{x \in Y_{n}} |d_x| \frac{\chi_{x Q_n}}{| Q_n |}
       \right)
       \ast \left( {\rm osc}_{Q_n}(K) + |K| \right) (y).
\end{align*}
%
Let now $q>1$ such that $1/q+1/p=1+1/r$. By using Young's inequality,
see Proposition \ref{prop:Young}, and Lemma \ref{lem:3.13},
we obtain

\begin{align*}
  &\bigg\| \sum_{x \in Y_n} d_x \pi(x) u \bigg\|_{\operatorname{Co}(L_{r,m})}
  = \bigg\| \sum_{x \in Y_{n}} d_x \lambda(x) K \bigg\|_{L_{r,m}} \\
  &\hspace{1cm}\le \bigg\|
         \Big(
           \sum_{x \in Y_{n}} |d_x| \chi_{x Q_n}
         \Big)
         \ast
         \left(
           {\rm osc}_{Q_n}(K) + |K|
         \right)
       \bigg\|_{L_{r,m}}
       \cdot | Q_n |^{-1} \\
  &\hspace{1cm}\le \bigg\|
         \sum_{x \in Y_n}
           |d_x| \chi_{x Q_n}
       \bigg\|_{L_{q, m}}\cdot | Q_n |^{-1} \\&\hspace{3cm} \cdot\max\left\{
                               \| {\rm osc}_{Q_n}(K) + |K| \|_{L_{p, w}},
                               \lVert{\operatorname{osc}_{Q_n}(K) + |K|}\rVert_{L_{p,w\Delta^{-1/p}}}
                            \right\}
                         \\
  &\hspace{1cm}\le |Q_n|^{\frac{1}{q} - 1}\cdot \mathcal{I}^{1-\frac{1}{q}}
       \cdot \sup_{q \in Q_n} w(q)
       \cdot \lVert{(d_x)}\rVert_{\ell_{q,m}} \\&\hspace{3cm}\cdot \max\left\{
                           \| {\rm osc}_{Q_n}(K) + |K| \|_{L_{p, w}},
                             \lVert{\operatorname{osc}_{Q_n}(K) + |K|}\rVert_{L_{p,w\Delta^{-1/p}}}
                          \right\}.
\end{align*}

%
%
%
\noindent
By the assumption \eqref{assumption} the expression on the right-hand
side is finite.

\end{proof}

\begin{remark}
The coefficients $c(T)_{n,x}$, $x\in X_n$, in Theorem \ref{maintheorem} i)
depend linearly on $T$ if and only if the projection $\pi_n$
from \eqref{eq:MetricProjectionDefinition} is linear.
\end{remark}


The following proposition presents a slight variation
of Theorem \ref{maintheorem}.

\begin{proposition}
Under the same assumptions as in Theorem \ref{maintheorem} the following holds:
Fix $\varepsilon > 0$ and $T \in \operatorname{Co}(L_{r,m})$; then there exists
$n^{*} = n^{*}_{T, \varepsilon} \in \mathbb{N}$ such that for all $n \geq n^{*}$
\begin{align}\label{eq:84}
  \frac{1}{\tau_n (1 + \varepsilon)}
  \cdot \|(c(T)_{n,x})_{x \in X_n}\|_{\ell_{q,m}}
  \leq \|T\|_{\operatorname{Co}(L_{r,m})}
\end{align}
and
\begin{align} \label{eq:85}
  \lVert{T}\rVert_{\operatorname{Co}(L_{r,m})}
  &\leq \varepsilon
          + D_n \cdot \|(c(T)_{n,x})_{x\in X_n}\|_{\ell_{q,m}},
\end{align}
where $D_n$ as in \eqref{eq:D_n}
and $\tau_n :=  C_n \cdot |X_n|^{\frac{1}{q}-\frac{1}{r}}$,
$|X_n|$ is the cardinality of $X_n$ and $1/q + 1/p = 1 + 1/r$.
\end{proposition}

\begin{proof}
Throughout this proof we use the same notations as in Theorem \ref{maintheorem}.
We first note that for any finite sequence $(d_x)_{x \in X_n}$, $n \in \mathbb{N}$,
by H\"older's inequality, it holds that
\begin{align*}
        \|(d_x)_{x \in X_n}\|_{\ell_{q,m}}
  &\leq \lVert{(d_x)_{x \in X_n}}\rVert_{\ell_{r,m}}
        \cdot \lVert{1_{X_n}}\rVert_{\ell_{\frac{rq}{r-q}}},
\end{align*}
where $1_{X_n}$ is a sequence of ones only.
Furthermore it holds
\begin{align*}
  \lVert{1_{X_n}}\rVert_{\ell_{\frac{rq}{r-q}}} = |X_n|^{\frac{1}{q}-\frac{1}{r}}.
\end{align*}
With $\tau_n := C_n \cdot |X_n|^{\frac{1}{q}-\frac{1}{r}}$ we then
obtain from Theorem \ref{maintheorem} i) the estimate
\begin{align*}
  \frac{1}{\tau_n (1 + \varepsilon)}
  \lVert{(c(T)_{n,x})_{x \in X_n}}\rVert_{\ell_{q,m}}
  \leq \frac{1}{C_n (1 + \varepsilon)}
       \lVert{(c(T)_{n,x})_{x \in X_n}}\rVert_{\ell_{r,m}}
  \leq \lVert{T}\rVert_{\operatorname{Co}(L_{r,m})},
\end{align*}
which proves \eqref{eq:84}.

It remains to show the second inequality \eqref{eq:85}.
For this we note that the sequence $(c(T)_{n,x})_{x \in X_n}$ can be understood
as a sequence over the index set $Y_n$ with only finitely many non-zero entries.
Therefore, by \eqref{eq:71a} and Theorem~\ref{maintheorem}~ii), this yields
\begin{align*}
  \lVert{T}\rVert_{\operatorname{Co}(L_{r,m})}
  &\leq \varepsilon + \Bigg\|
                 \sum_{x \in X_n} c(T)_{n,x} \pi(x) u
               \Bigg\|_{\operatorname{Co}(L_{r,m})} 
               \\
               &\leq \varepsilon + |Q_n|^{\frac{1}{q} - 1}
               \cdot \theta_n
               \cdot \lVert{(c(T)_{n,x})_{x \in X_n}}\rVert_{\ell_{q,m}},
\end{align*}
which concludes the proof.
\end{proof}

\subsection{An Example: Coorbit Theory for Paley-Wiener Spaces}
\label{sec:examples}


As an example we will discuss the case of band-limited functions on the
real line. This case cannot be handled with the classical coorbit theory,
since the reproducing kernel that arises is the sinc function,
which is not integrable. Thus, the band-limited functions are a
suitable example for our setting.

We will briefly recall the setting following the lines of
Sect. 4.2 in \cite{dadedelastte17}. Let $G$ denote the additive group $\mathbb{R}$
whose Haar measure is the Lebesgue measure $dx$.
Since the group is abelian, $\mathbb{R}$ is unimodular.
We denote by $\mathbb{S}(\mathbb{R})$ the Schwartz space of smooth, rapidly decaying
functions and by $\mathbb{S}'(\mathbb{R})$ the space of tempered distributions.
The Fourier transform on $\mathbb{S}(\mathbb{R})$
    and $\mathbb{S}'(\mathbb{R})$---defined for $f \in L^1 (\mathbb{R})$
    as $\mathcal{F} f (\xi) = \widehat{f}(\xi)
        = \int_{\mathbb{R}} f(x) e^{-2\pi i x \, \xi} \, dx$---is denoted
    by $\mathcal{F}$.
If $v\in\mathbb{S}'(\mathbb{R})$ we also set $\widehat{v} = \mathcal{F} v$.

The Hilbert space $\mathcal{H}$ we are interested in is the Paley-Wiener space
of functions with band in the fixed set $\Omega \subset \mathbb{R}$, namely
\begin{align*}
  \mathcal{H}
  = B_\Omega^2
  = \left\{ v \in L_2(\mathbb{R}) ~\middle|~ \operatorname{supp}(\widehat{v}) \subseteq \Omega \right\}
\end{align*}
equipped with the $L_2(\mathbb{R})$ scalar product.
Then, by defining $\pi$ for $b \in \mathbb{R}$ as
\begin{align*}
  \pi(b) v(x) = v(x-b),
  \quad v \in B_\Omega^2,
\end{align*}
$\pi$ becomes a unitary representation of the group $\mathbb{R}$ acting on $B_\Omega^2$.
With this definition of $\pi$, on the frequency side
$\widehat{\pi} = \mathcal{F} \pi \mathcal{F}^{-1}$ acts on
$\mathcal{F} \mathcal{H} = L_2(\Omega)$ by modulations:
\begin{align*}
  \widehat{\pi}(b) \widehat{v}(\xi)
  = e^{2\pi ib\xi} \widehat{v}(\xi),
  \quad v \in B_\Omega^2.
\end{align*}

From now on we set $\Omega$ to be a symmetrical interval,
$\Omega=[-\omega,\omega]$.
Proposition 4.6 in \cite{dadedelastte17} then shows that by choosing as
admissible vector the function $u = \mathcal{F}^{-1} \chi_\Omega \in B_\Omega^2$,
the resulting kernel $K$ as defined in \eqref{eq:KernelDefinition} is the
sinc function
\begin{equation}
  K(b) = \mathcal{F}^{-1} \chi_\Omega(b)
       = 2\omega\,\mathrm{sinc}(2\omega\pi b)
       = \frac{\sin(2\omega\pi b)}{\pi b},
  \label{eq:PaleyWienerKernelExplicit}
\end{equation}
where $\mathrm{sinc}~x = \sin\,x/x$.
Clearly, $K$ is not in $L_1(\mathbb{R})$, but it belongs to $L_p(\mathbb{R})$ for every $p > 1$.
Therefore we choose the weight $w = 1$ and take
\begin{align*}
  \mathcal{T} = \bigcap_{1 < p < \infty} L_p(\mathbb{R})
\end{align*}
as a target space to construct coorbits, see \eqref{basicsetting}.
As above, the (anti)-dual of $\mathcal{T}$ can be identified with
\begin{align*}
  \mathcal{T}'
  = \mathcal{U}
  = \operatorname{span}\bigcup_{1 < q < \infty} L_q(\mathbb{R}).
\end{align*}

For $p \in [1, \infty)$, we define the Paley-Wiener $p$-spaces
\begin{align*}
  B_\Omega^p
  := \left\{f \in L_p(\mathbb{R}) ~\middle|~ \operatorname{supp}(\mathcal{F} f) \subseteq \Omega \right\}.
\end{align*}
Recall that the Fourier transform maps $L_p(\mathbb{R})$ to $L_{p'}(\mathbb{R})$ for $p\leq 2$,
which follows from the Hausdorff-Young inequality.
In contrast, for $p > 2$ the space $\mathcal{F} L_p(\mathbb{R})$ contains distributions that
in general are not functions, see \cite[Theorem 7.6.6]{Ho03}.

The spaces $B_\Omega^{p}$ are sometimes defined in the literature as
the spaces of the entire functions of fixed exponential type whose restriction
to the real line is in $L_p(\mathbb{R})$. This definition is equivalent to ours since
a Paley-Wiener theorem holds for all $p \in [1, \infty)$.
In particular, all these functions are infinitely differentiable on $\mathbb{R}$.
Moreover, if $f \in B_\Omega^p$ with $p < \infty$, then $f(x) \to 0$ as
$x \to \pm \infty$, and hence
\begin{align*}
  B_\Omega^p\subset \mathcal{C}_0^\infty(\mathbb{R})
  = \left\{
      f \in \mathcal{C}^{\infty}(\mathbb{R})
      ~\middle|~
      f(x)\to 0 \mbox{ as } x \to \pm \infty
    \right\}
  , \quad 1 \leq p <\infty.
\end{align*}
Consequently, the Paley-Wiener spaces are nested and increase with $p$:
\begin{align*}
  B_\Omega^p \subseteq B_\Omega^q,
  \quad 1 \leq p \leq q < \infty.
\end{align*}


\begin{proposition}[Proposition 4.8 of \cite{dadedelastte17}]
Let $\Omega = [-\omega,\omega]$ and define $u := K := \mathcal{F}^{-1} \chi_\Omega$.
The ``test space'' (as defined in \eqref{eq:43}) is
\begin{align*}
  \mathcal{S} = \bigcap_{p \in (1, \infty)} B_\Omega^p
\end{align*}
and its dual space is
\begin{align*}
  \mathcal{S}' = \bigcup_{p \in (1, \infty)} B_\Omega^p.
\end{align*}
The extended voice transform is the inclusion
\begin{align*}
  V_e : \mathcal{S}' \hookrightarrow \mathcal{U}
\end{align*}
and the following identification holds:
\begin{align*}
  \operatorname{Co}(L_p(\mathbb{R})) = \mathcal{M}^p = B_\Omega^p.
\end{align*}
\end{proposition}

To obtain a discretization as laid out in Sect. \ref{sec:discretization},
we first need to show that Assumption \ref{assumption1} is fulfilled.
By Lemma \ref{lem:kernel_continuity} it suffices to show that the convolution
operator associated to $K$ is a bounded operator on $L_{p}(\mathbb{R})$.

\begin{corollary}
  Let $1 < p < \infty$, then $RC_K$ is a bounded
  operator on $L_p(\mathbb{R})$.
\end{corollary}

\begin{proof}
Since $K = \mathcal{F}^{-1} \chi_\Omega$, the convolution with $K$ is a bounded
operator on $L_p(\mathbb{R})$ if and only if $\chi_\Omega$ is a Fourier multiplier
on $L_p(\mathbb{R})$. By \cite[Example 2.5.15]{Gra04} this is true if and only if
$\chi_{[0,1]}$ is a Fourier multiplier on $L_p(\mathbb{R})$.
However, it is well-known that this is true because the Hilbert transform is
bounded as an operator acting on $L_p(\mathbb{R})$, see \cite[Theorem 5.1.7]{Gra04}.
\end{proof}

We will now apply the analysis outlined in Sect. \ref{atomicdecomp} to obtain
a discretization for these spaces.
To this end, for $n \in \mathbb{N}$, let
\begin{align}\label{ex:Y_n}
  Y_n := \left\{ 2^{-n}k \right\}_{k \in \mathbb{Z}} \subset \mathbb{R}.
\end{align}
Furthermore we fix
\begin{align} \label{ex:Q_n}
  Q_n = \left[-2^{-n-1},2^{-n-1}\right],
\end{align}
which is a compact neighborhood of zero, and we set
\begin{align}\label{ex:psi}
  \psi_{n,k} := \chi_{[-2^{-n-1},2^{-n-1})}(\cdot -2^{-n}k)
\end{align}
for $n \in \mathbb{N}$ and $k \in \mathbb{Z}$, where $\chi$ denotes the characteristic function.
Then, the verification of \eqref{eq:70a}--\eqref{eq:70c} is straightforward.
For later use we note that $|Q_n| = 2^{-n}$.
Furthermore the system $\{\psi_{n,k}\}_{k \in \mathbb{Z}}$, $n \in \mathbb{N}$ fixed,
is orthogonal with $\lVert{\psi_{n,k}}\rVert_{L_2(\mathbb{R})}^2 = |Q_n|$.
As a finite subset of $Y_n$, $n \in \mathbb{N}$,  we set
\begin{align} \label{ex:X_n}
  X_n := \left\{ 2^{-n} k ~\middle|~ -N(n) \leq k \leq N(n) \right\},
\end{align}
where $N(n) \in \mathbb{N}$ is chosen such that \eqref{eq:68} and \eqref{eq:69}
are fulfilled. A possible choice is $N = N(n) = n \cdot 2^n$.

According to \eqref{T_n} the operator
$T_n : B_\Omega^p \to V_n \subset B_\Omega^p$ is defined via
\begin{align*}
  T_n f(x) = \sum_{k = -N(n)}^{N(n)} \langle{f},{\psi_{n,k}}\rangle_{L_2} K(x - 2^{-n} k),
\end{align*}
for $f \in B_\Omega^p$, where
\begin{align*}
  \langle{f},{\psi_{n,k}}\rangle_{L_2}
  = \int_{2^{-n}(k - 1/2)}^{2^{-n}(k + 1/2)} f(y) ~dy.
\end{align*}
By \eqref{eq:70} this means
\begin{align*}
  V_n
  = \operatorname{span}
    \left\{
      \mathrm{sinc}(2\pi\omega(\cdot - 2^{-n} k))
      ~\middle|~
      -N(n) \leq k \leq N(n)
    \right\}.
\end{align*}

In order to apply Theorem \ref{maintheorem} we first need to show the following:

\begin{lemma}\label{ex:aux}
It holds  $K\in\mathcal{M}^\rho_{Q_n}(L_r)$,
and therefore $\operatorname{osc}_{Q_n} f \in L_r(\mathbb{R})$ for all $n \in \mathbb{N}$.
\end{lemma}


\begin{proof}
  We have $| K(y) | \leq 2\omega$ for all $y \in \mathbb{R}$, and
  $| K(b) | \leq 1 / (\pi |b|)$ for all $b \neq 0$.
  This implies
  \[
    | K(y) | \leq \frac{1+4\omega}{1 + |y|}.
  \]
  Indeed, for $|y| \leq 1$ we have $|K(y)| \leq 2\omega
  \leq \frac{4\omega}{1+|y|}$, while for $| y | \geq 1$, we have
  $1/|y| \leq 2/(1+|y|)$, and thus
  $| K(y) | \leq \frac{2}{\pi} \frac{1}{| y |} \leq \frac{1+4\omega}{1+|y|}$.

  Now, for $y \in x + Q_n \subset x + [-1,1]$ we have
  $1 + |x| \leq 2 + |y| \leq 2 (1+|y|)$, so that
  $|K(y)| \leq \frac{1+4\omega}{1+|y|} \leq \frac{2+8\omega}{1+|x|}$.
  Hence,
  \[
    \quad\quad
    \sup_{y \in x + Q_n} | K(y) |^r
    \leq \left( \frac{2+8\omega}{1+|x|} \right)^r \, ,
    \]
  and thus
  \[
    \int_{\mathbb{R}} \,\, \sup_{y \in x + Q_n} | K(y) |^r \, dx < \infty \, ,
  \]
  concluding the proof.
\end{proof}

With this at hand we can discretize the Paley-Wiener $p$-spaces according
to Theorem \ref{maintheorem}.

\begin{proposition}\label{ex:prop}
Let $1<p<\infty$.
\begin{enumerate}[i)]
  \item Fix $\varepsilon > 0$; then for any $f \in B^p_\Omega$ there exists
            an integer $n^{*} =n^{*}_{f,\varepsilon} \in \mathbb{N}$, such that for all
            $n\geq n^{*}$
            \begin{align*}
              \bigg\|
                f - \sum_{k = -N(n)}^{N(n)} c(f)_{n,k} K(\cdot - 2^{-n} k)
              \bigg\|_{L_p}
              \leq \varepsilon,
            \end{align*}
            where the family of coefficients
            $(c(f)_{n,k})_{-N(n) \leq k \leq N(n)}$ satisfies
            \begin{align*}
              \lVert{(c(f)_{n,k})_{-N(n) \leq k \leq N(n)}}\rVert_{\ell_p}
              \leq 2^{-n/p'} (1 + \varepsilon) \lVert{S_n}\rVert \cdot \|f\|_{L_p}.
            \end{align*}
            Here, as usual, $S_n$ is a linear right inverse for the operator
            $T_n$ defined in \eqref{T_n}.

  \item For any sequence $(d_x)_{x \in Y_n}\in\ell_q(Y_n)$, $n \in \mathbb{N}$, the function $f$ defined by    \linebreak       $f = \sum_{k \in \mathbb{Z}} d_{2^{-n}k} K(\cdot - 2^{-n} k)$
            is in $B_\Omega^p$ with
            \begin{align*}
              \lVert{f}\rVert_{L_p(\mathbb{R})}
              \leq C \cdot 2^{n(1 - 1/q)} \lVert{(d_x)_{x \in Y_n}}\rVert_{\ell_q},
            \end{align*}
            where $C = C(p,q) > 0$ is a constant and $q < p$.
\end{enumerate}
\end{proposition}

\begin{proof}
i) is an application of Theorem \ref{maintheorem} i), with $|Q_n| = 2^{-n}$.

\medskip{}

It remains to prove ii).
Again, we can apply Theorem \ref{maintheorem} ii) and note that,
by Lemma \ref{ex:aux}, the assumption \eqref{assumption} is fulfilled.
Moreover, Lemma \ref{ex:aux} shows that $\lVert{\operatorname{osc}_{Q_n}(K)+|K|}\rVert_{L_r}$
can estimated from above by a constant $C>0$ independent of $n \in \mathbb{N}$.
\end{proof}

As stated in Remark \ref{rem5} the asymptotic behaviour of the operator norm
of $S_n$ is crucial. In the following we apply Lemma \ref{BB} to obtain a useful
characterization of $\lVert{S_n}\rVert$.

For this we restrict ourselves to the case $p = 2$ and obtain with the
notation of Lemma \ref{BB}
\begin{align*}
  \lVert{S_n}\rVert^{-1}
  &= \varepsilon
   = \inf\left\{
            \frac{\lVert{T_n f}\rVert_{L_2}}{\lVert{f}\rVert_{L_2}}
            ~\middle|~
            f\in (\operatorname{Ker} T_n)^\bot
          \right\} \\
          &= \inf\left\{
           \frac{\langle{T_n^* T_n f},{f}\rangle_{L_2}}{\langle{f},{f}\rangle_{L_2}}
           ~\middle|~
           f\in (\operatorname{Ker} T_n)^\bot
         \right\}^{1/2} = \lambda_{\min}(U_n)^{1/2},
\end{align*}
where $\lambda_{\min}(U_n)$ denotes the smallest eigenvalue of the operator
\[
  U_n := T_n^* T_n
      : (\operatorname{Ker} T_n)^\bot \to (\operatorname{Ker} T_n)^\bot \, .
\]
Here, we used the well-known inclusion
$\operatorname{Ran} A^\ast \subset (\operatorname{Ker} A)^\perp$
which guarantees that $U_n$ is well-defined.

We have thus shown that the asymptotic behaviour of the smallest eigenvalue of
$U_n$ is equivalent to the asymptotic behaviour of $\lVert{S_n}\rVert$.

By using
\begin{align*}
  T_n f &= \sum_{j = -N(n)}^{N(n)}
              \langle{f},{\psi_{n,j}}\rangle_{L_2} K(\cdot - 2^{-n} j),
  \quad
  T_n^\ast g = \sum_{k = -N(n)}^{N(n)}
                 \langle{g},{K(\cdot - 2^{-n} k)}\rangle_{L_2} \psi_{n,k},
\end{align*}
we can rewrite $U_n$ as
\begin{align*}
  U_n f
  &= \sum_{j,k = -N(n)}^{N(n)}
       \langle{f},{\psi_{n,j}}\rangle_{L_2(\mathbb{R})}
       \langle{K(\cdot - 2^{-n} j)},{K(\cdot - 2^{-n} k)}\rangle_{L_2(\mathbb{R})}
       \psi_{n,k} \\
       &= \sum_{j,k = -N(n)}^{N(n)}
        \langle{f},{\psi_{n,j}}\rangle_{L_2(\mathbb{R})}
        K(2^{-n}(k-j))
        \psi_{n,k}
\end{align*}
for $f \in (\operatorname{Ker} T_n)^\bot$.

We set
$W_n := \operatorname{span} \left\{\psi_{n,k} ~\middle|~ -N(n) \leq k \leq N(n) \right\}$
and obtain the relation $W_n^\bot\subset\operatorname{Ker}T_n$;
thus $(\operatorname{Ker}T_n)^\bot\subset W_n$.
Next, we note that the family $\{\lambda(x) K\}_{x \in \mathbb{R}}$ is linearly
independent; indeed, we have $\mathcal{F} (\lambda(x) K)
= e^{-2\pi i x \cdot} \chi_{[-\omega,\omega]}$,
and by analyticity these functions are linearly independent if and only if
the functions $( \mathbb{R} \to \mathbb{C}, \xi \mapsto e^{-2\pi i x \xi})$ are. But
each of these functions is an eigenvector of the differential operator
$d/d\xi$ with pairwise distinct eigenvalues $2\pi i x, x \in \mathbb{R}$,
which yields the linear independence.
From this and from Lemma \ref{lem:VnCharacterization}, we see that
$\operatorname{Ran} T_n = V_n
= \operatorname{span} \left\{\lambda(x) K\right\}_{x \in X_n}$ satisfies
$\operatorname{dim} \operatorname{Ran} T_n = |X_n| = 1 + 2N(n)$.
But since $T_n : (\operatorname{Ker} T_n)^{\perp} \to V_n$ is an isomorphism,
we see $\operatorname{dim} (\operatorname{Ker} T_n)^{\perp} = 1 + 2N(n)$
as well, so that we finally see $W_n = (\operatorname{Ker} T_n)^{\perp}$
by comparing dimensions. Hence, $U_n : W_n \to W_n$.

Moreover, by the orthogonality of the family $\{\psi_{n,k}\}$, we see that
\begin{align} \label{eq:81}
  U_n \, \psi_{n,k}
  = \lVert{\psi_{n,k}}\rVert^2_{L_2}
    \sum_{\ell = -N(n)}^{N(n)} K(2^{-n}(\ell - k)) \, \psi_{n,\ell}
\end{align}
for any $-N(n) \leq k \leq N(n)$. 
Since $\operatorname{dim}W_n = 2N(n) + 1 < \infty$, we may define an isomorphism
\begin{align}\label{eq:82}
  P_n : W_n \to \mathbb{R}^{2N(n)+1}, \quad
        P_n(\psi_{n,k}) = \lVert{\psi_{n,k}}\rVert_{L_2} e_k,
\end{align}
where $e_k$ denotes the $k$-th canonical unit vector of $\mathbb{R}^{2N(n)+1}$.
Note that $P_n$ maps the orthonormal basis $(\psi_{n,k}/\| \psi_{n,k} \|_{L_2(\mathbb{R})})$
to the orthonormal basis $(e_k)_{k \in \mathbb{N}}$, so that $P_n$ is unitary.

The linear map $P_n U_n P_n^{-1} : \mathbb{R}^{2N(n)+1} \to \mathbb{R}^{2N(n)+1}$ is represented
by a matrix $M_n$, whose entries are given via
\begin{align*}
  (M_n)_{j,k}
  &= \langle{P_n U_n P_n^{-1} e_k},{e_j}\rangle_{\mathbb{R}^{2N(n)+1}}
   = \langle{U_n \frac{\psi_{n,k}}{\lVert{\psi_{n,k}}\rVert_{L_2}}},
          {\frac{\psi_{n,j}}{\lVert{\psi_{n,j}}\rVert_{L_2(\mathbb{R})}}}\rangle_{L_2} \\
  &= \frac{\lVert{\psi_{n,k}}\rVert_{L_2}}{\lVert{\psi_{n,j}}\rVert_{L_2}}
     \sum_{\ell = -N(n)}^{N(n)}
       K(2^{-n}(\ell - k))
       \langle{\psi_{n,\ell}},{\psi_{n,j}}\rangle_{L_2} \\
       &= \lVert{\psi_{n,k}}\rVert_{L_2} \lVert{\psi_{n,j}}\rVert_{L_2} K(2^{-n}(j - k)) \\
  &= 2^{-n} K(2^{-n}(j - k)),
\end{align*}
$1 \leq j,k \leq 2N(n)+1$.
Since $K$ is real, the matrix $M_n$ is a \emph{symmetric Toeplitz matrix},
which means that the entries of $M_n$ only depend on the quantity $|k-j|$,
thus yielding a band-structure.
Since the eigenvalues of $M_n$ coincide with those of the map $U_n$,
finding the smallest eigenvalue of $U_n$ is equivalent to finding the smallest
eigenvalue of the Toeplitz matrix $M_n$.

Unfortunately, this task is very difficult. To the best knowledge of the
authors, it is not possible to properly characterize the asymptotic behaviour
of the smallest eigenvalue of such a Toeplitz matrix.
We further refer to \cite{boettcher}, where the authors were told that
leading experts on the field of Toeplitz matrices are unaware of these
asymptotics.


Since there are already big obstacles in understanding the asymptotic behaviour
of $\lVert{S_n}\rVert$ in this rather simple setting, one cannot hope that easy answers
are available when turning to more complex groups and their associated
coorbit spaces.

\section{Obstructions to Discretization for Non-Integrable Kernels}
\label{sec:Obstructions}


In classical coorbit theory, the kernel
$K(x) = V u (x) = \langle u, \pi(x) \, u \rangle_{\mathcal{H}}$ is assumed to be
integrable; in other words, it has to satisfy $K \in L_{1,w}(G)$ for a suitable
weight $w \geq 1$ on $G$.
This assumption is introduced in order to guarantee two independent properties:
First, it ensures that one can construct a suitable reservoir
of ``distributions,'' and thus obtains well-defined coorbit spaces.
Second, it ensures that the right convolution operator
$f \mapsto f \ast K$ acts boundedly on the function space $Y$ which is used to
define the coorbit space $\operatorname{Co}(Y)$. For instance, this is the case if $Y = L_{r,m}(G)$
with a $w$-moderate weight $m$.

Replacing the integrability condition $K \in L_{1,w}(G)$
by the weaker assumption $K \in \bigcap_{1 < p < \infty} L_{p,w}(G)$,
one can still define a suitable reservoir and obtains well-defined decomposition
spaces, as we saw in Sect. \ref{sec:overview}.
However, we will see in the present section---precisely, in
Proposition \ref{prop:BadPaleyWienerSet}---that the modified
assumption $K \in \bigcap_{1 < p < \infty} L_{p,w}(G)$ is in general too weak
to ensure that right convolution with $K$ defines a bounded operator on
$L_{r,m}(G)$.
In other words, a given kernel $K$ satisfying the weak integrability assumption
might or might not act boundedly on $L_{r,m}(G)$ by right convolution.

For such ``bad'' kernels that do not act boundedly,
no discretization results similar to those from
classical coorbit theory can hold,
as we will prove in the present section.
Therefore, if such discretization results for the coorbit space $\operatorname{Co}(L_{r,m})$
are desired, one needs to assume that
$K \in \bigcap_{1 < p < \infty} L_{p,w}(G)$ and additionally that
$f \mapsto f \ast K$ defines a bounded operator on $L_{r,m}(G)$.
This second condition is highly nontrivial to verify in many cases where
the kernel $K$ is not integrable. However, it is possible in the setting of the group $(\mathbb{R},+)$ as discussed in Sect. \ref{sec:examples} above.

Since we aim to show that no discretization as for classical coorbit theory
is possible, we briefly recall these results:
Assuming the kernel $K$ to be well behaved, a combination
of Lemma 3.5~v) and Theorem 6.1 in \cite{fegr89a} shows that
the \emph{synthesis operator}
\[
  \operatorname{Synth}_{X} : \ell_{r,m_X}(I) \to \operatorname{Co}(L_{r,m}),
               (c_i)_{i \in I} \mapsto \sum_{i \in I} c_i \cdot \pi(x_i) \, u
  \quad \text{with} \quad
  (m_X)_i = m(x_i)
\]
is well-defined and bounded, for each $r \in (1,\infty)$,
each $w$-moderate weight $m$, and each family $X = (x_i)_{i \in I}$ in $G$ that
is sufficiently \emph{separated}---similar to $\delta \mathbb{Z}^d$ in $G=\mathbb{R}^d$.
The operator $\operatorname{Synth}_X$ even has a bounded linear right inverse,
provided that the family $X$ is sufficiently dense in $G$, where the required
density only depends on $w,u$.
If $\operatorname{Synth}_X$ indeed has a bounded linear right inverse,
the family $(\pi(x_i) \, u)_{i \in I}$ is called a
family of atoms for $\operatorname{Co}(L_{r,m})$ with coefficient space
$\ell_{r,m_X}(I)$.

Dual to the concept of atomic decompositions is the notion of
\emph{Banach frames}, which was introduced in \cite{gro91}.
By definition, the family $(\pi(x_i) \, u)_{i \in I}$ is a Banach
frame for $\operatorname{Co}(L_{r,m})$ with coefficient space $\ell_{r,m_X}(I)$
if the \emph{analysis operator}
\[
  A_X : \operatorname{Co}(L_{r,m}) \to \ell_{r,m_X}(I),
        f \mapsto \big(
                    \langle f, \pi(x_i) \, u \rangle_{\mathcal{S}_w}
                  \big)_{i \in I}
\]
is well-defined and bounded and has a bounded linear left inverse.
As shown in \cite[Theorem 5.3]{gro91},
this is satisfied if the sampling points $X = (x_i)_{i \in I}$ satisfy
the same properties as above: they should be sufficiently separated and
dense enough in $G$, where these conditions only depend on $w$ and $u$,
but not on the integrability exponent $r$ or the $w$-moderate weight $m$.
Provided that $(\pi(x_i) \, u)_{i \in I}$ is a Banach frame for $\operatorname{Co}(L_{r,m})$,
we have in particular $\|A_X f\|_{\ell_{r,m_X}} \asymp \|f\|_{\operatorname{Co}(L_{r,m})}$
for all $f \in \operatorname{Co}(L_{r,m})$; but in general this latter property is weaker
than the Banach frame property.

The preceding statements hold for all $w$-moderate
weights $m$ and for all exponents $r \in (1,\infty)$.
Since the reciprocal $m^{-1}$ of a $w$-moderate weight $m$ is
again $w$-moderate, see Lemma~\ref{lem:1/m_w-moderate},
it follows that if the above properties hold for
$L_{r,m}(G)$, then they also hold for $L_{r', m^{-1}}(G)$.
Therefore, classical coorbit theory provides discretization results that are
stronger than the assumptions of the following theorem.
The following theorem thus shows that discretization results as in classical coorbit
theory can only hold if the kernel $K$ acts boundedly on $L_{r,m}(G)$
via right convolutions.

\begin{theorem}\label{thm:DiscretizationObstruction}
  Let $r \in (1,\infty)$ be arbitrary.
  Assume that Assumption \ref{assume:KernelAlmostIntegrable} is satisfied,
  and let $m : G \to (0,\infty)$ be a $w$-moderate weight.
  Furthermore, assume that for some family $(x_{i})_{i \in I}$ in $G$ and for
  some weight $\theta = (\theta_{i})_{i \in I}$ on the index set $I$,
  the following hold:

  \begin{enumerate}[i)]
    \item ``Weak Banach frame condition for $\operatorname{Co}(L_{r,m})$'':
          The analysis map
          \[
            A : \operatorname{Co}(L_{r,m}) \to \ell_{r,\theta}(I),
                \varphi \mapsto \big(
                                  \left\langle
                                    \varphi \, , \, \pi(x_{i}) \, u
                                  \right\rangle_{\mathcal{S}_w}
                                \big)_{i\in I}
          \]
          is well-defined and bounded, with
          \begin{equation}
            \left\Vert A \, \varphi \right\Vert_{\ell_{r,\theta}}
            \asymp \left\Vert \varphi \right\Vert_{\operatorname{Co}(L_{r,m})}
            \qquad \text{for all} \: \varphi \in \operatorname{Co}(L_{r,m}) \, .
            \label{eq:WeakBanachFrameCondition}
          \end{equation}

    \item ``Weak atomic decomposition condition for
          $\operatorname{Co}(L_{r',m^{-1}})$'':
          The synthesis map
          \[
            S : \ell_{r', \theta^{-1}} (I) \to \operatorname{Co}(L_{r',m^{-1}}) ,
                \left(c_{i}\right)_{i \in I}
                \mapsto \sum_{i \in I}
                          \left[c_{i}\cdot\pi(x_i) \, u \right]
          \]
          is well-defined and bounded.
  \end{enumerate}
  Then the right convolution operator $RC_K : f \mapsto f \ast K$
  defines a bounded linear operator on $L_{r,m}(G)$.
\end{theorem}

For the proof of this theorem, we will need several technical lemmata.
Having shown in Sect. \ref{sec:overview} that the voice transform
can be extended from $\mathcal{H}$ to the reservoir $\mathcal{S}_w '$
(and thus to the coorbit spaces $\operatorname{Co}(L_{r,m})$),
our first lemma shows that one can also define a version of the voice transform
on the (anti)-dual space $\left[\operatorname{Co}(L_{r,m})\right]'$.

\begin{lemma}\label{lem:VoiceTransformOnCoorbitAntiDual}
  If Assumption \ref{assume:KernelAlmostIntegrable} is satisfied for $r\in(1,\infty)$, and if
  $m : G \to (0,\infty)$ is $w$-moderate, then there is a constant
  $C = C(m,r,w,K) > 0$ such that
  \[
    \text{for all } x \in G : \quad
    \pi(x) \, u \in \operatorname{Co}(L_{r,m})
    \quad \text{and} \quad
    \left\Vert \pi(x) \, u \right\Vert_{\operatorname{Co}(L_{r,m})}
    \leq C \cdot w(x) \, . 
  \]

  Therefore, for any (antilinear) continuous functional
  $\varphi \in \left[ \operatorname{Co}(L_{r,m}) \right]'$,
  the \emph{special voice transform}
  \[
    V_{\mathrm{sp}} \,\varphi
    : G \to     \mathbb{C},
    x \mapsto \varphi(\pi(x) \, u)
              = \left\langle
                  \varphi \, , \, \pi(x) u
                \right\rangle_{[\operatorname{Co}(L_{r,m})]'\times\operatorname{Co}(L_{r,m})}
  \]
  is a well-defined function.
\end{lemma}

\begin{proof}
First, let us set $C_1 := m(e)$, where $e$ is the unit element of $G$.
Since $m$ is $w$-moderate (see \eqref{eq:ModerateWeight}), we have
\[
  m(x) = m(x \cdot e) \leq w(x) \cdot m(e) \leq C_1 \cdot w(x)
  \qquad \text{for all} \, x \in G \, .
\]
Furthermore,
\[
  w (y)
  = w (x x^{-1} y)
  \leq w(x) \cdot w(x^{-1} y)
  = w(x) \cdot (\lambda(x) w) (y) \, .
\]
Now, recall from Sect. \ref{sec:overview} the embedding
$\mathcal{H} \hookrightarrow \mathcal{S}_w '$, and that the extended voice transform
$V_e$ coincides with the usual voice transform on $\mathcal{H}$.
Therefore, since $\pi(x) \, u \in \mathcal{H}$, and since $K = V u$, we get
\begin{align*}
    \left\Vert V_{e} \left[\pi(x) u \right]\right\Vert_{L_{r,m}}
  & = \left\Vert V \left[\pi(x) u \right] \right\Vert_{L_{r,m}}
    \leq C_1 \cdot \left\Vert V \left[\pi(x) u \right]\right\Vert_{L_{r,w}} \\
  & = C_1 \cdot \left \Vert
                  w \cdot \lambda(x) \left[ V u \right]
                \right\Vert_{L_{r}}
    \leq C_1 \cdot w(x) \cdot
         \left\Vert
           \lambda(x) \left[w \cdot V u \right]
         \right\Vert_{L_{r}} \\
  & = C_1 \cdot w(x) \cdot \left\Vert w \cdot V u \right\Vert_{L_{r}}
    = C_1 \cdot w(x) \cdot \left\Vert K \right\Vert_{L_{r,w}}
    = C \cdot w(x) \, ,
\end{align*}
where $C := C_1 \cdot \|K\|_{L_{r,w}}$ is finite thanks to
Assumption \ref{assume:KernelAlmostIntegrable}.
This proves the first part of the lemma, which then trivially implies
that $V_{\mathrm{sp}} \, \varphi$ is a well-defined function, for any
$\varphi \in [\operatorname{Co}(L_{r,m})]'$.
\end{proof}

Our next lemma shows that if the right convolution with $K$ does \emph{not}
act boundedly on $L_{r', m^{-1}}(G)$, then
there exist certain pathological functionals on $\operatorname{Co}(L_{r,m})$.

\begin{lemma}\label{lem:NonBoundednessYieldsBadFunctionals}
  Assume that Assumption \ref{assume:KernelAlmostIntegrable} is satisfied,
  and let $r \in (1,\infty)$.
  If the right convolution operator $RC_K : f \mapsto f \ast K$
  does \emph{not} yield a well-defined bounded linear operator
  on $L_{r',m^{-1}}(G)$, then there is an (antilinear) continuous functional
  $\varphi \in \left[\operatorname{Co}(L_{r,m})\right]'$
  satisfying $V_{\mathrm{sp}} \, \varphi \notin L_{r', m^{-1}}(G)$.
\end{lemma}

\begin{proof}
  We first claim that there is some $\Phi \in L_{r', m^{-1}}(G)$
  with $\Phi \ast K \notin L_{r', m^{-1}}(G)$; that is, we claim that
  $RC_K : L_{r', m^{-1}}(G) \to L_{r', m^{-1}}(G)$ is not well-defined.

  To see this, recall from Assumption \ref{assume:KernelAlmostIntegrable}
  that $K \in \bigcap_{1<p<\infty} L_{p,w}(G)$.
  Thus, since $m^{-1}$ is $w$-moderate (see Lemma \ref{lem:1/m_w-moderate}), Young's inequality
  (see Proposition~\ref{prop:Young}) shows that the
  right convolution operator $RC_K$ is bounded as a map
  $RC_K : L_{r', m^{-1}}(G) \to L_{q,m^{-1}}(G)$ for any $q \in (r',\infty)$.
  Therefore, if $RC_K : L_{r',m^{-1}}(G) \to L_{r',m^{-1}}(G)$
  \emph{was} well-defined, then the closed graph theorem
  would imply that $RC_K : L_{r',m^{-1}}(G) \to L_{r',m^{-1}}(G)$
  is bounded, contradicting our assumptions.
  Hence, there is a function $\Phi$ as desired.

  Now, define the antilinear functional
  \[
    \varphi : \operatorname{Co}(L_{r,m}) \to \mathbb{C}, \quad
                    f            \mapsto \int_{G}
                                           \Phi(y) \cdot \overline{V_{e} f (y)}
                                         \, dy.
  \]
  It is easy to see that $\varphi$ is well-defined and bounded;
  in fact,
  \[
    \left| \varphi (f) \right|
    \leq \left\Vert \Phi \right\Vert_{L_{r', m^{-1}}}
         \cdot \left\Vert V_{e} f \right\Vert_{L_{r,m}}
    = \left\Vert \Phi \right\Vert_{L_{r', m^{-1}}}
      \cdot \left\Vert f \right\Vert_{\operatorname{Co}(L_{r,m})}.
  \]
  Finally, note for all $x \in G$ that
  \begin{align*}
    V_{\mathrm{sp}} \, \varphi(x)
    & = \left\langle
          \varphi \,,\, \pi(x) \, u
        \right\rangle_{\left[\operatorname{Co}(L_{r,m})\right]'\times\operatorname{Co}(L_{r,m})}
    = \int_{G}
        \Phi(y) \cdot \overline{V_{e} \left[\pi(x) \, u \right](y)}
      \,dy\\
    &= \int_{G}
          \Phi(y) \cdot
          \overline{
                     \langle \pi(x) u \,,\, \pi(y) \, u \rangle_{\mathcal{H}}
                   }
        \,dy 
     = \int_{G}
          \Phi(y) \cdot
          \langle u \,,\, \pi(y^{-1} x) \, u \rangle_{\mathcal{H}}
        \,dy \\
     &= \int_{G} \Phi(y) \cdot K(y^{-1}x) \,dy
      = (\Phi \ast K) (x)
  \end{align*}
  with $\Phi \ast K \in L_{q,m^{-1}}(G)$ for all $q \in (r', \infty)$.
  But by our choice of $\Phi$, we have
  $V_{\mathrm{sp}} \, \varphi = \Phi \ast K \notin L_{r', m^{-1}}(G)$,
  as desired.
\end{proof}

Our next lemma shows that the assumptions of
Theorem \ref{thm:DiscretizationObstruction} exclude the existence of
pathological functionals as in the preceding lemma.

\begin{lemma}\label{lem:DiscretizationExcludesBadFunctionals}
  Under the assumptions of Theorem \ref{thm:DiscretizationObstruction}
  and with notation as in Lemma \ref{lem:VoiceTransformOnCoorbitAntiDual}, every
  antilinear continuous functional $\varphi \in \left[\operatorname{Co}(L_{r,m})\right]'$
  satisfies \linebreak $V_{\mathrm{sp}} \, \varphi \in L_{r', m^{-1}}(G)$.
\end{lemma}

\begin{proof}
  Let $\varphi \in \left[\operatorname{Co}(L_{r,m})\right]'$ be arbitrary, and
  let the analysis operator $A$ be as in the assumptions of Theorem
  \ref{thm:DiscretizationObstruction}.
  Using this operator, we define the (antilinear) functional
  \[
    \Lambda_{0} : A\left(\operatorname{Co}(L_{r,m})\right) \to \mathbb{C},
                  Af                         \mapsto \varphi(f) \, .
  \]
  Note that this is well-defined, since
  \eqref{eq:WeakBanachFrameCondition} ensures that $A$ is injective.
  Furthermore, with $A\left(\operatorname{Co}(L_{r,m})\right)$ considered as a subspace of
  $\ell_{r,\theta}(I)$, the functional $\Lambda_{0}$
  is bounded, since \eqref{eq:WeakBanachFrameCondition} yields a
  constant $C > 0$ such that each $c = Af\in A \left(\operatorname{Co}(L_{r,m})\right)$
  satisfies
  \begin{align*}
    \left| \Lambda_{0} (c) \right|
    &= \left| \varphi(f) \right|
      \leq \|\varphi\|_{[\operatorname{Co}(L_{r,m})]'}
           \cdot \left\Vert f \right\Vert_{\operatorname{Co}(L_{r,m})}\\
    &\leq C \|\varphi\|_{[\operatorname{Co}(L_{r,m})]'}
           \cdot \left\Vert Af \right\Vert_{\ell_{r,\theta}}
      = C \|\varphi\|_{[\operatorname{Co}(L_{r,m})]'}
        \cdot \left\Vert c \right\Vert_{\ell_{r,\theta}} \, .
  \end{align*}
  With $\Lambda_0$ being bounded, an antilinear version of the Hahn-Banach
  theorem yields a bounded (antilinear) extension
  $\Lambda : \ell_{r,\theta}(I) \to \mathbb{C}$ of $\Lambda_{0}$.
  Therefore, an antilinear version of the Riesz representation
  theorem for the dual of $\ell_{r,\theta}(I)$ ensures the existence of
  $\varrho = \left(\varrho_{i}\right)_{i \in I} \in \ell_{r', \theta^{-1}}(I)$
  satisfying $\Lambda(c) = \left\langle
  \varrho \,,\, c \right\rangle_{\ell_{r',\theta^{-1}} \times \ell_{r,\theta}}$
  for all $c \in \ell_{r,\theta}(I)$.
  Here, the pairing between $\ell_{r',\theta^{-1}}(I)$ and $\ell_{r,\theta}(I)$
  is given by
  $\langle
      (c_i)_{i \in I}, (e_i)_{i \in I}
   \rangle_{\ell_{r',\theta^{-1}} \times \ell_{r,\theta}}
   = \sum_{i \in I} c_i \cdot \overline{e_i}$.

  \medskip{}

  Having constructed the sequence $\varrho \in \ell_{r',\theta^{-1}}(I)$, we
  can now apply the second assumption of
  Theorem \ref{thm:DiscretizationObstruction}---the boundedness of the
  synthesis operator $S$---to define $g := S \varrho \in \operatorname{Co}(L_{r',m^{-1}})$.
  Furthermore, for arbitrary $x \in G$, we recall from
  Lemma \ref{lem:VoiceTransformOnCoorbitAntiDual} that
  $\pi(x) \, u \in \operatorname{Co}(L_{r,m})$, so that
  \[
    c^{(x)}
    = \big( c_{i}^{(x)} \,\big)_{i \in I}
    := A \left(\pi(x) u\right)
    = \left(
        \left\langle \pi(x) \, u \,,\, \pi(x_i) \, u \right\rangle_{\mathcal{H}}
      \right)_{i\in I} \in \ell_{r,\theta}(I)
  \]
  is well-defined.
  Combining our preceding observations, we see
  \begin{align}\label{eq:DiscretizationExclusedBadFunctionalsCalculation}
      V_{\mathrm{sp}} \, \varphi (x)
        & = \varphi( \pi (x) \, u)
        = \Lambda_0 \big( A ( \pi (x) \, u) \big)
        = \Lambda(c^{(x)})
        = \langle
            \varrho, c^{(x)}
          \rangle_{\ell_{r',\theta^{-1}},\times \ell_{r,\theta}} \notag\\
       &= \sum_{i \in I} \left[
                          \varrho_{i} \cdot \overline{
                                                       \left\langle
                                                         \pi(x) \, u
                                                         \,,\,
                                                         \pi(x_i) \, u
                                                       \right\rangle_{\mathcal{H}}
                                                     }
                        \right] 
      = \sum_{i \in I} \left[
                          \varrho_{i} \cdot \left\langle
                                              \pi(x_i) \, u \,,\, \pi(x) \, u
                                            \right\rangle_{\mathcal{S}_w}
                        \right]\\
      &\overset{\left(\ast\right)}{=}
         \left\langle
           \smash{\sum_{i\in I}}\vphantom{\sum}
             \left( \varrho_{i} \cdot \pi(x_i) \, u \right)
             \,,\,
             \pi(x) \, u
         \right\rangle_{\mathcal{S}_w} 
          = \left\langle
           S\varrho \,,\, \pi(x) \, u
         \right\rangle_{\mathcal{S}_w}
       = \left[V_{e} \, g\right](x).\notag
    \end{align}
  This identity---which will be fully justified below---completes the proof,
  since we have $g = S \varrho \in \operatorname{Co}(L_{r',m^{-1}})$, that is
  $V_{e} \, g \in L_{r',m^{-1}}(G)$. Therefore,
  \eqref{eq:DiscretizationExclusedBadFunctionalsCalculation} implies
  $V_{\mathrm{sp}} \, \varphi = V_{e} \, g \in L_{r',m^{-1}}(G)$, as claimed.

  It remains to justify the step marked with $(\ast)$ in
  \eqref{eq:DiscretizationExclusedBadFunctionalsCalculation}.
  At that step, we used on the one hand that
  $S \varrho = \sum_{i \in I} \left[\varrho_{i} \cdot \pi(x_i) \, u \right]$
  with unconditional convergence in $\operatorname{Co}(L_{r',m^{-1}})$.
  To see that this indeed holds, recall that $r' < \infty$,
  so that $\varrho = \sum_{i \in I} \varrho_{i} \, \delta_{i}$,
  with unconditional convergence in $\ell_{r',\theta^{-1}}(I)$;
  by the boundedness of $S$, this implies the claimed identity.
  On the other hand, we also used at $(\ast)$ that
  $\operatorname{Co}(L_{r',m^{-1}}) \to \mathbb{C}, f \mapsto \left\langle
  f \,,\, \pi(x) \, u \right\rangle_{\mathcal{S}_w}$
  is a bounded linear functional.
  Indeed, \eqref{eq:ExtendedVoiceTransformDuality}
  and Lemma~\ref{lem:VoiceTransformOnCoorbitAntiDual} imply
  \begin{align*}
  \left| \left\langle f \,,\, \pi(x) \, u \right\rangle_{\mathcal{S}_w} \right|
  & =\left|
       \left\langle
         V_{e} f \,,\, V \left[ \pi(x) \, u \right]
       \right\rangle_{L_2}
     \right|
  \leq \left\Vert V_{e} f \right\Vert_{L_{r',m^{-1}}}
       \cdot \left\Vert V \left[\pi(x) \, u \right] \right\Vert_{L_{r,m}}\\
   & = \left\Vert f \right\Vert_{\operatorname{Co}(L_{r',m^{-1}})}
       \cdot \left\Vert V \left[\pi(x) \, u \right]\right\Vert_{L_{r,m}}
  \leq C \cdot \left\Vert f \right\Vert_{\operatorname{Co}(L_{r',m^{-1}})}
             \cdot w(x) \, , 
  \end{align*}
  with $C = C(m,w,u,r)$.
\end{proof}

We can now finally prove Theorem \ref{thm:DiscretizationObstruction}.

\begin{proof}[Proof of Theorem \ref{thm:DiscretizationObstruction}]
  Assume towards a contradiction that the right convolution operator
  $RC_K : L_{r,m}(G) \to L_{r,m}(G)$ is not bounded.
  By Proposition \ref{prop:kernel_property}, and since the prerequisites of
  Theorem \ref{thm:DiscretizationObstruction} include
  Assumption \ref{assume:KernelAlmostIntegrable}, this implies that
  $RC_K : L_{r', m^{-1}}(G) \to L_{r',m^{-1}}(G)$ is also not bounded.
  Therefore, Lemma \ref{lem:NonBoundednessYieldsBadFunctionals}
  yields an antilinear continuous functional $\varphi \in [\operatorname{Co}(L_{r,m})]'$ with
  $V_{\mathrm{sp}} \varphi \notin L_{r', m^{-1}}(G)$.
  In view of Lemma \ref{lem:DiscretizationExcludesBadFunctionals}, this yields
  the desired contradiction.
\end{proof}

Before closing this section, we show that the ``weak integrability assumption''
$K \in \bigcap_{1 < p < \infty} L_{p,w}(G)$ does \emph{not} imply in general
that the right convolution operator $RC_K : f \mapsto f \ast K$ acts boundedly
on any $L_p$-space with $p \neq 2$.

To this end, we consider as in Sect. \ref{sec:examples} the Paley-Wiener space
\begin{equation}
  \mathcal{H}
  = B_\Omega^2
  = \{
      f \in L_2(\mathbb{R})
      \,:\,
      \widehat{f} \equiv 0 \text{ almost everywhere on } \mathbb{R} \setminus \Omega
    \} 
  \label{eq:GeneralPaleyWienerDefinition}
\end{equation}
for a fixed measurable subset $\Omega \subset \mathbb{R}$ of finite measure.
As seen in Sect. \ref{sec:examples}, the group $G = \mathbb{R}$ acts on this space
by translations; that is, if we set $\pi(x) f = \lambda(x) f$ for $f \in B_\Omega^2$,
then $\pi$ is a unitary representation of $\mathbb{R}$.
Setting $u := \mathcal{F}^{-1} \chi_\Omega \in B_\Omega^2$,
using Plancherel's theorem, and noting $\widehat{f} =
\widehat{f} \cdot \chi_\Omega = \widehat{f} \cdot \overline{\widehat{u}}$
for $f \in B_\Omega^2$, we see that the associated voice transform
is given by
\begin{align*}
  V f (x)
  = \langle f \,,\, \pi(x) u \rangle_{L_2}
  = \langle
      \widehat{f} \,,\, e^{-2\pi i x \cdot\,} \widehat{u} \,
    \rangle_{L_2} \!
  = \int_{\mathbb{R}} \widehat{f}(\xi) \cdot e^{2\pi i x \xi} d\xi 
  = (\mathcal{F}^{-1} \widehat{f} \, ) (x)
  = \! f(x) \, .
\end{align*}
Thus, $V : B_\Omega^2 \to L_2(\mathbb{R})$ is an isometry and the reproducing kernel
$K$ is simply given by $K(x) = V u (x) = u(x)$ for $x \in \mathbb{R}$.
In view of these remarks, the following proposition shows that there is a
reproducing kernel that satisfies the weak integrability assumption,
but for which the associated right convolution operator does \emph{not}
act boundedly on $L_p(\mathbb{R})$ for any $p \neq 2$.

\begin{proposition}\label{prop:BadPaleyWienerSet}
  There is a compact set $C \subset [0,1]$ with the following properties:
  \begin{enumerate}[i)]
    \item $\mathcal{F}^{-1} \chi_C \in \bigcap_{1<p \leq \infty} L_p (\mathbb{R})$.
          \vspace{0.2cm}

    \item For any $p \in (1,\infty) \setminus \{2\}$, the convolution operator
          $f \mapsto f \ast \mathcal{F}^{-1} \chi_C$ is \emph{not} bounded, and by Proposition \ref{prop:RCK_bounded} not well-defined, as an operator on $L_p (\mathbb{R})$.
  \end{enumerate}
\end{proposition}

Since the construction of the set $C$ is quite technical, we defer the
proof to the appendix.

\section{Improved Discretization Results Under Additional Assumptions}
\label{sec:DiscretziationUnderAssumptions}

In the preceding section we have seen that there are limitations to the possible
discretization theory for coorbit spaces with ``bad'' kernels, that is,
for kernels $K$ for which the right convolution with $K$ does not act
boundedly on $L_{r,m}(G)$.

But even if this right convolution operator \emph{does} act
boundedly, the results in the preceding sections only yield discretization
results that are weaker than those that one would expect to hold when coming
from classical coorbit theory.
In the present section we will see that a ``proper'' discretization theory
is possible even for relatively bad (i.e., non-integrable) kernels,
as long as the kernel in question acts boundedly on $L_{r,m}(G)$ and is
compatible with another ``well-behaved'' kernel $W : G \to \mathbb{C}$, in the sense
that it satisfies $K \ast W = K$ for the construction of Banach frames,
or $W \ast K = K$ for the construction of atomic decompositions.

We emphasize that we do \emph{not} assume that the kernel $W$ satisfies
$W \ast W = W$, thereby allowing a larger freedom in the choice of $W$.
To see that the property $K \ast K = K$ is indeed quite restrictive, let
us consider the case when $G = \mathbb{R}$ is the real line. Then $K \ast K = K$
implies
that $\widehat{K} = \widehat{K} \cdot \widehat{K}$, so that
$\widehat{K} = \chi_\Omega$ must be the indicator function of a (measurable)
set, see also \cite{dadedelastte17}. In particular, $K \notin L_1 (\mathbb{R})$ (unless $K \equiv 0$), since otherwise
$\widehat{K}$ would be continuous.
In stark contrast, at least if the set $\Omega$ is bounded, one can choose a
Schwartz function $\psi$ with $\psi \equiv 1$ on $\Omega$, so that
$W := \mathcal{F}^{-1} \psi \in \mathbb{S}(\mathbb{R})$ satisfies
$\widehat{W \ast K}= \widehat{W} \cdot \chi_\Omega = \chi_\Omega
= \widehat{K}$, and thus $K \ast W = W \ast K = K$.
It is worth noting that a related approach has been established in \cite{FG07}.



The section is structured as follows: In the first subsection, we recall some
basic notions from classical coorbit theory: Relatively separated sets, BUPUs,
etc. Then, in Subsect. \ref{sub:GoodKernelBanachFrames} we discuss conditions
on the well-behaved kernel $W$ which guarantee the existence of Banach frames
for the coorbit spaces.
The existence of atomic decompositions, under similar but different conditions
on $W$, is discussed in Subsect. \ref{sub:GoodKernelAtomicDecompositions}.
In the last subsection we apply the abstract results to the setting of
Paley-Wiener spaces.

Finally, we should mention that most of the proofs in this section are heavily
inspired by the original coorbit papers \cite{fegr88,fegr89a,fegr89b,gro91}.
The main novel ingredient here is the observation that instead of the idempotent
reproducing formula $K \ast K = K$, it suffices to have $K \ast W = K$ or
$W \ast K = K$ for potentially different kernels $K,W$.

\begin{remark}
Most of the results in this section can also be obtained for coorbit
spaces $\operatorname{Co}(Y)$ where $Y$ is a solid Banach space continuously embedded into $L_0(G)$.
For simplicity, we restrict our attention to the case $Y = L_{r,m}(G)$ as
in the rest of the paper.
\end{remark}


\subsection{Required Notions from Classical Coorbit Theory}
\label{sub:GoodKernelPreliminaries}

We would like to sample the continuous frame $\big(\pi(x)u\big)_{x \in G}$
to obtain a discrete (Banach) frame \linebreak $\big(\pi(x_i) u\big)_{i \in I}$.
In order for this to succeed, the family of sampling points $(x_i)_{i \in I}$
needs to be sufficiently well distributed in $G$.
This intuition is made precise in the following definition.
The reader might compare this to the definitions in the beginning of Sect. \ref{atomicdecomp}.

\begin{definition}\label{def:WellSpreadFamily}(cf.~\cite[Definition 3.2]{fegr89a})

  Let $X = (x_i)_{i \in I}$ be a family in $G$.
  \begin{enumerate}[i)]
    \item $X$ is \emph{$V$-dense} in $G$, for a unit neighborhood
          $V \subset G$, if $G = \bigcup_{i \in I} x_i V$.

    \item $X$ is \emph{$V$-separated}, for a unit neighborhood $V \subset G$,
          if the family $(x_i V)_{i \in I}$ is pairwise disjoint.

    \item $X$ is \emph{relatively separated} if for every compact unit
          neighborhood $Q \subset G$ there is a constant $N = N(X,Q) \in \mathbb{N}$
          with
          \[
            \sum_{i \in I} \chi_{x_i Q}(x) \leq N
            \qquad \mbox{for all} \; x \in G \, .
          \]

    \item $X$ is \emph{$V$-well-spread} for a unit neighborhood $V \subset G$
          if $X$ is relatively separated and $V$-dense.
  \end{enumerate}
\end{definition}

\begin{remark}
\begin{enumerate}[i)]
  \item Since we always assume the underlying group $G$ to be second countable,
  $G$ is in particular $\sigma$-compact. Therefore, \cite[Lemma 2.3.10]{voigt}
  shows that (the index set of) every relatively separated family in $G$ is
  countable.
  
  \item Usually, $X$ is called relatively separated if $X$ is a finite union
  of $V$-separated sets, for some compact unit neighborhood $V$.
  The two definitions are shown to be equivalent in
  in \cite[Lemma~2.9]{feigro85} and \cite[Lemma~2.3.11]{voigt}.
  \end{enumerate}
\end{remark}

Given a $V$-well-spread family $X = (x_i)_{i \in I}$, one often wants to
decompose a given function $f$ into building blocks $f_i$ which are supported
in the sets $(x_i V)_{i \in I}$. This can be done using suitable partitions
of unity; again the reader might compare this to Sect. \ref{atomicdecomp}.

\begin{definition}\label{def:BUPU}(cf.~\cite[Definition 3.6]{fegr89a})

  Let $V \subset G$ be a compact unit neighborhood.
  A family $\Psi = (\psi_i)_{i \in I}$ is called a
  $V$\emph{-BUPU (bounded uniform partition of unity) with localizing family} $X = (x_i)_{i \in I}$
  if the following holds:
  \begin{enumerate}[i)]
    \item Each $\psi_i : G \to [0,1]$ is a measurable function.

    \item $X$ is relatively separated and $\psi_i \equiv 0$ on
          $G \setminus x_i V$ for all $i \in I$.

    \item We have $\sum_{i \in I} \psi_i \equiv 1$ on $G$.
  \end{enumerate}
\end{definition}

One can find a $V$-BUPU for any compact unit neighborhood $V$:

\begin{lemma}\label{lem:BUPUExistence}
  (cf.~\cite[Theorem 2]{FeichtingerMinimalHomogeneous} and
  \cite[Lemma 2.3.12]{voigt})

  Let $V \subset G$ be an arbitrary compact unit neighborhood.
  Then there exists a \linebreak $V$-BUPU $\Psi = (\psi_i)_{i \in I}$
  with $\psi_i \in C_c(G)$ for all $i \in I$.
\end{lemma}

The following lemma points out an important property of relatively separated
families that we will use time and again:

\begin{lemma}\label{lem:RelativelySeparatedSynthesis}
  Let $X = (x_i)_{i \in I}$ be a relatively separated family and let
  $r \in [1,\infty)$. Let further $m:G\to(0,\infty)$ be a $w$-moderate weight.
  Define the weight $m_X$ on the index set $I$ by $(m_X)_i := m(x_i)$ for
  $i \in I$.

  Then for every compact unit neighborhood $U \subset G$, the synthesis operator
  \[
    \operatorname{Synth}_{X,U} : \ell_{r,m_X}(I) \to L_{r,m}(G),
                   (c_i)_{i \in I} \mapsto \sum_{i \in I}
                                             c_i \, \chi_{x_i U}
  \]
  is well-defined and bounded, with pointwise absolute convergence of the
  defining series.

  Furthermore, if $\Psi = (\psi_i)_{i \in I}$ is a $U$-BUPU
  with localizing family $X$, then
  the synthesis operator
  \[
    \operatorname{Synth}_{X,\Psi} : \ell_{r,m_X}(I) \to L_{r,m}(G),
                      (c_i)_{i \in I} \mapsto \sum_{i \in I} c_i \, \psi_i
  \]
  is well-defined and bounded, with pointwise absolute convergence
  of the defining series.
\end{lemma}
\begin{proof}
  The second part of the lemma is a consequence of the first one:
  Since $0 \leq \psi_i \leq 1$, and since $\psi_i$ vanishes outside of
  $x_i U$, we have
  \[
    |(\operatorname{Synth}_{X,\Psi} c)(x)|
    \leq \sum_{i \in I} |c_i| \, \psi_i(x)
    \leq \sum_{i \in I} |c_i| \, \chi_{x_i U}(x)
    = (\operatorname{Synth}_{X,U} |c|)(x) < \infty
  \]
  for all $x \in G$ and all $c = (c_i)_{i \in I} \in \ell_{r,m_X}(I)$,
  where $|c| = (|c_i|)_{i \in I} \in \ell_{r,m_X}(I)$ with
  $\|\, |c| \,\|_{\ell_{r,m_X}} = \|c\|_{\ell_{r,m_X}}$,
  so that
  \[
    \|\operatorname{Synth}_{X,\Psi} c\|_{L_{r,m}}
    \leq \|\operatorname{Synth}_{X,U} |c| \,\|_{L_{r,m}}
    \lesssim \|\, |c| \,\|_{\ell_{r,m_X}}
    = \|c\|_{\ell_{r,m_X}} \, .
  \]
  Thus, it remains to prove the first part of the lemma.

  \medskip{}

  By definition of a relatively separated family, there
  is $N = N(X,U) > 0$ with $\sum_{i \in I} \chi_{x_i U} \leq N$.
  On the one hand, this shows that for each $x \in G$ only finitely many terms
  of the series defining $(\operatorname{Synth}_{X,U} c)(x)$ do no vanish; in particular, the
  defining series is pointwise absolutely convergent.
  On the other hand, we see
  \begin{align*}
    \left|\, \big(\operatorname{Synth}_{X,U} c \big)(x) \,\right|^r
    & \leq \left(
             \sum_{i \in I}
               |c_i| \, \chi_{x_i U}(x) \, \chi_{x_i U}(x)
           \right)^r 
           \leq \left(
             \sup_{j \in I} |c_j| \, \chi_{x_j U}(x)
             \cdot
             \sum_{i \in I}
               \chi_{x_i U}(x)
           \right)^r \\
    & \leq N^r \cdot \sup_{j \in I} |c_j|^r \, \chi_{x_j U} (x)
    \leq N^r \cdot \sum_{i \in I} |c_i|^r \, \chi_{x_i U} (x) \, .
  \end{align*}
  Thus,
  \begin{align*}
    \|\, \operatorname{Synth}_{X,U} c \,\|_{L_{r,m}}^r
    &\leq N^{r} \cdot \int_G
                            (m(x))^r \cdot
                            \sum_{i \in I} |c_i|^r \, \chi_{x_i U}(x)
                          \, dx \\
           &\leq N^{r} \cdot
           \sum_{i \in I}
           \left(
             |c_i|^r \cdot \int_{x_i U} (m(x))^r \, dx
           \right) \, .
  \end{align*}
  But for $x = x_i u \in x_i U$, we have
  $m(x) = m(x_i u) \leq m(x_i) \cdot w(u) \leq C \cdot m(x_i)$
  for $C := \sup_{u \in U} w(u)$, which is finite since $U$ is compact and
  $w$ is continuous.
  Overall, since $|x_i U| = |U|$ for all $i \in I$, where $|U|$ is the Haar-measure of $U$, we  see
  \[
    \|\operatorname{Synth}_{X,U} c \,\|_{L_{r,m}}^r
    \leq N^{r} \cdot C^r \cdot |U| \cdot
         \sum_{i \in I} \big(m(x_i) \cdot |c_i|\big)^r \, ,
  \]
  which easily yields the boundedness of $\operatorname{Synth}_{X,U}$.
\end{proof}

\subsection{Banach Frames}\label{sub:GoodKernelBanachFrames}

In this subsection, we will assume the following:

\begin{assumption}\label{assume:GoodKernelBanachFrameAssumption}
  We fix some $r \in (1,\infty)$ and a $w$-moderate weight $m:G\to(0,\infty)$ and assume that the kernel $K$ from
  \eqref{eq:KernelDefinition} satisfies the following:

  \begin{enumerate}[i)]
    \item Assumption \ref{assume:KernelAlmostIntegrable} is satisfied,
          that is, $K \in L_{p,w}(G)$ for all $p \in (1,\infty)$.

    \item The right convolution operator $RC_K : f \mapsto f \ast K$ is
          well-defined, and by Proposition \ref{prop:RCK_bounded} bounded, as an operator on $L_{r,m}(G)$.


    \item There is some unit neighborhood $U_0 \subset G$ such that for each
          unit neighborhood $U \subset U_0$ there is a constant $C_U > 0$
          with
          \begin{equation}
            \text{for all } f \in \mathcal{M}_{r,m} : \quad
              \|\operatorname{osc}_U^\rho f\|_{L_{r,m}} \leq C_U \cdot \|f\|_{L_{r,m}} \, .
            \label{eq:GoodKernelBanachFrameMainAssumption}
          \end{equation}
          Here, $\mathcal{M}_{r,m}$ is the reproducing kernel space from
          \eqref{eq:ReproducingKernelLp}, and
          \begin{equation}
            \operatorname{osc}_U^\rho f (x) := \sup_{u \in U} |f(xu) - f(x)|
            \label{eq:RightOscillation}
          \end{equation}
          similar to \eqref{oscillation}.

    \item The constants $C_U$ from the preceding point satisfy $C_U \to 0$
          as $U \to \{e\}$.
          More precisely, for every $\varepsilon > 0$ there is a unit
          neighborhood $U_\varepsilon \subset U_0$ with $C_U \leq \varepsilon$ for all unit neighborhoods
          $U \subset U_\varepsilon$.
  \end{enumerate}
\end{assumption}

At a first glance it seems that the preceding assumptions have nothing to do
with the existence of a ``well-behaved'' kernel $W$ which is compatible with
the kernel $K$. But it turns out that the existence of such a kernel provides
an easy way of verifying the preceding assumptions:

\begin{lemma}\label{lem:BanachFrameKernelAssumption}
  Assume that $K \in L_{p,w}(G)$ for all $p \in (1,\infty)$ and that the
  operator $RC_K : L_{r,m}(G) \to L_{r,m}(G), f \mapsto f \ast K$ is
  well-defined and bounded.

  Furthermore assume that there is a kernel $W : G \to \mathbb{C}$ with the following
  properties:
  \begin{enumerate}[i)]
    \item $W$ is continuous.

    \item $K \ast W = K$.

    \item $M_{U_0}^\lambda W \in L_{1,w}(G) \cap L_{1,w \Delta^{-1}}(G)$ for some compact
          unit neighborhood $U_0 \subset G$. Here
          \begin{equation}
            M_{U_0}^\lambda W (x) := \|W\|_{L_\infty (x U_0)},
            \qquad x \in G \,
            \label{eq:MaximalFunctionDefinition}
          \end{equation}
          is the \emph{local maximal function (with respect to left regular representation)}, similar to \eqref{local_maximum_function}.
  \end{enumerate}
  Then Assumption \ref{assume:GoodKernelBanachFrameAssumption} is satisfied.
\end{lemma}

\begin{proof}
  We first note that our assumptions imply
  $M_U^\lambda W \in L_{1,w}(G) \cap L_{1,w \Delta^{-1}}(G)$ for \emph{every} compact unit
  neighborhood $U \subset G$. Indeed, by compactness, and since
  $U \subset \bigcup_{x \in G} x \operatorname{int}(U_0)$, there is a finite
  family $(x_i)_{i=1,\dots,n}$ with $U \subset \bigcup_{i=1}^n x_i U_0$.
  Therefore, $x U \subset \bigcup_{i=1}^n x x_i U_0$, whence
  \begin{align*}
    M_U^\lambda W (x)
     &= \|W\|_{L_\infty(xU)}
      \leq \sum_{i=1}^n \|W\|_{L_\infty(x x_i U_0)} \\
     &= \sum_{i=1}^n M_{U_0}^\lambda W (x x_i)
      = \sum_{i=1}^n \big[ \rho(x_i) (M_{U_0}^\lambda W) \big] (x) \, .
  \end{align*}
  But since $w$ and $w \Delta^{-1}$ are submultiplicative, both $L_{1,w}(G)$ and $L_{1, w \Delta^{-1}}(G)$ are invariant under right
  translations, and hence $M_U^\lambda W \in L_{1,w}(G) \cap L_{1, w \Delta_{-1}}(G)$.

  Next, if $V$ is an \emph{open} precompact unit neighborhood and
  $U := \overline{V}$, then by continuity of $W$, we have
  $|W(x)| \leq \sup_{v \in V} |W(xv)| = \|W\|_{L_\infty(xV)} \leq M_U^\lambda W (x)$
  for all $x \in G$.
  Therefore, $W \in L_{1,w}(G) \cap L_{1, w \Delta^{-1}}(G)$.

  \medskip{}

  Since by assumption the right convolution operator $RC_K$ acts boundedly
  on $L_{r,m}(G)$, Lemma \ref{lem:kernel_continuity} shows that the set
  $X_0 := \operatorname{span} \left\{\lambda(x) K\right\}_{x \in G}$ is dense in the reproducing
  kernel space $\mathcal{M}_{r,m}$. Furthermore, the assumption $K \ast W = K$ yields
  $(\lambda(x) K) \ast W = \lambda(x) (K \ast W) = \lambda(x) K$ for all $x \in G$, and thus
  $f \ast W = f$ for all $f \in X_0$.
  By density of $X_0$ in $\mathcal{M}_{r,m}$, and since the right convolution operator
  $f \mapsto f \ast W$ is continuous on $L_{r,m}(G)$ thanks to
  $W \in L_{1,w}(G) \cap L_{1, w \Delta^{-1}}(G)$ and Young's inequality (Proposition \ref{prop:Young}),
  we see
  \begin{equation}
    f \ast W = f
    \quad \text{for all} \quad f \in \mathcal{M}_{r,m} \, .
    \label{eq:BanachFrameReproducingFormula}
  \end{equation}

  \medskip{}

  We now use \eqref{eq:BanachFrameReproducingFormula}
  to prove \eqref{eq:GoodKernelBanachFrameMainAssumption}.
  To this end, let $U \subset G$ be an arbitrary compact unit neighborhood.
  Let $f \in \mathcal{M}_{r,m}$, $x \in G$ and $u \in U$ be arbitrary. Then
  \begin{align*}
    |f(xu) - f(x)|
    & = |(f \ast W)(xu) - (f \ast W)(x)| 
     \leq \int_G
             |f(y)| \cdot |W(y^{-1} xu) - W(y^{-1} x)|
           \, dy \\
    & \leq \int_G
             |f(y)| \cdot (\operatorname{osc}_U^\rho W)(y^{-1} x)
           \, dy
      = \big( |f| \ast (\operatorname{osc}_U^\rho W) \big) (x) \, .
  \end{align*}
  Since this holds for every $u \in U$, we get
  $\operatorname{osc}_U^\rho f (x) \leq |f| \ast (\operatorname{osc}_U^\rho W)(x)$ for all $x \in G$.
  By solidity of $L_{r,m}(G)$ and in view of Young's inequality
  (Proposition \ref{prop:Young}), this implies
  \[
    \|\operatorname{osc}_U^\rho f\|_{L_{r,m}}
    \leq \|f\|_{L_{r,m}}
         \cdot \max\{
                     \|\operatorname{osc}_U^\rho W\|_{L_{1,w}},
                     \|\operatorname{osc}_U^\rho W\|_{L_{1, w \Delta^{-1}}}
                   \} \, .
  \]
  But an easy generalization of Lemma \ref{lem:osc}
  shows that $\|\operatorname{osc}_U^\rho W\|_{L_{1,v}} \to 0$ as $U \to \{e\}$,
  for $v = w$ as well as for $v = w \Delta^{-1}$.
  From this it is not hard to see that the two remaining properties
  from Assumption \ref{assume:GoodKernelBanachFrameAssumption} are satisfied.
\end{proof}

We now prove that Assumption \ref{assume:GoodKernelBanachFrameAssumption}
ensures that a sufficiently fine sampling of the continuous frame
$(\pi(x) \, u)_{x \in G}$ provides a Banach frame for the coorbit space
$\operatorname{Co}(L_{r,m})$. For this, we will first show that we can sample the point
evaluation functionals to obtain a Banach frame for the reproducing kernel space
$\mathcal{M}_{r,m}$. In the end, we will then use the correspondence principle
to transfer the result from the reproducing kernel space to the coorbit space.

We begin by showing that a certain sampling operator is bounded:

\begin{lemma}\label{lem:BanachFrameSamplingOperator}
  Let Assumption \ref{assume:GoodKernelBanachFrameAssumption} be satisfied,
  and let $X = (x_i)_{i \in I}$ be a relatively separated family in $G$.

  Then, with the weight $m_X$ as in
  Lemma \ref{lem:RelativelySeparatedSynthesis}, the sampling operator
  \[
    \operatorname{Samp}_X : \mathcal{M}_{r,m} \to \ell_{r,m_X}(I),
              f        \mapsto (f(x_i))_{i \in I}
                               = \big(
                                   \langle f , \lambda_{x_i} \, K \rangle_{L_2}
                                 \big)_{i \in I}
  \]
  is well-defined and bounded.
\end{lemma}

\begin{proof}
  We first recall that each $f \in \mathcal{M}_{r,m}$ satisfies $f = f \ast K$, and hence
  \[
    f(x)
    = f \ast K (x)
    = \int_G f(y) \cdot K(y^{-1} x) \, dy
    = \int_G f(y) \cdot \overline{K(x^{-1} y)} \, dy
    = \langle f, \lambda(x) K \rangle_{L_2} \, .
  \]
  But $K \in L_{r',w}(G)$, and thus also
  $\lambda(x) K \in L_{r', w}(G)$, since $w$ is submultiplicative.
  Furthermore, since $m$ is $w$-moderate, we have
  $m(e) = m(xx^{-1}) \leq m(x) w(x^{-1})$, and thus
  $[m(x)]^{-1} \leq w(x^{-1}) / m(e) = w(x) / m(e)$, 
  whence
  $L_{r', w}(G) \hookrightarrow L_{r', m^{-1}}(G)$.
  Thus, the dual pairing $\langle f, \lambda(x) K \rangle_{L_2} \in \mathbb{C}$ is
  well-defined for every $x \in \mathbb{R}$. Therefore, each entry $f(x_i)$ of the
  sequence $\operatorname{Samp}_X f = (f(x_i))_{i \in I}$ makes sense.

  Now, let $U$ be a compact unit neighborhood with
  $\|\operatorname{osc}_U^\rho f\|_{L_{r,m}} \leq C \cdot \|f\|_{L_{r,m}}$ for all
  $f \in \mathcal{M}_{r,m}$.  Such a neighborhood exists by virtue of
  Assumption \ref{assume:GoodKernelBanachFrameAssumption}.
  Note that $U^{-1}$ is also a compact unit neighborhood, so that by definition
  of a relatively separated family there is a constant $N > 0$ with
  $\sum_{i \in I} \chi_{x_i U^{-1}}(x) \leq N$ for all $x \in G$.

  Next, fix any $i \in I$ and note that $\chi_{x_i U^{-1}}(x) \neq 0$
  can only hold if $x = x_i u^{-1}$ and thus $x_i = x u$ for some $u \in U$.
  But in this case, we see by definition of the oscillation $\operatorname{osc}_U^\rho f$ that
  \[
    |f(x_i)|
    \leq |f(x)| + |f(x_i) - f(x)|
    \leq |f(x)| + (\operatorname{osc}_U^\rho f)(x)
    =:   F(x) \, .
  \]
  We have thus shown $|f(x_i)| \cdot \chi_{x_i U^{-1}}(x)
  \leq F(x) \cdot \chi_{x_i U^{-1}} (x)$ for all $x \in G$.
  Summing over $i \in I$, we see
  \[
    \Theta_f (x)
    := \sum_{i \in I} |f(x_i)| \cdot \chi_{x_i U^{-1}}(x)
    \leq \left(\sum_{i \in I} \chi_{x_i U^{-1}}(x)\right) \cdot F(x)
    \leq N \cdot F(x)
  \]
  for all $x\in G$.

  Because of $r \geq 1$, we have $\ell_1(I) \hookrightarrow \ell_r (I)$, which
  implies $\sum_{i \in I} c_i^r \leq (\sum_{i \in I} c_i)^r$ for arbitrary
  $c_i \geq 0$. Therefore,
  \begin{align*}
    \int_G
       (m(x))^r
       \cdot \sum_{i \in I}
               |f(x_i)|^r \cdot \chi_{x_i U^{-1}} (x)
    \, dx
    & \leq \|\Theta_f\|_{L_{r,m}}^r
      \leq N^r \cdot \|F\|_{L_{r,m}}^r \\
    &\leq N^r \cdot \left(
                       \|f\|_{L_{r,m}} + \|\operatorname{osc}_U^\rho f\|_{L_{r,m}}
                     \right)^r \\
    & \leq N^r \cdot (1 + C)^r \cdot \|f\|_{L_{r,m}}^r \, .
  \end{align*}
  Finally, if $\chi_{x_i U^{-1}}(x) \neq 0$, then $x = x_i u^{-1}$ for
  some $u \in U$, and therefore
  $m(x_i) = m(x u) \leq m(x) \cdot w(u) \leq C' \cdot m(x)$
  for $C' := \sup_{u \in U} w(u)$, which is finite since $w$ is continuous and
  $U$ is compact.

  Overall, we have thus shown
  \begin{align*}
    (C')^{-r} \cdot \sum_{i \in I}
                      (m(x_i))^r \cdot |f(x_i)|^r \cdot |x_i U^{-1}|
    & \leq \int_G
              (m(x))^r
              \cdot \sum_{i \in I}
                      |f(x_i)|^r \cdot \chi_{x_i U^{-1}} (x)
           \, dx \\
    & \leq N^r \cdot (1+C)^r \cdot \|f\|_{L_{r,m}}^r \, ,
  \end{align*}
  which---because of $|x_i U^{-1}| = |U^{-1}|$---shows
  \[
    \|\operatorname{Samp}_X f\|_{\ell_{r,m_X}}
    \leq C' N (1+C) \cdot |U^{-1}|^{-1/r} \cdot \|f\|_{L_{r,m}}
    \quad \text{for all} \, f \in \mathcal{M}_{r,m} \, ,
  \]
  which finally proves that $\operatorname{Samp}_X$ is well-defined and bounded.
\end{proof}

Now we can prove that a sufficiently fine sampling of the point evaluations
yields a Banach frame for the reproducing kernel space $\mathcal{M}_{r,m}$.

\begin{proposition}\label{prop:GoodKernelBanachFrame}
  Let Assumption \ref{assume:GoodKernelBanachFrameAssumption} be satisfied, and let $U \subset U_0^{-1}$ be a compact unit neighborhood such that the constant
  $C_{U^{-1}}$ from \eqref{eq:GoodKernelBanachFrameMainAssumption}
  satisfies
  \begin{equation}
    \| RC_K \|_{L_{r,m} \to L_{r,m}} \cdot C_{U^{-1}} < 1 \, .
    \label{eq:BanachFrameDensityCondition}
  \end{equation}
  
  Let $X = (x_i)_{i \in I}$ be any relatively separated family in $G$
  for which there exists a $U$-BUPU $\Psi = (\psi_i)_{i \in I}$ with
  localizing family $X$, and let
  the weight $m_X$ be defined as in
  Lemma \ref{lem:RelativelySeparatedSynthesis}.

  Then there is a bounded linear \emph{reconstruction map}
  $R : \ell_{r,m_X}(I) \to \mathcal{M}_{r,m}$ which satisfies
  $R \circ \operatorname{Samp}_X = \operatorname{id}_{\mathcal{M}_{r,m}}$ for the sampling map $\operatorname{Samp}_X$
  from Lemma \ref{lem:BanachFrameSamplingOperator}.

  In other words, the family $(\delta_{x_i})_{i \in I}$ of point evaluations
  forms a \emph{Banach frame} for $\mathcal{M}_{r,m}$ with coefficient space
  $\ell_{r,m_X}(I)$.
\end{proposition}

\begin{remark}
  The proof shows that the action of the reconstruction operator is
  \emph{independent} of the choice of $r,m$.

  In other words, if \eqref{eq:BanachFrameDensityCondition}
  is satisfied for $L_{r_1,m_1}(G)$ and $L_{r_2,m_2}(G)$ and if \linebreak
  $R_1 : \ell_{r_1,m_{1,X}}(I) \to \mathcal{M}_{r_1,m_{1}}$ and $R_2 : \ell_{r_2,m_{2,X}}(I) \to \mathcal{M}_{r_2,m_{2}}$
  denote the respective reconstruction operators, then $R_1 c = R_2 c$ for all
  $c \in \ell_{r_1,m_{1,X}}(I) \cap \ell_{r_2,m_{2,X}}(I)$.
\end{remark}

\begin{proof}
  With the synthesis operator $\operatorname{Synth}_{X,\Psi}$ from
  Lemma \ref{lem:RelativelySeparatedSynthesis}, we define
  \[
    B := \operatorname{Synth}_{X,\Psi} \circ \operatorname{Samp}_X : \mathcal{M}_{r,m} \to L_{r,m}(G) \, .
  \]
  Because of $f(x) = \sum_{i \in I} \psi_i (x) f(x)$, we have
  \begin{align*}
    |f(x) - Bf(x)|
    \leq \sum_{i \in I} \psi_i (x) \cdot |f(x) - f(x_i)| \, .
  \end{align*}
  But if $\psi_i (x) \neq 0$, then $x = x_i u \in x_i U$, so that
  $x_i = x u^{-1} \in x U^{-1}$, and hence
  $|f(x) - f(x_i)| = |f(x) - f(x u^{-1})| \leq \operatorname{osc}_{U^{-1}}^\rho f (x)$.
  Therefore,
  \[
    |f(x) - Bf(x)|
    \leq \sum_{i \in I} \psi_i (x) \operatorname{osc}_{U^{-1}}^\rho f (x)
    = \operatorname{osc}_{U^{-1}}^\rho f (x) \, .
  \]

  By Proposition~\ref{prop:kernel_property} the operator $RC_K$ is a projection onto $\mathcal{M}_{r,m}$, therefore $RC_K f = f$ for $f \in \mathcal{M}_{r,m}$.
  Thus, the operator $A := RC_K \circ B : \mathcal{M}_{r,m} \to \mathcal{M}_{r,m}$
  is well-defined and bounded, and we have
  \begin{align*}
    \|f - A f\|_{L_{r,m}}
    &= \|RC_K (f - Bf)\|_{L_{r,m}}
    \leq \|RC_K\|_{L_{r,m} \to L_{r,m}} \cdot \|f - Bf\|_{L_{r,m}} \\
    & \leq \|RC_K\|_{L_{r,m} \to L_{r,m}}
           \cdot \|\operatorname{osc}_{U^{-1}}^\rho f\|_{L_{r,m}} 
    \leq \|RC_K\|_{L_{r,m} \to L_{r,m}}
           \cdot C_{U^{-1}} \cdot \|f\|_{L_{r,m}}
  \end{align*}
  for all $f \in \mathcal{M}_{r,m}$.

  In view of \eqref{eq:BanachFrameDensityCondition}, a Neumann series
  argument (see \cite[Sect. 5.7]{AltFunctionalAnalysis}) shows that the
  bounded linear operator
  $R_0 := \sum_{n=0}^\infty (\operatorname{id}_{\mathcal{M}_{r,m}} - A)^n : \mathcal{M}_{r,m}\to\mathcal{M}_{r,m}$
  satisfies
  \[
    (R_0 \circ RC_K \circ \operatorname{Synth}_{X,\Psi}) \circ \operatorname{Samp}_X
    = R_0 \circ A
    = \operatorname{id}_{\mathcal{M}_{r,m}} \, .
  \]
  Thus,
  $R := R_0 \circ RC_K \circ \operatorname{Synth}_{X,\Psi} : \ell_{r,m_X}(I) \to \mathcal{M}_{r,m}$
  is the desired reconstruction operator. Note that the action of this operator
  on a given sequence is independent of the choice of $r,m$, since the action
  of the operators $RC_K$, $\operatorname{Synth}_{X,\Psi}$ and
  $A = RC_K \circ \operatorname{Synth}_{X,\Psi} \circ \operatorname{Samp}_X$ is independent of $r,m$, so
  that the same holds for $R_0 = \sum_{n=0}^\infty (\operatorname{id} - A)^n$.
\end{proof}

Using the correspondence principle, we can finally lift the result from
the reproducing kernel space $\mathcal{M}_{r,m}$ to the coorbit space $\operatorname{Co}(L_{r,m})$.

\begin{theorem}\label{thm:GoodKernelBanachFrameCoorbit}
  Under the assumptions of Proposition \ref{prop:GoodKernelBanachFrame},
  the sampled family $(\pi(x_i) u)_{i \in I} \subset (\operatorname{Co}(L_{r,m}))'$
  forms a Banach frame for $\operatorname{Co}(L_{r,m})$
  with coefficient space $\ell_{r,m_X}(I)$.

  More precisely, the sampling operator
  \[
    \operatorname{Samp}_{X,\mathrm{Co}} : \operatorname{Co}(L_{r,m}) \to \ell_{r,m_X}(I) ,
                      f \mapsto \big( V_e f (x_i) \big)_{i \in I}
                              = \big(
                                  \langle f, \pi(x_i) u \rangle_{\mathcal{S}_w}
                                \big)_{i \in I}
  \]
  is well-defined and bounded, and there is a bounded linear reconstruction
  operator $R_{\mathrm{Co}} : \ell_{r,m_X}(I) \to \operatorname{Co}(L_{r,m})$ satisfying
  $R_{\mathrm{Co}} \circ \operatorname{Samp}_{X,\mathrm{Co}} = \operatorname{id}_{\operatorname{Co}(L_{r,m})}$.

  Finally, the action of the reconstruction operator $R_{\mathrm{Co}}$ is
  independent of the choice of $r,m$, that is, if the assumptions of the current
  theorem are satisfied for $L_{r_1,m_1}(G)$ and for $L_{r_2,m_2}(G)$ and if $R_1, R_2$
  denote the corresponding reconstruction operators, then $R_1 c = R_2 c$
  for all $c \in \ell_{r_1,m_{1,X}}(I) \cap \ell_{r_2,m_{2,X}}(I)$.
\end{theorem}

\begin{proof}
  The correspondence principle (Proposition \ref{prop6}) states that the
  extended voice transform $V_e : \operatorname{Co}(L_{r,m}) \to \mathcal{M}_{r,m}$ is an isometric
  isomorphism. Now, with the sampling map $\operatorname{Samp}_X$ from
  Proposition \ref{prop:GoodKernelBanachFrame}, we have
  \[
    (\operatorname{Samp}_X \circ V_e) f
    = \big( V_e f (x_i) \big)_{i \in I}
    = \big( \langle f, \pi(x_i) u \rangle_{\mathcal{S}_w} \big)_{i \in I}
    = \operatorname{Samp}_{X, \mathrm{Co}} f \, .
  \]
  Thus, the sampling operator
  $\operatorname{Samp}_{X,\mathrm{Co}} = \operatorname{Samp}_X \circ V_e : \operatorname{Co}(L_{r,m}) \to \ell_{r,m_X}(I)$
  is indeed well-defined and bounded.

  Now, with the reconstruction operator $R : \ell_{r,m_X}(I) \to \mathcal{M}_{r,m}$
  from Proposition \ref{prop:GoodKernelBanachFrame}, define
  $R_{\mathrm{Co}} := V_e^{-1} \circ R : \ell_{r,m_X}(I) \to \operatorname{Co}(L_{r,m})$.
  Then
  \[
    R_{\mathrm{Co}} \circ \operatorname{Samp}_{X,\mathrm{Co}}
    = V_e^{-1} \circ R \circ \operatorname{Samp}_X \circ V_e
    = V_e^{-1} \circ V_e = \operatorname{id}_{\operatorname{Co}(L_{r,m})},
  \]
  as desired. Since the action of $R$ is independent of the choice of $r,m$,
  so is the action of $R_{\mathrm{Co}}$.
\end{proof}

\subsection{Atomic Decompositions}\label{sub:GoodKernelAtomicDecompositions}

For the case of atomic decompositions we will impose slightly different
conditions compared to the case of Banach frames.
In this case, our assumptions immediately refer to a ``well-behaved'' kernel
$W$.

\begin{assumption}\label{assume:GoodKernelAtomicDecomposition}
  We fix some $r \in (1,\infty)$ and some $w$-moderate weight \linebreak $m:G\to(0,\infty)$, and we assume that the kernel $K$ from
  \eqref{eq:KernelDefinition} satisfies the following:

  \begin{enumerate}[i)]
    \item Assumption \ref{assume:KernelAlmostIntegrable} is satisfied,
          that is, $K \in L_{p,w}(G)$ for all $p \in (1,\infty)$.

    \item The right convolution operator $RC_K : f \mapsto f \ast K$ is
          well-defined and bounded as an operator on $L_{r,m}(G)$.


    \item There is a continuous kernel $W : G \to \mathbb{C}$ with the following
          properties:
          \begin{enumerate}[a)]
            \item $W \ast K = K$.

            \item $\widecheck{M}_Q^\rho W \in L_{1,w}(G)\cap L_{1, w \Delta^{-1}}(G)$
                      for some compact unit neighborhood $Q \subset G$.
                      Here, the maximal function $\widecheck{M}_Q^\rho W$ is defined
                      as in \eqref{local_maximum_function}.
          \end{enumerate}
  \end{enumerate}
\end{assumption}

\begin{remark}\label{rem:MaximalFunctionNeighborhoodIndependence}
  We will use below that
  $\widecheck{M}_U^\rho W \in L_{1,w}(G) \cap L_{1, w \Delta^{-1}}(G)$
  for \emph{every} compact unit neighborhood $U \subset G$ if we assume
  $\widecheck{M}_Q^\rho W \in L_{1,w}(G) \cap L_{1, w \Delta^{-1}}(G)$ for \emph{some}
  unit neighborhood $Q \subset G$.

  Indeed, by compactness of $U$, and since
  $U \subset \bigcup_{x \in G} (\operatorname{int} Q) x$, there is a finite
  family $(x_i)_{i=1,\dots,n}$ in $G$ with $U \subset \bigcup_{i=1}^n Q x_i$.
  Therefore, $U x \subset \bigcup_{i=1}^n Q x_i x$, whence
  \[
    \widecheck{M}_U^\rho W (x)
    = \|W\|_{L_\infty(U x)}
    \leq \sum_{i=1}^n \|W\|_{L_\infty (Q x_i x)}
    =    \sum_{i=1}^n (\widecheck{M}_Q^\rho W) (x_i x) \, .
  \]
  By solidity and (left) translation invariance of $L_{1,v}(G)$ for $v = w$ or
  $v = w \Delta^{-1}$, this implies
  \[
    \|\widecheck{M}_U^\rho W\|_{L_{1,v}}
    \leq \sum_{i=1}^n \| \lambda(x_i^{-1}) (\widecheck{M}_Q^\rho W) \|_{L_{1,v}}
    < \infty \, .
  \]
  Here, the left-translation invariance of $L_{1,v}(G)$ is a consequence
  of the submultiplicativity of $v$.
\end{remark}

As in the preceding subsection, our first goal is to show that certain
synthesis and analysis operators are bounded.

\begin{lemma}\label{lem:GoodKernelAtomicDecompositionOperators}
  Let Assumption \ref{assume:GoodKernelAtomicDecomposition} be satisfied,
  and let $X = (x_i)_{i \in I}$ be any relatively separated family in $G$.
  Let the weight $m_X$ be as in Lemma \ref{lem:RelativelySeparatedSynthesis}.
  Then the following hold:

  \begin{enumerate}[i)]
    \item If $\Psi = (\psi_i)_{i \in I}$ is a $U$-BUPU with
          localizing family $X$, then
          the analysis operator
          \[
            \operatorname{Ana}_{X,\Psi} : L_{r,m}(G) \to \ell_{r,m_X}(I),
                            f \mapsto \left(
                                        \langle f, \psi_i \rangle_{L_2}
                                      \right)_{i \in I}
          \]
          is a well-defined bounded linear map.

    \item The synthesis map
          \[
            \operatorname{Synth}_{X,W} : \ell_{r,m_X}(I) \to L_{r,m}(G),
                           (c_i)_{i \in I} \mapsto \sum_{i \in I}
                                                     c_i \cdot \lambda(x_i) W
          \]
          is a well-defined bounded linear map, where the defining series
          is almost everywhere absolutely convergent.
  \end{enumerate}
\end{lemma}

\begin{proof}
  For $x = x_i u \in x_i U$ we have
  $m(x_i) = m(x u^{-1}) \leq m(x) w(u^{-1}) \leq C \cdot m(x)$, where
  $C := \sup_{u \in U} w(u^{-1})$ is finite by continuity of $w$ and
  compactness of $U$.
  Since we also have $\psi_i \equiv 0$ on $G \setminus x_i U$, then we see by following 
  the lines of the proof of Proposition~\ref{prop:upper_bound} and using Jensen's inequality, see \cite[Theorem~10.2.6]{DudleyRealAnalysisProbability}:
  \begin{align*}
    (m(x_i))^r \cdot |\langle f, \psi_i \rangle_{L_2}|^r
    &\leq |x_i U|^r \cdot \left(
                              \int_{x_i U}
                                |f(x)| m(x_i) \psi_i (x)
                              \frac{dx}{|x_i U|}
                            \right)^r \\
    &\leq |x_i U|^r \! \cdot \! \int_{x_i U}
                                        (
                                          |f(x)|
                                          \cdot C
                                          \cdot m(x)
                                          \psi_i (x)
                                        )^r
                                     \frac{dx}{|x_i U|} \\
    & \leq |U|^{r-1} \cdot C^r
           \cdot \int_G |(m \cdot f)(x)|^r \cdot \psi_i (x) \, dx \, ,
  \end{align*}
  where the last step used the left invariance of the Haar measure and the
  estimate $(\psi_i(x))^r \leq \psi_i (x)$ which holds since
  $\psi_i (x) \in [0,1]$ and $r \geq 1$.

  Summing over $i \in I$ and applying the monotone convergence theorem,
  we thus get
  \begin{align*}
    \|\operatorname{Ana}_{X,\Psi} f\|_{\ell_{r,m_X}}^r
    & = \sum_{i \in I}
        \big(m(x_i) \cdot |\langle f, \psi_i \rangle_{L_2}|\big)^r \\
    &\leq |U|^{r-1} \cdot C^r \cdot \int_G
                                            |(m \cdot f)(x)|^r
                                            \cdot \sum_{i \in I} \psi_i (x)
                                          \, dx \\
    & = |U|^{r-1} \cdot C^r \cdot \|f\|_{L_{r,m}}^r < \infty \, ,
  \end{align*}
  thereby proving the boundedness and well-definedness of $\operatorname{Ana}_{X,\Psi}$.

  \medskip{}

  We now consider the synthesis map $\operatorname{Synth}_{X,W}$.
  Let $V \subset G$ be any compact unit neighborhood, and set
  $Q := \operatorname{int} V$, so that $U := \overline{Q} \subset V$ is a
  compact unit neighborhood that satisfies
  $\overline{\operatorname{int} U} \supset \overline{Q} = U$ and hence $U=\overline{\operatorname{int} U}$.
  As a consequence, as seen in the proof of Lemma \ref{lem:osc}
  (see p. \pageref{proof:OscillationLemmaProof}), we have
  $\|W\|_{L_\infty(Ux)} = \sup_{y \in Ux} |W(y)|$ for all $x \in G$.
  Here we used that $W$ is continuous.

  Now, let $x \in G$ and $i \in I$ be arbitrary. For any $y = x_i u \in x_i U$
  we then have $x_i^{-1} x = (y u^{-1})^{-1} x = u y^{-1} x \in U y^{-1} x$,
  and thus
  \[
    |W(x_i^{-1} x)|
    \leq \|W\|_{L_\infty(U y^{-1} x)}
    = (\widecheck{M}_U^\rho W)(y^{-1} x)
    \quad \text{for all} \, x \in G \text{ and } y \in x_i U \, .
  \]
  Writing $\Theta := \widecheck{M}_U^\rho W$, and averaging this estimate over
  $y \in x_i U$, we get
  \begin{equation}
    |\lambda_{x_i} W (x)|
    \leq \frac{1}{|U|} \int_G
                            \chi_{x_i U}(y) \cdot \Theta(y^{-1} x)
                          \, dy
    \quad \text{for all} \, x \in G, \, i \in I \, .
    \label{eq:GoodKernelLeftTranslationPointwiseEstimate}
  \end{equation}

  Now, let $c = (c_i)_{i \in I} \in \ell_{r,m_X} (I)$ be arbitrary, and
  set $\Upsilon := \sum_{i \in I} |c_i| \cdot \chi_{x_i U}$.
  With the notation introduced in
  Lemma \ref{lem:RelativelySeparatedSynthesis} we get
  $\Upsilon = \operatorname{Synth}_{X,U} |c|$ with $|c| = (|c_i|)_{i \in I}$.
  This easily implies $\Upsilon \in L_{r,m}(G)$ with
  \begin{equation}
    \|\Upsilon\|_{L_{r,m}} \leq C \cdot \|c\|_{\ell_{r,m_X}}
    \label{eq:LittleEllRSpaceIsDiscretization}
  \end{equation}
  for a constant $C = C(m,X,U,r)$ independent of $c$.

  By weighting estimate \eqref{eq:GoodKernelLeftTranslationPointwiseEstimate}
  with $|c_i|$ and summing over $i \in I$, and by invoking the monotone
  convergence theorem, we see for all $x \in G$ that
  \[
    \sum_{i \in I} |c_i| \cdot |(\lambda(x_i) W)(x)|
    \leq \frac{1}{|U|} \cdot \int_G
                                  \Upsilon(y) \cdot \Theta (y^{-1} x)
                                \, dy
    =    \frac{1}{|U|} \cdot (\Upsilon \ast \Theta)(x) \, .
  \]
  But since $\Theta = \widecheck{M}_U^\rho W \in L_{1,w}(G) \cap L_{1, w \Delta^{-1}}(G)$
  (see Assumption \ref{assume:GoodKernelAtomicDecomposition} and Remark
  \ref{rem:MaximalFunctionNeighborhoodIndependence}) and since
  $\Upsilon \in L_{r,m}(G)$, Young's inequality (Proposition \ref{prop:Young})
  shows $\Upsilon \ast \Theta \in L_{r,m}(G)$.
  In particular, $\Upsilon \ast \Theta (x) < \infty$ almost everywhere.
  Therefore, we already see that the series defining $\operatorname{Synth}_{X,W} c$ is almost
  everywhere absolutely convergent.
  Finally, we also see
  \begin{align*}
    \|\operatorname{Synth}_{X,W} c\|_{L_{r,m}}
    & \leq \Big\|\, \sum_{i \in I} |c_i| \cdot |\lambda(x_i) W| \,\Big\|_{L_{r,m}}
      \leq \frac{1}{|U|} \cdot \|\Upsilon \ast \Theta\|_{L_{r,m}} \\
    & \leq \frac{1}{|U|}\cdot\max\{\|\Theta\|_{L_{1,w}}, \|\Theta\|_{L_{1,w \Delta^{-1}}}\}
         \cdot \|\Upsilon\|_{L_{r,m}} \, .
  \end{align*}
  In view of \eqref{eq:LittleEllRSpaceIsDiscretization}, this proves
  the boundedness and well-definedness of $\operatorname{Synth}_{X,W}$.
\end{proof}

Now we can prove the desired atomic decomposition result:

\begin{proposition}\label{prop:GoodKernelAtomicDecomposition}
  Let Assumption \ref{assume:GoodKernelAtomicDecomposition} be satisfied.
  For each compact unit neighborhood $U \subset G$ write
  \begin{equation}
    C_U := \max\{
                 \|\operatorname{osc}_U W\|_{L_{1,w}},
                 \|\operatorname{osc}_U W\|_{L_{1,w \Delta^{-1}}}
               \}.
    \label{eq:AtomicDecompositionConstantDefinition}
  \end{equation}
  Assume that
  \begin{equation}
    C_U \cdot \| RC_K \|_{L_{r,m} \to L_{r,m}} < 1 \, .
    \label{eq:AtomicDecompositionOscillationCondition}
  \end{equation}

  Finally, let $X = (x_i)_{i \in I}$ be a relatively separated family for which
  there exists a $U$-BUPU $\Psi = (\psi_i)_{i \in I}$ with
  localizing family $X$,
  and let the weight $m_X$ be as defined in
  Lemma \ref{lem:RelativelySeparatedSynthesis}.

  Then the family $(\lambda(x_i)K)_{i \in I}$ forms a family of atoms for
  $\mathcal{M}_{r,m}$ with associated sequence space $\ell_{r,m_X}(I)$.
  This means:
  \begin{enumerate}[i)]
    \item The synthesis operator
          \[
            \operatorname{Synth}_{X,K} : \ell_{r,m_X}(I) \to \mathcal{M}_{r,m},
                           (c_i)_{i \in I} \mapsto \sum_{i \in I}
                                                      c_i \cdot \lambda(x_i) K
          \]
          is well-defined and bounded, with unconditional convergence of the
          defining series. This even holds without assuming
          \eqref{eq:AtomicDecompositionOscillationCondition}.

    \item There is a bounded \emph{coefficient operator}
          \[
            C : \mathcal{M}_{r,m} \to \ell_{r,m_X}(I)
            \quad \text{with} \quad
            \operatorname{Synth}_{X,K} \circ \, C = \operatorname{id}_{\mathcal{M}_{r,m}} \, .
          \]
  \end{enumerate}
\end{proposition}
\begin{remark}
\begin{enumerate}[i)]
  \item We note that condition \eqref{eq:AtomicDecompositionOscillationCondition}
  is always satisfied for $U$ small enough, thanks to Lemma \ref{lem:osc}
  and Assumption \ref{assume:GoodKernelAtomicDecomposition}.
  \item As in Proposition \ref{prop:GoodKernelBanachFrame}, the action of the
  coefficient operator $C$ is independent of the choice of $r,m$, that is,
  if condition \eqref{eq:AtomicDecompositionOscillationCondition} is satisfied
  for $L_{r_1,m_1}(G)$ and $L_{r_2,m_2}(G)$ and if $C_1 : \mathcal{M}_{r_1,m_1} \to \ell_{r_1,m_{1,X}}(I)$
  and $C_2 : \mathcal{M}_{r_2,m_2} \to \ell_{r_2,m_{2,X}}(I)$ are the respective coefficient
  operators, then $C_1 f = C_2 f$ for all $f \in \mathcal{M}_{r_1,m_1} \cap \mathcal{M}_{r_2,m_2}$.
  \end{enumerate}
\end{remark}

\begin{proof}
  \textbf{Step 1 (Boundedness of the synthesis operator):}
  For this step we will \emph{not} use condition
  \eqref{eq:AtomicDecompositionOscillationCondition}. By Assumption
  \ref{assume:GoodKernelAtomicDecomposition}, 
  $RC_K : L_{r,m}(G) \to L_{r,m}(G)$ is bounded, and
  we have $W \ast K = K$, which implies
  $(\lambda(x) W) \ast K = \lambda(x) (W \ast K) = \lambda(x) K$ for all $x \in G$.

  Furthermore, Lemma \ref{lem:GoodKernelAtomicDecompositionOperators}
  shows that the map
  \[
    \operatorname{Synth}_{X,W} : \ell_{r,m_X}(I) \to L_{r,m}(G),
                   (c_i)_{i \in I} \mapsto \sum_{i \in I} c_i \cdot \lambda(x_i) W
  \]
  is well-defined and bounded. Because of $r < \infty$, each
  $c = (c_i)_{i \in I} \in \ell_{r,m_X}(I)$ satisfies
  $c = \sum_{i \in I} c_i \delta_i$ with unconditional convergence in
  $\ell_{r,m_X}(I)$, where $(\delta_i)_{i \in I}$ denotes the standard basis
  of $\ell_{r,m_X} (I)$. This implies that the series defining $\operatorname{Synth}_{X,W} c$
  converges unconditionally in $L_{r,m}(G)$.
  Since bounded operators preserve unconditional convergence, we see
  \begin{align*}
    RC_K \big( \operatorname{Synth}_{X,W} c \big)
    = RC_K \left(\sum_{i \in I} c_i \, \lambda(x_i) W\right)
    = \sum_{i \in I} c_i \, [(\lambda(x_i) W) \ast K]
    = \sum_{i \in I} c_i \, \lambda(x_i) K 
  \end{align*}
  with unconditional convergence of the series.
  We have thus shown that
  $\operatorname{Synth}_{X,K} = RC_K \circ \operatorname{Synth}_{X,W} : \ell_{r,m_X}(I) \to L_{r,m}(G)$
  is well-defined and bounded, with unconditional convergence of the defining
  series.

  Since $\lambda(x_i) K \in \mathcal{M}_{r,m}$ for all $i \in I$, we also see that
  the range of $\operatorname{Synth}_{X,K}$ is contained in the closed subspace
  $\mathcal{M}_{r,m} \subset L_{r,m}(G)$.

  \medskip{}

  \textbf{Step 2 (An alternative reproducing formula for $\mathcal{M}_{r,m}$):}
  In this step we will prove
  \begin{equation}
    f = (f \ast W) \ast K
    \qquad \text{for all} \, f \in \mathcal{M}_{r,m} \, .
    \label{eq:AtomicDecompositionAlternativeReproducing}
  \end{equation}
  This is almost trivial: For $f \in \mathcal{M}_{r,m}$, we have $f = f \ast K$ by
  definition of $\mathcal{M}_{r,m}$, and we have $K = W \ast K$ by
  Assumption \ref{assume:GoodKernelAtomicDecomposition}.
  By combining these facts, we get $f = f \ast K = f \ast (W \ast K)$.
  Thus, all we need to verify is that the convolution is associative in the
  setting that we consider here.


  In light of \cite[Lemma 6.3]{dadedelastte17} to prove this it remains to show $( (|f| \ast |W|) \ast |K|) (x) < \infty$ for
  almost all $x \in G$.
  To this end, we first show $W \in L_{1,w}(G) \cap L_{1,w \Delta^{-1}}(G)$.
  In order to see this, let $V \subset G$ be any compact unit neighborhood,
  and set $Q := \operatorname{int} V$, so that $U := \overline{Q} \subset V$
  is a compact unit neighborhood that satisfies
  $\overline{\operatorname{int} U} \supset \overline{Q} = U$ and hence $U=\overline{\operatorname{int} U}$.
  As a consequence of this and of the continuity of $W$, as seen in the proof of Lemma \ref{lem:osc}
  (see p. \pageref{proof:OscillationLemmaProof}), we have
  \[
         \widecheck{M}_U^\rho W (x)
    =    \|W\|_{L_\infty(Ux)}
    =    \sup_{y \in Ux} |W(y)|
    \geq |W(x)|
    \quad \text{for all} \, x \in G \, .
  \]
  Since $\widecheck{M}_U^\rho W \in L_{1,w}(G) \cap L_{1, w \Delta^{-1}}(G)$ (see
  Assumption \ref{assume:GoodKernelAtomicDecomposition} and
  Remark \ref{rem:MaximalFunctionNeighborhoodIndependence}), we see
  $W \in L_{1,w}(G) \cap L_{1, w \Delta^{-1}}(G)$.

  Now, fix some $s \in (r, \infty)$ and let $f \in \mathcal{M}_{r,m}$.
  Because of $W \in L_{1,w}(G) \cap L_{1,w \Delta^{-1}}(G)$,
  Proposition \ref{prop:Young} shows $|f| \ast |W| \in L_{r,m}(G)$.
  Therefore, by the second part of Proposition \ref{prop:Young}, we see
  $(|f| \ast |W|) \ast |K| \in L_{s,m}(G)$, since $|K(x^{-1})| = |K(x)|$ and
  since $K \in L_{p,w}(G)$ for all $p \in (1,\infty)$.
  In particular, $((|f| \ast |W|) \ast |K|) (x) < \infty$
  for almost all $x \in G$.  By the considerations from above,
  we thus see that \eqref{eq:AtomicDecompositionAlternativeReproducing}
  holds.

  \medskip{}

  \textbf{Step 3 (Approximating $f \mapsto f \ast W$):}
  Let $\operatorname{Ana}_{X,\Psi} : L_{r,m}(G) \to \ell_{r,m_X}(I)$
  and $\operatorname{Synth}_{X,W} : \ell_{r,m_X}(I) \to L_{r,m}(G)$ be as defined in
  Lemma \ref{lem:GoodKernelAtomicDecompositionOperators}, and define
  $A := \operatorname{Synth}_{X,W} \circ \operatorname{Ana}_{X,\Psi} : L_{r,m}(G) \to L_{r,m}(G)$.
  In this step, we will show
  \begin{equation}
    \|f \ast W - A f\|_{L_{r,m}}
    \leq C_U \cdot \|f\|_{L_{r,m}}
    \qquad \text{for all} \, f \in L_{r,m}(G) \, ,
    \label{eq:AtomicDecompositionApproximateWConvolution}
  \end{equation}
  with $C_U$ as in \eqref{eq:AtomicDecompositionConstantDefinition}.

  To this end, recall from the previous step that
  $W \in L_{1,w}(G) \cap L_{1,w \Delta^{-1}}(G)$, so that Young's inequality
  (Proposition \ref{prop:Young}) shows $|f| \ast |W| \in L_{r,m}(G)$ for
  $f \in L_{r,m}(G)$. In particular, this implies $|f| \ast |W| (x) < \infty$
  for almost all $x \in G$. For each such $x \in G$, the dominated convergence
  theorem and the definition of the operators $\operatorname{Ana}_{X,\Psi}$, $\operatorname{Synth}_{X,W}$
  and $A$ shows
  \begin{align*}
    |f \ast W (x) - A f (x)|
    & = \left|
          \sum_{i \in I} \int_G f(y) \psi_i (y) W(y^{-1} x) \, dy
          - \sum_{i \in I} \langle f, \psi_i \rangle_{L_2} (\lambda(x_i) W)(x)
        \right| \\
    & \leq \sum_{i \in I} \int_G
                            \psi_i (y)
                            \cdot |f(y)|
                            \cdot |W(y^{-1} x) - W(x_i^{-1} x)|
                          \, dy \, .
  \end{align*}
  Fix $i \in I$ for the moment. For $y \in G$ with
  $\psi_i (y) \neq 0$, we have $y = x_i u \in x_i U$, and thus
  $x_i^{-1} x = u y^{-1} x \in U y^{-1} x$.
  Therefore,
  $|W(y^{-1} x) - W(x_i^{-1} x)| \leq (\operatorname{osc}_U W) (y^{-1} x)$,
  by definition of the oscillation
  $\operatorname{osc}_U W$ (see \eqref{oscillation}).

  If we combine this with the estimate from above and with
  the monotone convergence theorem, we get
  \begin{align*}
    |f \ast W (x) - A f (x)|
    & \leq \sum_{i \in I} \int_G
                            \psi_i (y) \cdot |f(y)| \cdot (\operatorname{osc}_U W)(y^{-1} x)
                          \, dy \\
    & = \int_G |f(y)| \cdot (\operatorname{osc}_U W)(y^{-1} x) \, dy
      = (|f| \ast \operatorname{osc}_U W) (x) \, .
  \end{align*}
  In view of Young's inequality (Proposition \ref{prop:Young}) and
  the definition of $C_U$
  (see \eqref{eq:AtomicDecompositionConstantDefinition}), we see
  that \eqref{eq:AtomicDecompositionApproximateWConvolution} holds
  true.

  \medskip{}

  \textbf{Step 4 (Completing the proof):} Recall that
  $RC_K : L_{r,m}(G) \to \mathcal{M}_{r,m}$ is bounded by
  Assumption \ref{assume:GoodKernelAtomicDecomposition}.
  Thus, $B := RC_K \circ A|_{\mathcal{M}_{r,m}} : \mathcal{M}_{r,m} \to \mathcal{M}_{r,m}$
  is well-defined and bounded, with $A$ as in the preceding step.
  Now, for arbitrary $f \in \mathcal{M}_{r,m}$ our results from Steps 2 and 3 show
  \[
    \|f - B f\|_{L_{r,m}}
    = \|RC_K (f \ast W - A f)\|_{L_{r,m}}
    \leq \|RC_K\|_{L_{r,m} \to L_{r,m}} \cdot C_U \cdot \|f\|_{L_{r,m}} \, .
  \]
  In view of our assumption \eqref{eq:AtomicDecompositionOscillationCondition},
  a Neumann series argument (see \cite[Sect.~5.7]{AltFunctionalAnalysis})
  shows that $C_0 := \sum_{n=0}^\infty (\operatorname{id}_{\mathcal{M}_{r,m}} - B)^n$ defines
  a bounded linear operator $C_0 : \mathcal{M}_{r,m} \to \mathcal{M}_{r,m}$ with
  $B \circ C_0 = \operatorname{id}_{\mathcal{M}_{r,m}}$.

  But we saw in Step 1 that $\operatorname{Synth}_{X,K} = RC_K \circ \operatorname{Synth}_{X,W}$, so that
  \begin{align*}
    B &= RC_K \circ A|_{\mathcal{M}_{r,m}}
      = RC_K \circ \operatorname{Synth}_{X,W} \circ \operatorname{Ana}_{X,\Psi} |_{\mathcal{M}_{r,m}}\\
      &= \operatorname{Synth}_{X,K} \circ \operatorname{Ana}_{X,\Psi}|_{\mathcal{M}_{r,m}} \, .
  \end{align*}
  Thus, the operator
  $C := \operatorname{Ana}_{X,\Psi}|_{\mathcal{M}_{r,m}} \circ C_0 : \mathcal{M}_{r,m} \to \ell_{r,m_X}(I)$
  satisfies
  \[
    \operatorname{Synth}_{X,K} \circ \, C = B \circ C_0 = \operatorname{id}_{\mathcal{M}_{r,m}} \, .
  \]

  It is not hard to see that the action of the coefficient operator $C$
  is independent of the choice of $r,m$. For more details see the end of the
  proof of Proposition \ref{prop:GoodKernelBanachFrame}, where a similar
  claim is shown.
\end{proof}

Finally, as in the preceding section we apply the correspondence principle
to obtain atomic decomposition results for the coorbit spaces from those
for the reproducing kernel spaces.

\begin{theorem}\label{thm:GoodKernelAtomicDecompositionCoorbit}
  Under the assumptions of Proposition \ref{prop:GoodKernelAtomicDecomposition},
  the sampled family $(\pi(x_i) u)_{i \in I} \subset \operatorname{Co}(L_{r,m})$ forms a family of atoms for $\operatorname{Co}(L_{r,m})$ with coefficient space
  $\ell_{r,m_X}(I)$.

  More precisely, the synthesis operator
  \[
    \operatorname{Synth}_{X,\mathrm{Co}} :
    \ell_{r,m_X}(I) \to \operatorname{Co}(L_{r,m}),
    (c_i)_{i \in I} \mapsto \sum_{i \in I} c_i \cdot \pi(x_i) u
  \]
  is well-defined and bounded, and there is a bounded linear coefficient
  operator $C_{\mathrm{Co}} : \operatorname{Co}(L_{r,m}) \to \ell_{r,m_X}(I)$ satisfying
  $\operatorname{Synth}_{X,\mathrm{Co}} \circ C_{\mathrm{Co}} = \operatorname{id}_{\operatorname{Co}(L_{r,m})}$.

  Finally, the action of the coefficient operator $C_{\mathrm{Co}}$ is
  independent of the choice of $r,m$. In other words, if the assumptions
  of the current theorem are satisfied for $L_{r_1,m_1}$ and for $L_{r_2,m_2}$ and if
  $C_1, C_2$ denote the corresponding coefficient operators, then
  $C_1 f = C_2 f$ for all $f \in \operatorname{Co}(L_{r_1,m_1}) \cap \operatorname{Co}(L_{r_2,m_2})$.
\end{theorem}

\begin{proof}
  The correspondence principle (Proposition \ref{prop6}) states that the
  extended voice transform $V_e : \operatorname{Co}(L_{r,m}) \to \mathcal{M}_{r,m}$ is an isometric
  isomorphism.
  Furthermore,
  \begin{align*}
    (V_e \, \pi(x) \, u)(y)
    = \langle \pi(x) \, u, \pi(y) \, u \rangle_{\mathcal{S}_w}
    = \langle \pi(x) \, u, \pi(y) \, u \rangle_{\mathcal{H}}
    = K(x^{-1} y)
    = (\lambda(x) K) (y)
  \end{align*}
  for all $x,y \in G$.
  In other words, $V_e \, \pi(x) \, u = \lambda(x) K$ for all $x \in G$.

  Now, since the bounded linear operator $V_e^{-1} : \mathcal{M}_{r,m} \to \operatorname{Co}(L_{r,m})$
  preserves unconditional convergence of series, the synthesis operator
  $\operatorname{Synth}_{X,K}$ from Proposition \ref{prop:GoodKernelAtomicDecomposition}
  satisfies
  \begin{align*}
    (V_e^{-1} \circ \operatorname{Synth}_{X,K}) \,\, (c_i)_{i \in I}
    & = V_e^{-1} \left(\sum_{i \in I} c_i \cdot \lambda(x_i) K \right)
    = \sum_{i \in I} c_i \cdot V_e^{-1}(\lambda(x_i) K) \\
    & = \sum_{i \in I} c_i \cdot \pi(x_i) u
    = \operatorname{Synth}_{X,\mathrm{Co}} \,\, (c_i)_{i \in I} \, ,
  \end{align*}
  for arbitrary $(c_i)_{i \in I} \in \ell_{r,m_X}(I)$, with unconditional
  convergence of all involved series.
  Thus, the operator
  $\operatorname{Synth}_{X,\mathrm{Co}} = V_e^{-1} \circ \operatorname{Synth}_{X,K}
  : \ell_{r,m_X}(I) \to \operatorname{Co}(L_{r,m})$
  is indeed well-defined and bounded.

  Now, with the coefficient operator $C : \mathcal{M}_{r,m} \to \ell_{r,m_X}(I)$
  from Proposition~\ref{prop:GoodKernelAtomicDecomposition}, define
  $C_{\mathrm{Co}} := C \circ V_e : \operatorname{Co}(L_{r,m}) \to \ell_{r,m_X}(I)$.
  Then
  \[
    \operatorname{Synth}_{X,\mathrm{Co}} \circ C_{\mathrm{Co}}
    = V_e^{-1} \circ \operatorname{Synth}_{X,K} \circ C \circ V_e
    = V_e^{-1} \circ V_e
    = \operatorname{id}_{\operatorname{Co}(L_{r,m})} \, ,
  \]
  as desired. Since the action of $C$ is independent of the choice of $r,m$,
  so is the action of $C_{\mathrm{Co}}$.
\end{proof}

\subsection{An Application: Discretization Results for General Paley-Wiener Spaces}
\label{sub:SatisfactoryPaleyWiener}


In this subsection we will apply the abstract results from this section
to the Paley-Wiener spaces $B_\Omega^p$, thereby improving on the discretization
results derived in Sect. \ref{sec:examples}.
Furthermore, our proofs clearly point out those properties that the set
$\Omega \subset \mathbb{R}$ has to satisfy if one wants discretization results to hold
for the associated Paley-Wiener spaces:

\begin{assumption}\label{assume:PaleyWiener}
  Let $\Omega \subset \mathbb{R}$ be measurable, and let
  $r \in (1,\infty) \setminus \{2\}$.
  Assume that the following properties hold:
  \begin{enumerate}[i)]
    \item $\Omega$ is bounded;

    \item the kernel $K := \mathcal{F}^{-1} \chi_{\Omega}$ satisfies
          $K \in \bigcap_{1<p<\infty} L_p (\mathbb{R})$;

    \item the convolution operator
          $RC_K$
          is well-defined on $L_r(\mathbb{R})$.
  \end{enumerate}
\end{assumption}

\begin{remark}
  The last property means that the indicator function $\chi_{\Omega}$
  is an $L_r(\mathbb{R})$-Fourier multiplier, which implies that
  $\chi_{\Omega^c} = 1 - \chi_\Omega$ is an $L_r(\mathbb{R})$-Fourier multiplier
  as well.
  Therefore, \cite[Theorem 1]{LebedevOlevskiiFourierMultiplierIdempotents} shows
  that there is an open set $U \subset \mathbb{R}$ with
  $\chi_{\Omega^c} = \chi_{U}$ almost everywhere, and thus
  $\chi_{U^c} = \chi_{\Omega}$ almost everywhere.
  But since $\Omega$ is bounded, we have $\Omega \subset [-R,R]$ for some
  $R > 0$, and then $\chi_{\Omega} = \chi_{\Omega}\chi_{[-R,R]}
  = \chi_{U^c} \chi_{[-R,R]} = \chi_{[-R,R] \setminus U}$
  almost everywhere.
  Thus, by modifying $\Omega$ on a null set, we can (and will) always
  assume that $\Omega$ is compact. This neither changes the kernel $K$,
  nor the underlying Hilbert space
  \[
    \mathcal{H}
    := B_\Omega^2
    := \{
        f \in L_2 (\mathbb{R})
        \,:\,
        \widehat{f} \equiv 0 \text{ a.e. on } \mathbb{R} \setminus \Omega
      \} \, .
  \]
\end{remark}

As seen in the discussion before Proposition \ref{prop:BadPaleyWienerSet},
if we set $u := K = \mathcal{F}^{-1} \chi_\Omega$, then all standing
assumptions from Sect. \ref{sec:overview} are satisfied for $m = w \equiv 1$,
so that the coorbit spaces $\operatorname{Co}(L_p)$ are well-defined for $1 < p < \infty$.
Furthermore, we saw before Proposition \ref{prop:BadPaleyWienerSet} that
the associated voice transform satisfies $V f = f$ for all
$f \in \mathcal{H} = B_\Omega^2 \subset L_2(\mathbb{R})$.
Using this identity, we can now identify the abstractly defined coorbit spaces
with more concrete reproducing kernel or Paley-Wiener spaces.

\begin{lemma}\label{lem:PaleyWienerSpacesConcrete}
  Setting $\mathcal{T} := \bigcap_{1 < p < \infty} L_p (\mathbb{R})$,
  the space $\mathcal{S}$ from \eqref{eq:43} satisfies
  \[
    \mathcal{S} = \left\{ f \in \mathcal{T} ~\middle|~ f \ast K = f  \right\} \, ,
  \]
  with topology generated by the norms $(\|\cdot\|_{L_p})_{1 < p < \infty}$.

  Furthermore, with $\mathcal{M}_p = \left\{f \in L_p (\mathbb{R}) ~\middle|~ f \ast K = f \right\}$ and
  $\mathcal{M} := \bigcup_{1 < p < \infty} \mathcal{M}_p$, the map
  \[
    \iota : \mathcal{M} \to \mathcal{S}', f \mapsto \Phi_f
    \quad \text{with} \quad
    \langle \Phi_f \,,\, g \rangle_{\mathcal{S}}
    := \langle f, g \rangle_{L_2}
  \]
  is a bijection.
  If we identify each $\varphi \in \mathcal{S} '$ with its inverse
  image $\iota^{-1} \varphi \in \mathcal{M}$ under this map, then the extended voice
  transform is the identity map, that is $V_e \, \varphi = \iota^{-1} \varphi$.

  According to the general result, the coorbit spaces $\operatorname{Co}(L_p)$ are given by
  \begin{equation}
    \operatorname{Co}(L_p) = \iota (\mathcal{M}_p)
    \qquad \text{for all } \, p \in (1,\infty) \, ,
    \label{eq:PaleyWienerCoorbitReproducingKernelIdentification}
  \end{equation}
  which means that if we identify $\varphi$ with $\iota^{-1} \varphi$,
  then $\operatorname{Co}(L_p) = \mathcal{M}_p$.

  Finally for $p \in (1,2]$ we have
  \[
    \mathcal{M}_p
    = B_\Omega^p
    := \left\{
         f \in L_p (\mathbb{R})
         ~\middle|~
         \widehat{f} \equiv 0 \text{ a.e. on } \mathbb{R} \setminus \Omega
        \right\} \, .
  \]
  Therefore, up to canonical identifications, the coorbit spaces $\operatorname{Co}(L_p)$
  coincide with the Paley-Wiener spaces $B_\Omega^p$,
  at least for $p \in (1,2]$.
\end{lemma}

\begin{remark}
  We do not know if in general the identity
  $ \mathcal{M}_p = B_\Omega^p$ with
  \[
    B_\Omega^p = \left\{
                   f \in L_p (\mathbb{R})
                   ~\middle|~
                   \text{the tempered dist. } \widehat{f} \text{ has }
                   \operatorname{supp}(\widehat{f}\,) \subset \Omega
                  \right\}
  \]
  also holds for $p \in (2,\infty)$.
  In case of $\Omega = [-\omega,\omega]$, it was shown in
  \cite[Proposition 4.8]{dadedelastte17} that this is true.
  Using this, one can show $\mathcal{M}_p = B_\Omega^p$ even if $\Omega$ is a finite
  disjoint union of compact intervals.
  For more general sets $\Omega$, however, we do not know whether
  $\mathcal{M}_p = B_\Omega^p$ for $p \in (2,\infty)$.
\end{remark}

\begin{proof}[of Lemma \ref{lem:PaleyWienerSpacesConcrete}]
The following proof is similar to the proof of \cite[Proposition 4.8]{dadedelastte17} 
with some significant improvements and generalizations.

\medskip{}

  The first property is an immediate consequence of the definitions, combined
  with $V f = f$ for $f \in \mathcal{H}$.

  \medskip{}

  The map $\iota$ is indeed well-defined, since if $f \in \mathcal{M}_p$ for some
  $p \in (1,\infty)$, then
  $|\langle f,g \rangle_{L_2}| \leq \|f\|_{L_p} \cdot \|g\|_{L_{p'}}$,
  where $\|\cdot\|_{L_{p'}}$ is a continuous norm on $\mathcal{S}$.

  To prove the surjectivity of $\iota$, we first show that $\mathcal{M}$ is a (complex)
  vector space.
  Since each $\mathcal{M}_p$ is closed under multiplication with complex numbers, we
  only need to show that $\mathcal{M}$ is closed under addition.
  To this end, note for $f \in \mathcal{M}_p$ because of $K \in L_{p'}(\mathbb{R})$ that
  \[
    |f(x)|
    = |(f \ast K)(x)|
    = |\langle f, \lambda(x) K \rangle_{L_2}|
    \leq \|f\|_{L_p} \cdot \|\lambda(x) K\|_{L_{p'}}
    \leq C_p \cdot \|f\|_{L_p}
  \]
  for all $x\in\mathbb{R}$,
  and thus $\mathcal{M}_p \hookrightarrow L_\infty$.
  This embedding implies $\mathcal{M}_p \subset \mathcal{M}_q$ for $p \leq q$, and thus
  $\mathcal{M}_p + \mathcal{M}_q \subset \mathcal{M}_q + \mathcal{M}_q = \mathcal{M}_q \subset \mathcal{M}$.
  From this, we easily see that $\mathcal{M}$ is a vector space.

  With $\mathcal{M}$ being a vector space, we see $\mathcal{M} = \operatorname{span} \bigcup_{1<p<\infty} \mathcal{M}_p$.
  With notation as in \eqref{eq:38}, this means $\mathcal{M} = \mathcal{M}^{\mathcal{U}}$.
  Hence, Theorem \ref{intersections-1} shows for arbitrary $\varphi \in \mathcal{S} '$
  that $f := V_e \varphi \in \mathcal{M}^{\mathcal{U}} = \mathcal{M}$, and
  \eqref{eq:ExtendedVoiceTransformDuality} shows because of $V g = g$ for
  $g \in \mathcal{S} \subset \mathcal{H}$ that
  \[
    \langle \Phi_f \,,\, g \rangle_{\mathcal{S}}
    = \langle f, g \rangle_{L_2}
    = \langle V_e \varphi \,,\, Vg \rangle_{L_2}
    = \langle \varphi, g \rangle_{\mathcal{S}} \, .
  \]
  Hence, $\varphi = \Phi_f = \iota f = \iota V_e \varphi$.
  On the one hand, this shows that $\iota$ is surjective, and on the
  other hand---once we know that $\iota$ is bijective---it proves
  that the inverse of $\iota$ is given by $\iota^{-1} = V_e : \mathcal{S}' \to \mathcal{M}$.

  In order to prove that $\iota$ is injective, note $\lambda(x) K \in \mathcal{S}$ for all
  $x \in \mathbb{R}$ and recall $\overline{K(x)} = K(-x)$. Hence,
  \[
    \langle \Phi_f \,,\, \lambda(x) K \rangle_{\mathcal{S}}
    = \langle f, \lambda(x) K \rangle_{L_2}
    = (f \ast K)(x)
    = f(x)
    \quad \text{for} \quad f \in \mathcal{M} \, .
  \]
  Therefore, if $\Phi_f = 0$, then $f = 0$ as well.
  Since the domain $\mathcal{M}$ of $\iota$ is a vector space, this shows
  that $\iota$ is injective.

  \medskip{}

  Equation \eqref{eq:PaleyWienerCoorbitReproducingKernelIdentification}
  is seen to be true by combining the identity $V_e = \iota^{-1}$ with the
  correspondence principle (see Proposition \ref{prop6}),
  which states that
  $V_e : \operatorname{Co}(L_p) \to \left\{f \in L_p(\mathbb{R}) ~\middle|~ f \ast K = f \right\} = \mathcal{M}_p$
  is an isomorphism.

  \medskip{}

  To prove $\mathcal{M}_p = B_\Omega^p$ for $p \in (1,2]$, first note
  $\mathcal{F}(f \ast g) = \widehat{f} \cdot \widehat{g}$ for arbitrary $f,g \in L_2$, see e.g. \cite[p. 270]{sch66}.
  Therefore, for $f \in \mathcal{M}_p \hookrightarrow \mathcal{M}_2$
  (here we used that $p \leq 2$) we see that
  $\widehat{f} = \widehat{f \ast K} = \widehat{f} \cdot \widehat{K}
  = \widehat{f} \cdot \chi_\Omega$, where the equality holds in the sense
  of tempered distributions. But since both sides are $L_2(\mathbb{R})$ functions, this
  implies $\widehat{f} = \widehat{f} \cdot \chi_\Omega$ almost everywhere,
  and thus $f \in B_\Omega^p$.

  Conversely, let $f \in B_\Omega^p$ be arbitrary.
  Because of $p \leq 2$,
  \cite[Theorem in Sect. 1.4.1]{TriebelTheoryOfFunctionSpaces}
  shows $f \in L_2(\mathbb{R})$. Furthermore, since $\widehat{f} \equiv 0$ almost
  everywhere on $\mathbb{R} \setminus \Omega$, we have
  $\mathcal{F}(f \ast K) = \widehat{f} \cdot \widehat{K}
  = \widehat{f} \cdot \chi_\Omega = \widehat{f}$, and thus $f = f \ast K$,
  so that $f \in \mathcal{M}_p$.
\end{proof}

With Lemma \ref{lem:PaleyWienerSpacesConcrete} showing that the coorbit spaces
$\operatorname{Co}(L_p)$ coincide with the reproducing kernel spaces $\mathcal{M}_p$, we will in the
following concentrate on the latter spaces for proving discretization results.

In Sects. \ref{sub:GoodKernelBanachFrames} and
\ref{sub:GoodKernelAtomicDecompositions}, we showed that the sampled frame
$(\pi(x_i) \, u)_{i \in I}$ forms a Banach frame or an atomic decomposition
for the coorbit space $\operatorname{Co}(L_{r,m})$ \emph{if the family of sampling points
$(x_i)_{i \in I}$ is sufficiently dense in $G$}.
For the case of the Paley-Wiener spaces, one can state quite
precisely how dense the sampling points need to be:

\begin{proposition}\label{prop:PaleyWienerDiscretization}
  Suppose that Assumption \ref{assume:PaleyWiener} is satisfied, and choose
  $R>0$ and $\xi_0 \in \mathbb{R}$ with $\Omega \subset \xi_0 + [-R,R]$.

  Then the family $(\lambda(k/(2R)) K)_{k \in \mathbb{Z}}$ forms a Banach frame and an
  atomic decomposition for the reproducing kernel space
  $\mathcal{M}_r$
  with coefficient space $\ell_r (\mathbb{Z})$.
  More precisely, the operators
  \[
    \operatorname{Samp} : \mathcal{M}_r \to \ell_r (\mathbb{Z}),
    f \mapsto \big(f (k/(2R))\big)_{k \in \mathbb{Z}}
              = \big(\langle f, \lambda(k/(2R)) K \rangle_{L_2}\big)_{k \in \mathbb{Z}}
  \]
  and
  \[
    \operatorname{Synth} : \ell_r (\mathbb{Z}) \to \mathcal{M}_r,
    (c_k)_{k \in \mathbb{Z}} \mapsto \sum_{k \in \mathbb{Z}} c_k \cdot \lambda(k/(2R)) K
  \]
  are well-defined and bounded with
  $\operatorname{Synth} \circ \operatorname{Samp} = (2R)^{-1} \cdot \operatorname{id}_{\mathcal{M}_r}$.
\end{proposition}

\begin{proof}
  Since $\Omega \subset \mathbb{R}$ is bounded, there is a Schwartz function
  $\psi \in \mathbb{S}(\mathbb{R})$ with $\psi \equiv 1$ on $\Omega$.
  We then have $W := \mathcal{F}^{-1} \psi \in \mathbb{S}(\mathbb{R})$,
  so that $W$ is continuous.
  Furthermore,
  \[
    \widehat{W \ast K}
    = \widehat{K \ast W}
    = \widehat{K} \cdot \widehat{W}
    = \chi_\Omega \cdot \psi
    = \chi_\Omega
    = \widehat{K} \, ,
  \]
  and hence $W \ast K = K \ast W = K$.
  Since $W$ is a Schwartz function, there is some $C > 0$ with
  $|W(x)| \leq C \cdot (1+|x|)^{-2}$
  for all $x \in \mathbb{R}$.
  Because of
  \[
    1 + |x| \leq 2 + |x-y| \leq 2 \cdot (1+|x-y|)
  \]
  for any $\vphantom{\sum_j} y \in Q := U_0 := [-1,1]$, this shows
  $|W(x+y)| \leq 4C \cdot (1+|x|)^{-2}$,
  and hence $\widecheck{M}_Q^\rho W \in L_1(\mathbb{R})$ and $M_{U_0}^\lambda W \in L_1 (\mathbb{R})$.
  Overall, noting that the modular function $\Delta$ of the abelian
  group $G = \mathbb{R}$ satisfies $\Delta \equiv 1$, we see using Lemma~\ref{lem:BanachFrameKernelAssumption} that
  Assumptions \ref{assume:GoodKernelBanachFrameAssumption} and
  \ref{assume:GoodKernelAtomicDecomposition} are both satisfied for
  $w \equiv m \equiv 1$.

  Now, define $I := \mathbb{Z}$ and $x_k := k/(2R)$ for $k \in \mathbb{Z}$. It is not hard to
  see that the family $(x_k)_{k \in \mathbb{Z}}$ is relatively separated in $G = \mathbb{R}$.
  Therefore, Lemma~\ref{lem:BanachFrameSamplingOperator} and
  Proposition~\ref{prop:GoodKernelAtomicDecomposition}
  show that the two operators from the statement of the current proposition
  are well-defined and bounded. It remains to show
  $\operatorname{Synth} \circ \operatorname{Samp} = (2R)^{-1} \operatorname{id}_{\mathcal{M}_r}$.

  For this, it suffices to show $\operatorname{Synth} (\operatorname{Samp} f) = (2R)^{-1} \cdot f$
  for $f \in \mathcal{M}_r \cap L_2(\mathbb{R})$, since Lemma \ref{lem:kernel_continuity} shows
  that $\operatorname{span} \{ \lambda(x) K  \}_{x \in \mathbb{R}} \subset \mathcal{M}_r \cap L_2 (\mathbb{R})$ is dense
  in $\mathcal{M}_r$.
  But it is well-known that the family $(e_k)_{k \in \mathbb{Z}}
  = \big( (2R)^{-1/2} \cdot e^{2\pi i \frac{k}{2R} \cdot}\, \big)_{k \in\mathbb{Z}}$
  forms an orthonormal basis of $L_2 (\Omega_0)$ where
  $\Omega_0 := \xi_0 + [-R,R]$.
  To make use of this orthonormal basis, first note for
  $f \in \mathcal{M}_r \cap L_2 (\mathbb{R})$
  that $\widehat{f} = \widehat{f \ast K} = \widehat{f} \cdot \widehat{K}
  = \chi_\Omega \cdot \widehat{f}$.
  Because of $\widehat{f} = \widehat{f} \cdot \chi_\Omega$, we get
  $\widehat{f} \equiv 0$ almost everywhere on
  $\mathbb{R} \setminus \Omega \supset \mathbb{R} \setminus \Omega_0$.

  Overall, since $\mathcal{F}(\lambda(k/(2R)) K)
  = e^{-2\pi i \frac{k}{2R} \cdot} \, \chi_\Omega
  = (2R)^{1/2} \cdot e_{-k} \cdot \chi_\Omega$, we see
  \begin{align*}
    \widehat{f}
    & = \chi_\Omega \cdot \widehat{f}
      = \chi_\Omega \cdot
        \sum_{k \in \mathbb{Z}} \langle \widehat{f} , e_k \rangle_{L_2} \, e_k \\
     &= (2R)^{-1} \cdot \sum_{k \in \mathbb{Z}}
                          \Big\langle
                            \widehat{f} , \mathcal{F}\big(\lambda(-k/(2R)) K\big)
                          \Big\rangle_{L_2}
                          \cdot \mathcal{F}(\lambda(-k/(2R)) \, K) \\
    & = (2R)^{-1} \cdot \mathcal{F} \left(
                              \sum_{\ell \in \mathbb{Z}}
                                \langle f, \lambda(\ell/(2R)) K \rangle_{L_2}
                                \cdot \lambda(\ell/(2R)) \, K
                            \right) \\
    &= (2R)^{-1} \cdot \mathcal{F} \left(\operatorname{Synth} (\operatorname{Samp} f)\right) \, ,
  \end{align*}
  which implies $f = (2R)^{-1} \cdot (\operatorname{Synth} \circ \operatorname{Samp}) f$ for all
  $f \in L_2(\mathbb{R}) \cap \mathcal{M}_r$, as desired.
\end{proof}

To close this section, we show that the existence of
a ``well-behaved'' kernel $W$ with $K \ast W = K$ is \emph{independent}
of the property that $K$ acts boundedly on $L_r(\mathbb{R})$ via right convolutions,
even when we restrict to the class of reproducing kernels $K$ which satisfy
the weak integrability property $K \in \bigcap_{1<p<\infty} L_p (\mathbb{R})$.
In the proof of Proposition \ref{prop:PaleyWienerDiscretization}, we saw
that for every \emph{bounded} set $\Omega \subset \mathbb{R}$, there is a such
a well-behaved kernel $W$ associated to the reproducing kernel
$K = \mathcal{F}^{-1} \chi_\Omega$.
But the set $C \subset [0,1]$ that we constructed in
Proposition \ref{prop:BadPaleyWienerSet} is bounded and the associated
kernel $K = \mathcal{F}^{-1} \chi_C$ satisfies the weak integrability property.
Still, $K$ does not act boundedly via right convolution on \emph{any} $L_p(\mathbb{R})$
space with $p \neq 2$.
Conversely, the following proposition shows the existence of a kernel
$K$ that acts boundedly via right convolution on all $L_p$ spaces
for $1 < p < \infty$, but for which no well-behaved kernel $W$ with
$K = W \ast K$ exists.

\begin{proposition}\label{prop:UnboundedSemiNicePaleyWienerSet}
  The set $\Omega := \bigcup_{j=1}^\infty [3 \cdot 2^{j-2} + (0, 2^{-2j})]$
  with the associated kernel $K := \mathcal{F}^{-1} \chi_\Omega$
  has the following properties:
  \begin{enumerate}[i)]
    \item $K \in \bigcap_{1 < p < \infty} L_p (\mathbb{R})$.

    \item There is no $W \in L_1 (\mathbb{R})$ with $K = K \ast W$.

    \item The operator $RC_K$ is bounded on $L_p(\mathbb{R})$ for every $p\in(1,\infty)$.
  \end{enumerate}
\end{proposition}

\begin{proof}
We first verify that the union defining $\Omega$ is indeed disjoint.
To this end, set $I_j := 3 \cdot 2^{j-2} + (0,2^{-2j})$ for $j \in \mathbb{N}$,
and note $3j - 2 > 0$, so that $2^{3j - 2} > 1$, and hence $2^{-2j} < 2^{j-2}$.
This implies
\begin{equation}
  2^{j-1}
  = 2 \cdot 2^{j-2}
  < 3 \cdot 2^{j-2}
  < 3 \cdot 2^{j-2} + 2^{-2j}
  < 3 \cdot 2^{j-2} + 2^{j-2}
  = 2^{j} \, .
  \label{eq:UnboundedPaleyWienerSetComponentsAreDyadic}
\end{equation}
Therefore, $I_j \subset (2^{j-1}, 2^j)$, which easily yields the desired
disjointness. Next, we verify the three claimed properties.

\medskip{}

\textbf{First property:} A direct computation shows that
$F := \mathcal{F}^{-1} \chi_{(0,1)}$ satisfies
$F(x) = \frac{e^{2\pi i x} - 1}{2\pi i x}$ for $x \neq 0$, and hence
$F \in \bigcap_{1<p<\infty} L_p (\mathbb{R})$.
Since $\chi_{I_j}
= \lambda(3 \cdot 2^{j-2}) \big( \chi_{(0,1)} (2^{2j} \cdot) \big)$,
we thus see by elementary properties of the Fourier transform that
$\mathcal{F}^{-1} \chi_{I_j}(x)
= 2^{-2j} \cdot e^{6 \pi i 2^{j-2} x} \cdot F(2^{-2j} x)$.
Therefore,
\begin{align*}
  \| K \|_{L_p}
  = \left\|
      \vphantom{\sum} \smash{\sum_{j=1}^\infty} \mathcal{F}^{-1} \chi_{I_j}
    \right\|_{L_p}
  \leq \sum_{j=1}^\infty 2^{-2j} \cdot \|F(2^{-2j} \cdot)\|_{L_p}
  = \|F\|_{L_p} \cdot \sum_{j=1}^\infty 2^{-2j (1 - p^{-1})}
  < \infty 
\end{align*}
for arbitrary $p \in (1,\infty)$.

\medskip{}

\textbf{Second property:} Assume towards a contradiction that $K = K \ast W$
for some $W \in L_1 (\mathbb{R})$.
This implies $\chi_\Omega = \widehat{K} = \widehat{K} \cdot \widehat{W}
= \chi_\Omega \cdot \widehat{W}$ almost everywhere.
In particular, there is a null-set $N \subset \mathbb{R}$ with $\widehat{W}(\xi) = 1$
for all $\xi \in \Omega \setminus N$.
But the Riemann-Lebesgue lemma shows
$\lim_{\xi\to\infty} \widehat{W}(\xi) = 0$,
so that $|\widehat{W}(\xi)| \leq 1/2$ for all $\xi \in \mathbb{R}$ with
$|\xi| \geq 2^{j_0-2}$, for a suitable $j_0 \in \mathbb{N}$.
Hence, for any $\xi$ belonging to the positive measure set $I_{j_0} \setminus N
= (3 \cdot 2^{j_0-2}, 3 \cdot 2^{j_0 - 2} + 2^{-2 j_0}) \setminus N
\subset \Omega \setminus N$, we have $1 = |\widehat{W}(\xi)| \leq 1/2$,
a contradiction.

\medskip{}

\textbf{Third property:} Here, we will use the
\emph{strong Marcinkiewicz multiplier theorem} which states the following:
\smallskip{}

\noindent
\textbf{Strong Marcinkiewicz multiplier theorem}
  (see \cite[Theorem 8.3.1]{GaudryEdwardsLittlewoodPaley})
  {\itshape Let $(\Delta_j)_{j \in \mathbb{Z}}$ denote the usual dyadic decomposition of
  $\mathbb{R}$,
  \[
    \Delta_j
    := \begin{cases}
        [2^{j-1}, 2^j),           & \text{if } j > 0 \, , \\
        (-1,1),                   & \text{if } j = 0 \, , \\
        (-2^{|j|}, -2^{|j| - 1}], & \text{if } j < 0 \, .
      \end{cases}
  \]
  Assume that $\phi : \mathbb{R} \to \mathbb{C}$ is measurable and satisfies
  \[
    \sup_{\xi \in \mathbb{R}} |\phi(\xi)| < \infty
    \quad \text{and} \quad
    \sup_{j \in \mathbb{Z}} \mathrm{Var}_{\Delta_j} \phi < \infty \, ,
  \]
  where $\mathrm{Var}_I \phi$ denotes the total variation of the function $\phi$ when
  restricted to the interval $I$.

  Then $\phi$ is an $L_p(\mathbb{R})$-Fourier multiplier for all $p \in (1,\infty)$.
  In other words, the map
  $\mathbb{S}(\mathbb{R}) \to \mathbb{S}'(\mathbb{R}), f \mapsto \mathcal{F}^{-1}(\widehat{f} \cdot \phi)$
  extends to a bounded linear operator on $L_p (\mathbb{R})$, for any
  $p \in (1,\infty)$.\smallskip{}}

We want to apply this theorem for $\phi := \chi_\Omega$.
To this end, first note $\sup_{\xi \in \mathbb{R}} |\phi(\xi)| = 1 < \infty$.
Second, \eqref{eq:UnboundedPaleyWienerSetComponentsAreDyadic} shows
for $j \in \mathbb{Z}$ with $j \leq 0$ that $\phi|_{\Delta_j} \equiv 0$, and for
$j \in \mathbb{N}$ that $\phi|_{\Delta_j} = \chi_{I_j}$ is the indicator function
of an interval. In both cases, $\mathrm{Var}_{\Delta_j} \phi \leq 2$.
All in all, the strong Marcinkiewicz multiplier theorem shows that the map
$\mathbb{S}(\mathbb{R}) \to \mathbb{S}'(\mathbb{R}), f \mapsto \mathcal{F}^{-1} (\phi \cdot \widehat{f} \, )
= f \ast K$ extends to a bounded linear operator on $L_p (\mathbb{R})$ for
any $p \in (1,\infty)$.
Finally, since $K \in \bigcap_{1<p<\infty} L_p(\mathbb{R})$, Young's inequality
(Proposition \ref{prop:Young}) shows that
$L_p(\mathbb{R}) \to L_q(\mathbb{R}), f \mapsto f \ast K$ is well-defined and bounded for any
$q \in (p,\infty)$.
Therefore, the extended map is still given by $L_p(\mathbb{R}) \to L_p(\mathbb{R}), f \mapsto f \ast K$.
\end{proof}

\appendix

\setcounter{section}{0}
\renewcommand{\thesection}{\Alph{section}}
\section{}

In this appendix we provide proofs for several technical auxiliary results
that we used above.
We first present some weighted versions of well-known facts for the reader's convenience.

The first lemma is a weighted version of Schur's test.



\begin{lemma}[Schur's test]\label{lem:Schur}
Let $K : G \times G \to \mathbb{C}$ be measurable, let $w > 0$ denote
a weight on $G$, and let $p,q,r \in [1,\infty]$ with $1 + 1/p = 1/q + 1/r$.
Assume that there is a constant $C_K > 0$ such that
\begin{align}
  &\Big\| K(x,\cdot) \cdot \frac{w(x)}{w} \Big\|_{L_r} \leq C_K
  \quad \text{for a.e. } x \in G, \label{eq:SchurAssumption1}\\
  &\Big\| K(\cdot,y) \cdot \frac{w}{w(y)} \Big\|_{L_r} \leq C_K
  \quad \text{for a.e. } y\in G. \label{eq:SchurAssumption2}
\end{align}
If $f \in L_{q,w}(G)$, then the integral
\begin{align*}
  I_f(x) = \int_G K(x,y) f(y) ~dy
\end{align*}
converges for a.e. $x \in G$. The function $I_f$ is in $L_{p,w}(G)$ and fulfills
\begin{align*}
  \lVert{I_f}\rVert_{L_{p,w}} \leq C_K \lVert{f}\rVert_{L_{q,w}}.
\end{align*}
\end{lemma}

\begin{proof}
It suffices to assume $f \geq 0$ and $K \geq 0$.
Indeed, temporarily writing $I_{K,f}$ instead of $I_f$ to emphasize the role
of the kernel $K$, we have $|I_{K,f}| \leq I_{|K|,|f|}$;
furthermore, if \eqref{eq:SchurAssumption1} and \eqref{eq:SchurAssumption2}
hold for $K$, then they also hold for $|K|$, with the same constants,
and we have $\|f\|_{L_{q,w}} = \|\, |f| \,\|_{L_{q,w}}$.
Hence, if the claim holds for $K,f \geq 0$, then also
\[
  \|I_{K,f}\|_{L_{p,w}}
  \leq \|I_{|K|, |f|}\|_{L_{p,w}}
  \leq C_{|K|} \cdot \|\, |f| \,\|_{L_{q,w}}
  = C_K \cdot \|f\|_{L_{q,w}} \, .
\]

Thus, we will assume in the following that $K,f \geq 0$.
Hence, also $I_f \geq 0$, so that \cite[Theorem 6.14]{fol99} shows
\begin{equation}
  \lVert{I_f}\rVert_{L_{p,w}}
  = \sup_{0 \leq h \in L_{p',w^{-1}}(G) \setminus \{0\}}
      \frac{\langle{I_f},{h}\rangle_{L_2}}{\lVert{h}\rVert_{L_{p',w^{-1}}}}.
  \label{eq:SchursTestDuality}
\end{equation}
We denote with $d(x,y)$ the product measure on $G\times G$.
Furthermore, for brevity we set $M_x (y) := \frac{w(x)}{w(y)} \cdot K(x,y)$
and observe $\|M_x\|_{L_r} \leq C_K$ for almost all $x \in G$, thanks to
\eqref{eq:SchurAssumption1}.
Likewise, \eqref{eq:SchurAssumption2} shows
$\|M^{(y)}\|_{L_r} \leq C_K$ for almost all $y \in G$,
where $M^{(y)} (x) := \frac{w(x)}{w(y)} \cdot K(x,y)$.

\medskip{}

We first consider a number of special cases, so that we can then concentrate
on the case where $p,q,r \in (1,\infty)$.


\textbf{Case 1:} At least one of $p,q,r$ is infinite.
In case of $p < \infty$, we have $1 < 1 + p^{-1} = q^{-1} + r^{-1}$.
But if $q=\infty$, then the right-hand side of this inequality is
$r^{-1} \leq 1$, which leads to a contradiction.
Similarly, we see that $r=\infty$ leads to a contradiction.
Therefore, we necessarily have $p = \infty$ in the present case.

Because of $1 = 1 + p^{-1} = q^{-1} + r^{-1}$, this implies $q = r'$, and hence
\[
  w(x) \cdot I_f (x)
  = \int_G M_x (y) \cdot w(y) \cdot f(y) \, dy
  \leq \|M_x\|_{L_r} \cdot \|f\|_{L_{r',w}}
  \leq C_K \cdot \|f\|_{L_{q,w}} < \infty
\]
for almost all $x \in G$, proving the claim in Case 1, since $p=\infty$.

\medskip{}

\textbf{Case 2:} We have $p,q,r < \infty$, but at least one of $p,q,r$ is equal
to one. This leaves three subcases:

\textbf{Case 2-A:} We have $p=1$, and hence
$2 = 1 + p^{-1} = q^{-1} + r^{-1} \leq 2$, which implies $q=r=1$.
Hence, by Fubini's theorem,
\begin{align*}
  \|I_f\|_{L_{p,w}}
  & = \int_G w(x) \int_G K(x,y) \cdot f(y) \, dy \, dx
    = \int_G w(y) \cdot f(y) \cdot \int_G M_x (y) \, dx \, dy \\
  & \leq C_K \cdot \|f\|_{L_{1,w}}
    =    C_K \cdot \|f\|_{L_{q,w}} \, ,
\end{align*}
which proves the claim in Case 2-A.

\textbf{Case 2-B:} We have $p \in (1,\infty)$, but $r = 1$.
Since $1 + p^{-1} = q^{-1} + r^{-1} = 1 + q^{-1}$,
this implies $p=q \in (1,\infty)$.
Hence, for each nonnegative $h \in L_{p',w^{-1}}(G)\setminus \{0\}$,
Fubini's theorem and H\"older's inequality show
\begin{align*}
   \langle I_f , h \rangle_{_2}
    &= \int_G h(x) \int_G K(x,y) \, f(y) \, dy \, dx \\
  & = \! \int_{G\times G} \!
        \frac{h(x)}{w(x)}
        \cdot [M^{(y)}(x)]^{\frac{1}{p'}} \,
        [M^{(y)} (x)]^{\frac{1}{p}}
        \cdot w(y) \, f(y)
      \, d(x,y) \\
  & \leq \left(
           \int_G \left(\frac{h(x)}{w(x)}\right)^{p'} \int_G M^{(y)}(x) \, dy \, dx
         \right)^{1/p'} \\
         &\hspace{1cm}\cdot \left(
                 \int_G (w(y) \, f(y))^p \int_G M^{(y)} (x) \, dx \, dy
               \right)^{1/p}
         \\
  & \leq C_K \cdot \|h\|_{L_{p',w^{-1}}} \cdot \|f\|_{L_{p,w}} \, .
\end{align*}
In view of \eqref{eq:SchursTestDuality} and because of $p=q$,
this proves the claim in Case 2-B.

\textbf{Case 2-C:} We have $p, r \in (1,\infty)$, but $q = 1$.
This implies $p=r \in (1,\infty)$,
since $1 + p^{-1} = q^{-1} + r^{-1} = 1 + r^{-1}$,
For nonnegative $h \in L_{p', w^{-1}}(G) = L_{r',w^{-1}}(G)$, we thus have
\begin{align*}
  \langle I_f, h \rangle_{L_2}
  & = \int_G w(y) \cdot f(y) \int_G M_x (y) \cdot \frac{h(x)}{w(x)} \, dx \, dy \\
  &\leq \int_G
           w(y) \cdot f(y) \cdot \|M_x\|_{L_p} \cdot \|h\|_{L_{p',w^{-1}}}
         \, dy \\
  & \leq C_K \cdot \|h\|_{L_{p',w^{-1}}} \cdot \|f\|_{L_{1,w}}
    =    C_K \cdot \|h\|_{L_{p',w^{-1}}} \cdot \|f\|_{L_{q,w}} \, .
\end{align*}
In view of \eqref{eq:SchursTestDuality},
this proves the claim in Case 2-C.

\medskip{}

Finally, we handle the case $p,q,r \in (1,\infty)$.
By elementary calculations one can show
$r/p + r/q' = q/p + q/r' = p'/q' + p'/r' = 1$, where all occurring numbers
$\frac{r}{p}, \frac{r}{q'}$ and so on are elements of the interval $(0,1)$.
Thus, for any $0 \leq h \in L_{p',w^{-1}}(G)$,
it follows from H\"older's inequality and Fubini's theorem that
\begin{align*}
   \langle{I_f},{h}\rangle_{L_2}
  & = \int_{G\times G}
          K(x,y) \frac{w(x)}{w(y)}
          \cdot f(y)w(y)
          \cdot h(x)w(x)^{-1}
        ~d(x,y) \\
  &= \int_{G\times G}
        \big( M^{(y)}(x) \big)^{r/p}
        \cdot \big(f(y)w(y)\big)^{q/p}
        \cdot \big( M_x (y) \big)^{r/q'} \\ &\hspace{1cm}
        \cdot \big(h(x)w(x)^{-1}\big)^{p'/q'} 
        \cdot \big(f(y)w(y)\big)^{q/r'} 
        \cdot \big(h(x)w(x)^{-1}\big)^{p'/r'}
     ~d(x,y) \\
  &\overset{(\ast)}{\leq}
        \left(
          \int_G
            |f(y)w(y)|^q
            \int_G
              \big( M^{(y)}(x) \big)^r
            ~d x
          ~d y
        \right)^{1/p} \\
        &\hspace{1cm} \cdot \left(
                        \int_G
                          |h(x)w(x)^{-1}|^{p'}
                          \int_G
                            \big( M_x (y) \big)^{r}
                          ~d y
                        ~d x
                      \right)^{1/q'} \\
  &\hspace{1cm} \cdot \left(
                        \int_{G\times G}
                          |f(y)w(y)|^q
                          |h(x)w(x)^{-1}|^{p'}
                        ~d(x,y)
                      \right)^{1/r'} \\
  &\leq C_K \cdot \lVert{f}\rVert_{L_{q,w}} \cdot \lVert{h}\rVert_{L_{p',w^{-1}}} < \infty,
\end{align*}
where the step marked with $(\ast)$ used
$\frac{1}{p} + \frac{1}{q'} + \frac{1}{r'}
= \frac{1}{p} + 1 - \frac{1}{q} + 1 - \frac{1}{r} = 1$.
In view of \eqref{eq:SchursTestDuality}, this proves the claim
for the case $p,q,r \in (1,\infty)$.
%
\end{proof}

Next we present a weighted version of the classical Young's inequality.

\begin{proposition}[Young's inequality] \label{prop:Young}
Let $m$ be a $w$-moderate weight on $G$, see \eqref{eq:ModerateWeight},
and let $p,q,r \in [1,\infty]$ such that
$1 + 1/p = 1/q + 1/r$. Then it follows for $f \in L_{q,m}(G)$ and
$g \in L_{r,w}(G) \cap L_{r,w\Delta^{-1/r}}(G)$ that $f \ast g \in L_{p,m}(G)$ and
\begin{align} \label{eq:young1}
  \lVert{f\ast g}\rVert_{L_{p,m}}
  \leq \max\{
             \lVert{g}\rVert_{L_{r,w}},
             \lVert{g}\rVert_{L_{r,w\Delta^{-1/r}}}
           \}
       \cdot \lVert{f}\rVert_{L_{q,m}}.
\end{align}
If, instead of $g \in L_{r,w}(G) \cap L_{r, w \Delta^{-1/r}}(G)$,
it holds $g \in L_{r,w}(G)$ and $|g(x)| = |g(x^{-1})|$
as well as $w(x) = w(x^{-1})$ for all $x \in G$,
or if $g \in L_{r,w}(G)$ and $G$ is unimodular, then
\begin{align} \label{eq:young2}
  \lVert{f \ast g}\rVert_{L_{p,m}} \leq \lVert{g}\rVert_{L_{r,w}} \cdot \lVert{f}\rVert_{L_{q,m}}.
\end{align}
\end{proposition}

\begin{proof}
We apply Lemma \ref{lem:Schur} for the case $K(x,y) = g(y^{-1}x)$ and the weight $m$.
It suffices to show that there exists a constant $C_K$ that fulfills
\eqref{eq:SchurAssumption1} and \eqref{eq:SchurAssumption2}.
We first consider the case $r < \infty$ and use \eqref{eq:ModerateWeight}
and the left invariance of the Haar measure to conclude
\begin{align*}
        \int_G |g(y^{-1}x)|^r \cdot \frac{m(x)^r}{m(y)^r} ~dx
  &=    \int_G |g(z)|^r \cdot \frac{m(yz)^r}{m(y)^r} ~dz 
  \leq \int_G |g(z)|^r \cdot \frac{m(y)^r w(z)^r}{m(y)^r}~dz \\
  &=    \int_G |g(z)|^r \cdot w(z)^r~dz 
   =    \lVert{g}\rVert_{L_{r,w}}^r
\end{align*}
for almost all $y \in G$. Now, using the change of variables
$z = x^{-1} y$, and recalling the formula
$d \varrho(x) = \Delta(x^{-1}) dx$ (see \cite[Proposition~2.31]{fol95})
for the right Haar measure $\varrho$ given by $\varrho(M) = \beta(M^{-1})$,
we see
\begin{align*}
  \int_G |g(y^{-1}x)|^r \cdot \frac{m(x)^r}{m(y)^r} ~dy
  &= \int_G |g(z^{-1})|^r \cdot \frac{m(x)^r}{m(x z)^r} ~dz 
  \leq \int_G
          |g(z^{-1})|^r \cdot [w(z^{-1})]^r
        ~dz \\
  &= \int_G |g(y)|^r \cdot [w(y)]^r \cdot \Delta(y)^{-1} ~dy
   = \lVert{g}\rVert_{L_{r, w\Delta^{-1/r}}}^r
\end{align*}
for almost all $x \in G$. By setting
$C_K = \max\{\lVert{g}\rVert_{L_{r,w}},\lVert{g}\rVert_{L_{r,w\Delta^{-1/r}}}\}<\infty$,
Lemma~\ref{lem:Schur} yields
\begin{align*}
  \lVert{f \ast g}\rVert_{L_{p,m}} \leq C_K \cdot \lVert{f}\rVert_{L_{q,m}}
  \qquad \text{for all} \, f \in L_{q,m}(G) \, ,
\end{align*}
which proves \eqref{eq:young1}.

Finally, for the case $r = \infty$, observe
$m(x) = m(y y^{-1} x) \leq m(y) \cdot w(y^{-1} x)$, so that we get
\[
  |g(y^{-1} x)| \cdot \frac{m(x)}{m(y)}
  \leq | g(y^{-1} x) | \cdot w(y^{-1} x)
  \leq \| g \|_{L_{\infty,w}}
\]
for almost every $x \in G$ and almost every $y \in G$, which establishes
\eqref{eq:SchurAssumption1} and \eqref{eq:SchurAssumption2}.

It remains to prove \eqref{eq:young2}. If we assume $|g(x)| = |g(x^{-1})|$ and
$w(x) = w(x^{-1})$, the formula $d \varrho(x) = \Delta(x^{-1}) dx$ from
above yields for $r < \infty$ that
\begin{align*}
  \| g \|_{L_{r, w \Delta^{-1/r}}}^r
  &= \int_G | g(y) |^r \cdot [w (y)]^r \cdot \Delta(y^{-1}) \, dy
  = \int_G | g(z^{-1}) |^r \cdot [w (z^{-1})]^r \, dz \\
  &= \int_G | g(z) |^r \cdot [w(z)]^r \, dz = \| g \|_{L_{r,w}}^r \, .
\end{align*}
This identity trivially holds if $G$ is unimodular, so that $\Delta \equiv 1$.
For $r = \infty$, we always have
$\lVert{g}\rVert_{L_{r,w\Delta^{-1/r}}} = \lVert{g}\rVert_{L_{r,w}}$.
In all of these cases \eqref{eq:young2} is a direct consequence of
\eqref{eq:young1}.
\end{proof}


\begin{lemma}\label{BB}
Let $A$ be a bounded and surjective linear operator that maps a Banach space
$W$ onto a Banach space $V$. Suppose that the kernel of $A$ admits a complement
$L$ in $W$. Set
\begin{align*}
  \varepsilon
  := \inf\Big\{
            \sup\left\{
                    |\langle{Ax},{y}\rangle_{V\times V^\ast}| \,
                    ~\middle|~
                    y \in V^\ast, \lVert{y}\rVert_{V^\ast} = 1
                \right\}
            \, \Big| \,
            x \in L, \lVert{x}\rVert_W=1
         \Big\}.
\end{align*}
%

Then the map $S := (A|_L)^{-1} : V \to L \subset W$ is a linear right inverse
of $A$ with
\begin{align*}
  \lVert{S}\rVert = \varepsilon^{-1}.
\end{align*}
\end{lemma}

\begin{proof}
It is straightforward that $A|_L : L \to V$ is a bijection.
Therefore $S$ is indeed a linear right inverse of $A$, and we have
\begin{align*}
   &\hspace{-0.5cm}\inf\left\{
         \sup\left\{
                |\langle{Ax},{y}\rangle_{V\times V^\ast}| \,
                ~\middle|~
                y \in V^\ast, \lVert{y}\rVert_{V^\ast} = 1
             \right\}
             ~\middle|~
             x\in L, \lVert{x}\rVert_W = 1
       \right\} \\
  &= \inf \left\{ \| Ax \|_V ~\middle|~ x \in L , \| x \|_W = 1\right\} \\
  &= \inf \left\{
            \frac{\| A x \|_V}{\| x \|_W}
            ~\middle|~
            x \in L \setminus \{0\}
          \right\} 
  = \inf \left\{
            \frac{\| A S v \|_V}{\| Sv \|_W}
            ~\middle|~
            v \in V \setminus \{0\}
          \right\} \\
  &= \left(
        \sup \left\{
               \frac{\|Sv\|_W}{\|ASv\|_V}
               ~\middle|~
               v \in V \setminus \{0\}
             \right\}
     \right)^{-1} 
  = \left(
        \sup \left\{
                \frac{\|Sv\|_W}{\|v\|_V}
                ~\middle|~
                v \in V \setminus \{0\}
             \right\}
     \right)^{-1} \\
   &= \|S\|^{-1} ,
\end{align*}
which proves the claim.
\end{proof}

Lemma \ref{lem:osc} and Lemma \ref{lem:3.13} as well as Proposition \ref{prop:BadPaleyWienerSet} were left unproven.
The proofs are presented here.

\medskip

\emph{Proof of Lemma \ref{lem:osc}.}
\label{proof:OscillationLemmaProof}
We start with an auxiliary observation:
We claim that $\| g \|_{L_\infty(Q x)} = \sup_{y \in Q x} | g(y) |$
if $g : G \to \mathbb{C}$ is continuous and if $Q \subset G$ is a compact unit
neighborhood with $Q = \overline{\operatorname{int} Q}$.

Indeed, the inequality ``$\leq$'' is trivial. Conversely, if we set
$\alpha := \| g \|_{L_\infty(Q x)}$, then the set
$M := \left\{y \in G ~\middle|~ | g(y) | > \alpha \right\}$ is open, and $M \cap Q x$ is a
null-set. Hence, $M \cap (\operatorname{int} Q) x = \emptyset$, since this is an
open null-set. In other words, $| g(y) | \leq \alpha$ for all
$y \in (\operatorname{int} Q) x$. By continuity of $g$ and since
$Q \subset \overline{\operatorname{int} Q}$, we see $|g(y)| \leq \alpha$
for all $y \in Q x$.

In particular, this implies
$\widecheck{M}^\rho_Q g (x) = \sup_{q \in Q} | g(qx) |$,
and thus (because of $e \in Q$) $\widecheck{M}^\rho_Q g \geq |g|$.

\medskip{}

To prove i) we note that $\widecheck{M}^\rho_{Q_0} f \in L_{p,w}(G)$ which implies
$f \in L_{p,w}(G)$, since we just saw that $\widecheck{M}^\rho_{Q_0} f \geq | f |$.
We intend to show $\lVert{\operatorname{osc}_{Q_0}f}\rVert_{L_{p,w}} < \infty$.
But we have
\[
       \operatorname{osc}_{Q_0} f (x)
  =    \sup_{q \in Q_0} | f(qx) - f(x) |
  \leq \sup_{q \in Q_0} | f(qx) | + | f(x) |
  \leq |f(x)| + \widecheck{M}^\rho_{Q_0} f (x) \, .
\]
Therefore,
\begin{align}\label{eq:osc3}
  \lVert{\mbox{osc}_{Q_0} f}\rVert_{L_{p,w}}
  \leq \lVert{\widecheck{M}^\rho_{Q_0} f}\rVert_{L_{p,w}} + \lVert{f}\rVert_{L_{p,w}} < \infty \, .
\end{align}

\medskip{}

It remains to prove ii).
For this we first note that
$\mbox{osc}_Q f \leq \mbox{osc}_{Q_0} f$ if $Q \subset Q_0$.
Furthermore, by part i) we have $\operatorname{osc}_{Q_0} f \in L_{p,w}(G)$.
Hence, since $G$ is $\sigma$-compact, for any $\varepsilon > 0$,
there exists a compact set $K \subset G$ of positive measure such that
\begin{align} \label{eq:osc1}
  \int_{G\setminus K} |\mbox{osc}_{Q} f(x) w(x)|^p ~dx
  \leq \int_{G\setminus K} |\mbox{osc}_{Q_0} f(x) w(x)|^p ~dx
  < \frac{\varepsilon}{2}
\end{align}
for all unit neighborhoods $Q \subset Q_0$.

Next, we observe that since $f$ is continuous, it is uniformly continuous
on $K$ in the following sense: For every $\delta > 0$ there is a unit
neighborhood $U_\delta \subset G$ with $| f(x) - f(ux) | < \delta$ for all
$x \in K$ and $u \in U_\delta$.

The uniform continuity described above simply means
$\operatorname{osc}_{U_\delta} f (x) \leq \delta$ for all $x \in K$.
Choosing
$\delta := \varepsilon^{1/p} / ([2 \cdot |K|]^{1/p} \sup_{y \in K} w(y))$,
we see for every unit neighborhood $Q \subset Q_0 \cap U_\delta$ that
\begin{align} \label{eq:osc2}
  \int_K |\mbox{osc}_{Q} f(x) w(x)|^p ~dx
  &\leq \int_K
          \frac{\varepsilon}{2 |K|}
          \cdot \frac{w(x)^p}{\sup_{y \in K} w(y)^p}
        ~dx
   \leq \int_K \frac{\varepsilon}{2 |K|} ~dx
   =    \frac{\varepsilon}{2}.
\end{align}
Equations \eqref{eq:osc1} and \eqref{eq:osc2} yield
$\lVert{\operatorname{osc}_{Q}f}\rVert_{L_{p,w}}^p < \varepsilon$, which concludes the proof.

\bigskip{}

\textit{Proof of Lemma \ref{lem:3.13}.}
Let $1\leq p<\infty$ and $(d_x)_{x\in Y_n}\in\ell_{p,m}(Y_n)$,
then we first note that for all $x\in G$ it holds
\begin{align*}
\int_{xQ_n}m(y)^p~dy &= \int_{Q_n}m(xy)^p~dy \leq m(x)^p\int_{Q_n}w(y)^p~dy 
\\&\leq  \sup_{q\in Q_n}w(q)^p\cdot|Q_n|\cdot m(x)^p.
\end{align*}
With this at hand and since $Y_n$ is relatively $Q_n$-separated, as
stated in~\eqref{eq:75}, we derive
\begin{align*}
        \bigg\|\sum_{x\in Y_n}|d_x|\chi_{xQ_n}\bigg\|_{L_{p,m}}
  &\leq \sum_{i=1}^\mathcal{I}\bigg\|\sum_{x\in Z_{n,i}}|d_x|\chi_{xQ_n}\bigg\|_{L_{p,m}} 
  =    \sum_{i=1}^\mathcal{I}
           \left(
             \sum_{x\in Z_{n,i}}
                |d_x|^p
                \int_{xQ_n} m(y)^p~dy
           \right)^{\frac{1}{p}}
  \\
  &\leq \sum_{i=1}^\mathcal{I}
           \left(
             \sum_{x\in Z_{n,i}}|d_x|^p m(x)^p
           \right)^{\frac{1}{p}}
        \cdot \sup_{q\in Q_n} w(q) \cdot |Q_n|^{\frac{1}{p}} \\
  &\leq \mathcal{I}^{1-\frac{1}{p}}
        \cdot \sup_{q \in Q_n} w(q)
        \cdot |Q_n|^{\frac{1}{p}} \cdot \lVert{(d_x)_{x \in Y_n}}\rVert_{\ell_{p,m}}.
\end{align*}

It remains to prove the case $p=\infty$. Similarly as above, we see that
\begin{align*}
  \bigg\| \sum_{x \in Y_n} |d_x| \chi_{xQ_n} \bigg\|_{L_{\infty,m}}
  &\leq \sum_{i=1}^\mathcal{I}
           \bigg\|
              \sum_{x \in Z_{n,i}} |d_x| \chi_{xQ_n}
           \bigg\|_{L_{\infty,m}} 
  = \sum_{i=1}^\mathcal{I}
        \sup_{x\in Z_{n,i}}
           \left(|d_x| \cdot \sup_{y \in xQ_n} m(y)\right) \\
  &\leq \mathcal{I}
        \cdot \left(\sup_{x \in Y_n} |d_x| \cdot m(x)\right)
        \cdot\sup_{y \in Q_n} w(y) \\
  &= \mathcal{I} \cdot \sup_{q \in Q_n} w(q)
     \cdot \lVert{(d_x)_{x \in Y_n}}\rVert_{\ell_{\infty,m}}.
  \qquad \qquad \qquad
\end{align*}

\textit{Proof of Proposition \ref{prop:BadPaleyWienerSet}.}
  We will construct $C \subset [0,1]$ as a certain ``fat Cantor set''.
  In particular, we will show below that $C$ has
  positive measure and fulfills the following two additional properties:
  \begin{equation}
    |C \cap B| < |B|
    \qquad \text{ for all open intervals } \emptyset \neq B \subset \mathbb{R} \, ,
    \label{eq:FatCantorDoesNotContainInterval}
  \end{equation}
  and
  \begin{equation}
    C^c
    = \bigcup_{n = 0}^\infty
        \bigcup_{j=0}^{2^n - 1}
          B_j^n
    \quad \text{with} \quad
    B_j^n
    := \frac{a_j^{(n)} + b_j^{(n)}}{2}
       + \left( - \frac{\mu_{n+1}}{2}, \frac{\mu_{n+1}}{2} \right) \, ,
  \label{eq:FatCantorComplement}
  \end{equation}
  where the complement $C^c$ is taken relative to $[0,1]$,
  and where $a_j^{(n)}, b_j^{(n)} \in \mathbb{R}$ are suitable,
  while $\mu_{n} := \min\{4^{-n}, n^{-n} \}$ for $n \in \mathbb{N}$.

  Before we provide the precise construction of such a set $C$,
  let us see how the properties \eqref{eq:FatCantorDoesNotContainInterval}
  and \eqref{eq:FatCantorComplement} imply the properties of $C$
  that are stated in the proposition.

  \medskip{}

  First, \cite[Theorem 1]{LebedevOlevskiiFourierMultiplierIdempotents}
  shows that if the operator $f \mapsto f \ast \mathcal{F}^{-1}\chi_C$
  is bounded on $L_p(\mathbb{R})$ for some
  $p \in (1,\infty) \setminus \{2\}$, that is, if $\chi_C$ is an
  $L_p(\mathbb{R})$-Fourier multiplier, then $C$ would be equivalent to an open set.
  In other words, there would be an open set $U \subset \mathbb{R}$ with
  $\chi_C = \chi_U$ Lebesgue almost everywhere.
  But since $C$ has positive measure, this is only possible if $U$ is a
  \emph{nonempty} open set.
  Therefore, $U$ contains a nonempty open interval $B \subset U$.
  Since $\chi_C = \chi_U$ almost everywhere, this implies
  $|B \cap C| = |B \cap U| = |B|$, in contradiction to
  \eqref{eq:FatCantorDoesNotContainInterval}. In summary, we have thus shown
  that the convolution operator $f \mapsto f \ast \mathcal{F}^{-1} \chi_C$
  is \emph{not} bounded on any $L_p(\mathbb{R})$ space for
  $p \in (1,\infty) \setminus \{2\}$.
  But this even implies that
  $L_p (\mathbb{R}) \to L_p (\mathbb{R}), f \mapsto f \ast \mathcal{F}^{-1} \chi_C$
  is not well-defined, by Proposition~\ref{prop:RCK_bounded}.

  \medskip{}

  Second, we will see that \eqref{eq:FatCantorComplement}
  ensures $\mathcal{F}^{-1} \chi_{C^c} \in \bigcap_{1 < p \leq \infty} L_p(\mathbb{R})$,
  which then implies
  $\mathcal{F}^{-1} \chi_C = \mathcal{F}^{-1}\chi_{(0,1)} - \mathcal{F}^{-1}\chi_{C^c}
  \in \bigcap_{1 < p \leq \infty} L_p(\mathbb{R})$.
  Here, we used that
  $F := \mathcal{F}^{-1} \chi_{(0,1)} \in \bigcap_{1 < p \leq \infty} L_p(\mathbb{R})$,
  since a direct computation shows $F(x) = \frac{e^{2\pi i x} - 1}{2\pi i x}$
  for $x \neq 0$, which implies $|F(x)| \lesssim (1 + |x|)^{-1}$.
  It remains to show
  $\mathcal{F}^{-1} \chi_{C^c} \in \bigcap_{1 < p \leq \infty} L_p(\mathbb{R})$.
  To this end, we set
  $\xi_j^{(n)} := \frac{a_j^{(n)} + b_j^{(n)}}{2} - \frac{\mu_{n+1}}{2}$,
  recall the definition of the intervals
  $B_j^n = \xi_j^{(n)} + \mu_{n+1} \cdot (0,1)$ from
  \eqref{eq:FatCantorComplement},
  and use standard properties of the Fourier transform to compute
  \[
    \mathcal{F}^{-1} \chi_{B_j^n}
    = \mu_{n+1}
      \cdot M_{\xi_j^{(n)}} \big[
                              (\mathcal{F}^{-1} \chi_{(0,1)})(\mu_{n+1} \cdot)
                            \big]
  = \mu_{n+1} \cdot M_{\xi_j^{(n)}} \big(F (\mu_{n+1} \cdot)\big) \, ,
  \]
  where $(M_\xi f)(x) = e^{2\pi i x \xi} f(x)$ denotes the \emph{modulation}
  with frequency $\xi$ of a function $f$.
  Next, \eqref{eq:FatCantorComplement} shows
  \[
    \mathcal{F}^{-1} \chi_{C^c}
    = \sum_{n=0}^\infty
        \sum_{j=0}^{2^n - 1}
          \mathcal{F}^{-1} \chi_{B_j^n} \, .
  \]
  Combining this with the triangle inequality for $L_p$ and with the elementary
  identities $\|M_\xi f\|_{L_p} = \|f\|_{L_p}$ and
  $\|f ( a \cdot )\|_{L_p(\mathbb{R})} = a^{-1/p} \|f\|_{L_p(\mathbb{R})}$ for $a > 0$ and
  $f \in L_p (\mathbb{R})$, we see because of $\mu_n \leq n^{-n}$ and $1 - p^{-1} > 0$
  for each fixed $p \in (1,\infty]$ that
  \begin{equation}
    \begin{split}
      \|\mathcal{F}^{-1} \chi_{C^c}\|_{L_p}
      & \leq \sum_{n=0}^\infty
               \sum_{j=0}^{2^n - 1}
                 \mu_{n+1} \cdot
                 \|M_{\xi_j^{(n)}} \big( F(\mu_{n+1} \cdot) \big)\|_{L_p} \\
      & \leq \|F\|_{L_p} \cdot
             \sum_{n=0}^\infty
               \sum_{j=0}^{2^n - 1}
                 \mu_{n+1}^{1 - p^{-1}} 
        \leq \|F\|_{L_p} \cdot
             \sum_{\ell=1}^\infty
               2^{\ell-1} \cdot \ell^{-\ell (1 - p^{-1})} \, .
    \end{split}
    \label{eq:FatCantorSetLpMainEstimate}
  \end{equation}
  But for $\ell \geq \ell_0 = \ell_0 (p)$, we have
  $(1-p^{-1}) \cdot \log_2 (\ell) \geq 2$, and thus
  \[
    2^{\ell-1} \cdot \ell^{-\ell (1 - p^{-1})}
    = \frac{1}{2} \cdot 2^\ell \cdot 2^{- \ell (1 - p^{-1}) \cdot \log_2 (\ell)}
    \leq 2^{\ell \big(1 - (1-p^{-1}) \cdot \log_2 (\ell)\big)}
    \leq 2^{-\ell} \, ,
  \]
  so that the series on the right-hand side of
  \eqref{eq:FatCantorSetLpMainEstimate} converges.
  Hence, $\mathcal{F}^{-1} \chi_{C^c} \in L_p(\mathbb{R})$ for every $p \in (1,\infty]$.

  \medskip{}

  Finally, we note because of $\mu_n \leq 4^{-n}$ that property
  \eqref{eq:FatCantorComplement} also implies
  \[
    |C^c|
    = \sum_{n=0}^\infty \sum_{j=0}^{2^n - 1} |B_j^n|
    = \sum_{n=0}^\infty 2^n \mu_{n+1}
    \leq \sum_{n=0}^\infty 2^n \cdot 4^{-(n+1)}
    \leq \frac{1}{4} \cdot \sum_{n=0}^\infty 2^{-n}
    = \frac{1}{2} < 1 \, ,
  \]
  so that $C \subset [0,1]$ necessarily has positive measure if it satisfies
  properties \eqref{eq:FatCantorDoesNotContainInterval} and
  \eqref{eq:FatCantorComplement}.
  It remains to show that one can indeed construct a compact set
  $C \subset [0,1]$ that satisfies properties
  \eqref{eq:FatCantorDoesNotContainInterval} and \eqref{eq:FatCantorComplement}.

  \medskip{}

  To this end, as for the construction of the classical Cantor set,
  we will set $C := \bigcap_{n=0}^\infty C^n$ where the sets
  $C^n := \bigcup_{j=0}^{2^n - 1} C_j^n$ will be defined inductively.

  For the start of the induction set
  $C_0^0 := [a_1^{(0)}, b_1^{(0)}] := [0,1]\vphantom{\sum_j}$.

  For the induction step, assume for some $n \in \mathbb{N}_0 \vphantom{a_1^{(0)}}$
  that we have constructed
  closed intervals $C_\ell^n = [a_\ell^{(n)}, b_\ell^{(n)}] \subset [0,1]$,
  for $\ell = 0, \dots, 2^n - 1$, with
  \begin{equation}
    4^{-n} \leq b_\ell^{(n)} - a_\ell^{(n)} \leq 2^{-n}
    \quad \text{for all} \quad
    0 \leq \ell < 2^n
    \label{eq:FatCantorIntervalLengths}
  \end{equation}
  and
  \begin{equation}
    b_\ell^{(n)} < a_{\ell+1}^{(n)}
    \qquad \text{for} \qquad \, 0 \leq \ell < 2^n - 1 \, .
    \label{eq:FatCantorSetProperties}
  \end{equation}
  Now, for $0 \leq j < 2^{n+1}$ we can write $j = 2\ell + k$ with uniquely
  determined $k \in \{0,1\}$ and $0 \leq \ell < 2^n$.
  We then recall from after \eqref{eq:FatCantorComplement}
  that $\mu_{n+1} = \min \{4^{-(n+1)}, (n+1)^{-(n+1)} \}$, and define
  \begin{equation}
    \begin{split}
      C_j^{n+1}
       := [a_j^{(n+1)}, b_j^{(n+1)}]
       := \begin{cases}
             \Big[a_\ell^{(n)},
              \frac{a_\ell^{(n)} + b_\ell^{(n)}}{2} - \frac{\mu_{n+1}}{2}\Big]
             \subset [a_\ell^{(n)}, b_\ell^{(n)}] = C_\ell^n
             & \text{if } k = 0 \, , \\[0.2cm]
             \Big[\frac{a_\ell^{(n)} + b_\ell^{(n)}}{2} + \frac{\mu_{n+1}}{2},
              b_\ell^{(n)}\Big]
             \subset [a_\ell^{(n)}, b_\ell^{(n)}] = C_\ell^n
             & \text{if } k = 1 \, .
           \end{cases}
    \end{split}
    \label{eq:FatCantorDefinition}
  \end{equation}
  With this choice, we see from \eqref{eq:FatCantorIntervalLengths} and
  because of $\mu_{n+1} \leq 4^{-(n+1)}$ that
  \[
    b_j^{(n+1)} - a_j^{(n+1)}
    = \frac{b_\ell^{(n)} - a_\ell^{(n)}}{2} - \frac{\mu_{n+1}}{2}
    \geq \frac{1}{2} \cdot \big( 4^{-n} - 4^{-(n+1)} \big)
    = \frac{3}{8} \cdot 4^{-n} \geq 4^{-(n+1)}
  \]
  and
  \[
    b_j^{(n+1)} - a_j^{(n+1)}
    = \frac{b_\ell^{(n)} - a_\ell^{(n)}}{2} - \frac{\mu_{n+1}}{2}
    \leq \frac{1}{2} (b_\ell^{(n)} - a_\ell^{(n)})
    \leq 2^{-(n+1)} \, ,
  \]
  thereby proving \eqref{eq:FatCantorIntervalLengths} for $n+1$ instead of $n$.

  For the proof of \eqref{eq:FatCantorSetProperties} for
  $0 \leq j < 2^{n+1} - 1$ with $j = 2\ell + k$ and
  $k \in \{0,1\}$, we distinguish two cases:

  \textbf{Case 1:} $k = 0$. In this case, $j+1 = 2\ell + 1$, and hence
  \[
    b_j^{(n+1)}
    = \frac{a_\ell^{(n)} + b_\ell^{(n)}}{2} - \frac{\mu_{n+1}}{2}
    < \frac{a_\ell^{(n)} + b_\ell^{(n)}}{2} + \frac{\mu_{n+1}}{2}
    = a_{j+1}^{(n+1)} \, .
  \]

  \textbf{Case 2:} $k = 1$. In this case, $2(\ell + 1) + 0 = j+1 < 2^{n+1}$,
  so that $1 \leq \ell + 1 < 2^n$. Therefore, \eqref{eq:FatCantorSetProperties}
  shows
  $b_j^{(n+1)}
    = b_\ell^{(n)}
    < a_{\ell+1}^{(n)}
    = a_{j+1}^{(n+1)}$.

  \smallskip{}

  We have thus verified \eqref{eq:FatCantorSetProperties} for $n+1$ instead
  of $n$.

  As indicated above, we define $C^n := \bigcup_{j=0}^{2^n - 1} C_j^n$ and
  observe as a consequence of \eqref{eq:FatCantorDefinition}
  that each $C^n$ is closed with $C^{n+1} \subset C^n$ for all $n \in \mathbb{N}_0$.
  Hence, $C := \bigcap_{n=0}^\infty C^n \subset C^0 = [0,1]$ is compact.

  \medskip{}

  Having defined the set $C$, our first goal is to prove property
  \eqref{eq:FatCantorDoesNotContainInterval}.
  Let $B \subset \mathbb{R}$ be a nonempty open interval.
  If $C \cap B$ is a finite set, the inequality in
  \eqref{eq:FatCantorDoesNotContainInterval} is trivially satisfied.
  Hence, we can assume that
  $C \cap B$ is infinite, so that there are $x,y \in C \cap B$ with $x < y$.
  Choose $n \in \mathbb{N}_0$ with $2^{-n} < y-x$ and note because of
  $x,y \in C \subset C^n = \bigcup_{j=0}^{2^n - 1} C_j^n$ that there are
  $j_x, j_y \in \{0,\dots, 2^n - 1\}$ with $x \in C_{j_x}^n$ and
  $y \in C_{j_y}^n$.
  In case of $j_y \leq j_x$, we would get because of
  $a_\ell^{(n)} \leq b_\ell^{(n)} \leq a_{\ell+1}^{(n)}$ for
  $0 \leq \ell < 2^n - 1$ and because of
  $b_\ell^{(n)} - a_\ell^{(n)} \leq 2^{-n}$ for $0 \leq \ell < 2^n$
  (see \eqref{eq:FatCantorIntervalLengths}, \eqref{eq:FatCantorSetProperties})
  that
  \[
    2^{-n}
    < y-x
    \leq b_{j_y}^{(n)} - a_{j_x}^{(n)}
    \leq b_{j_y}^{(n)} - a_{j_y}^{(n)}
    \leq 2^{-n},
  \]
  a contradiction. Hence, $j_y > j_x$, so that
  \eqref{eq:FatCantorIntervalLengths} and \eqref{eq:FatCantorSetProperties} show
  \[
    B \ni x
      \leq b_{j_x}^{(n)}
      \leq b_{j_y - 1}^{(n)}
      < a_{j_y}^{(n)}
      \leq y \in B \, ,
  \]
  and thus $(b_{j_y - 1}^{(n)}, a_{j_y}^{(n)}) \subset B \setminus C^n
  \subset B \setminus C$.
  But since this interval has positive measure, we see
  $|B| = |B \setminus C| + |B \cap C| > |B \cap C|$,
  thereby proving \eqref{eq:FatCantorDoesNotContainInterval}.

  \medskip{}

  Finally, we prove the formula \eqref{eq:FatCantorComplement} for the
  complement $C^c$ of $C$, with the complement taken relative to $[0,1]$.
  To see this, note $C^c = \bigcup_{n=0}^\infty (C^n)^c$. By disjointization,
  and since $(C^0)^c = \emptyset$ and $(C^n)^c \subset (C^{n+1})^c$,
  this yields
  \[
    C^c
    = \bigcup_{n=1}^\infty (C^n)^c \setminus (C^{n-1})^c
    = \bigcup_{n=1}^\infty C^{n-1} \setminus C^{n}
    = \bigcup_{n=0}^\infty C^n \setminus C^{n+1} \, .
  \]
  Next, recall $C^n = \bigcup_{j=0}^{2^n - 1} C_j^n$ and also
  recall from \eqref{eq:FatCantorDefinition} that
  $C_{2\ell + k}^{n+1} \subset C_\ell^n$ for $0 \leq \ell < 2^n$ and $k \in \{0,1\}$.
  Therefore, by \eqref{eq:FatCantorDefinition} and the definition of  $B_j^n$ in \eqref{eq:FatCantorComplement} it holds
  \[
    C_j^n \setminus C^{n+1}
    = \bigcap_{\ell=0}^{2^n - 1}
         \bigcap_{k=0}^1
           C_j^n \setminus C_{2\ell + k}^{n+1}
    = \bigcap_{k=0}^1
        C_j^n \setminus C_{2j + k}^{n+1}
    = B_j^n \, .
  \]
  Putting everything together, we see that \eqref{eq:FatCantorComplement} holds.
\section*{Acknowledgements}

F. Voigtlaender would like to thank Werner Ricker for helpful discussions
regarding idempotent Fourier multipliers.

E. De Vito and F. De Mari are members of the Gruppo Nazionale per l'Analisi Matematica, la Probabilit\`a e le loro Applicazioni (GNAMPA) of the Instituto Nazionale di Alta Matematica (INdAM).

\bibliographystyle{amsplain}

\providecommand{\bysame}{\leavevmode\hbox to3em{\hrulefill}\thinspace}
\providecommand{\MR}{\relax\ifhmode\unskip\space\fi MR }
\providecommand{\MRhref}[2]{%
  \href{http://www.ams.org/mathscinet-getitem?mr=#1}{#2}
}
\providecommand{\href}[2]{#2}

\end{document}